\documentclass[11pt,twoside,reqno]{amsart}

\usepackage{microtype}
\usepackage{cite}
\usepackage[OT1]{fontenc}
\usepackage{lmodern}
\usepackage{type1cm}
\usepackage{amssymb}
\usepackage{enumerate}
\usepackage{comment}
\usepackage{xcolor}
\usepackage{mathrsfs}
\usepackage{color}

\usepackage{geometry}
\usepackage{hyperref}
\hypersetup{
  colorlinks=true,
  linkcolor=black,
  anchorcolor=black,
  citecolor=black,
  filecolor=black,      
  menucolor=red,
  runcolor=black,
  urlcolor=black,
}

\numberwithin{equation}{section}

\theoremstyle{plain}
\newtheorem{theorem}{Theorem}[section]

\newtheorem{proposition}[theorem]{Proposition}

\newtheorem{lemma}[theorem]{Lemma}

\newtheorem{tmptheoremno}{Theorem}
\newtheorem{tmppropositionno}{Proposition}
\newtheorem{tmplemmano}{Lemma}

\theoremstyle{remark}
\newtheorem{remark}[theorem]{Remark}

\newtheorem{example}[theorem]{Example}

\theoremstyle{definition}

\newenvironment{propositionno}[1]
  { \renewcommand\thetmppropositionno{#1}\begin{tmppropositionno} }
  { \end{tmppropositionno} }

\newcommand{\GL}{\mathrm{GL}}
\newcommand{\Gr}{\mathrm{Gr}}

\newcommand{\HH}{\mathcal{H}}
\newcommand{\LL}{\mathcal{L}}
\newcommand{\PP}{\mathcal{P}}
\newcommand{\QQ}{\mathcal{Q}}
\newcommand{\RR}{\mathcal{R}}
\newcommand{\MM}{\mathcal{M}}

\newcommand{\R}{\mathbb{R}}
\newcommand{\I}{\mathcal{I}}
\newcommand{\J}{\mathcal{J}}
\newcommand{\K}{\mathcal{K}}

\newcommand{\Z}{\mathbb{Z}}
\newcommand{\N}{\mathbb{N}}

\newcommand{\iii}{\mathtt{i}}
\newcommand{\jjj}{\mathtt{j}}
\newcommand{\kkk}{\mathtt{k}}
\renewcommand{\lll}{\mathtt{l}}
\newcommand{\eps}{\varepsilon}
\newcommand{\fii}{\varphi}

\newcommand{\A}{\mathsf{A}}

\newcommand{\End}{\mathrm{End}}

\newcommand{\dd}{\,\mathrm{d}}

\renewcommand{\ge}{\geqslant}
\renewcommand{\le}{\leqslant}
\renewcommand{\geq}{\geqslant}
\renewcommand{\leq}{\leqslant}

\newcommand{\threebar}[1]{{\left\vert\kern-0.25ex\left\vert\kern-0.25ex\left\vert #1 
    \right\vert\kern-0.25ex\right\vert\kern-0.25ex\right\vert}}

\DeclareMathOperator{\Span}{span}

\DeclareMathOperator{\dimh}{dim_H}

\DeclareMathOperator{\dimaff}{dim_{aff}}
\DeclareMathOperator{\udimaff}{\overline{dim}_{aff}}
\DeclareMathOperator{\ldimaff}{\underline{dim}_{aff}}

\DeclareMathOperator{\dist}{dist}
\DeclareMathOperator{\diam}{diam}

\renewcommand{\atop}[2]{\genfrac{}{}{0pt}{}{#1}{#2}}

\makeatletter
\@namedef{subjclassname@2020}{%
  \textup{2020} Mathematics Subject Classification}
\makeatother

\begin{document}

\title{Thermodynamic formalism of countably generated self-affine sets}

\author{Antti K\"aenm\"aki}
\address[Antti K\"aenm\"aki]
        {University of Eastern Finland \\ 
         Department of Physics and Mathematics \\
         P.O.\ Box 111 \\ 
         FI-80101 Joensuu \\ 
         Finland}
\email{antti.kaenmaki@uef.fi}

\author{Ian D. Morris}
\address[Ian D. Morris]
        {School of Mathematical Sciences \\
         Queen Mary University of London \\
         Mile End Road \\
         London E1 4NS \\
         United Kingdom}
\email{i.morris@qmul.ac.uk}

\thanks{The research of Ian D.~Morris was partially supported by the Leverhulme Trust (Research Project Grant RPG-2016-194).}
\date{\today}

\begin{abstract}
  In this article, we further develop the thermodynamic formalism of affine iterated function systems with countably many transformations by showing the existence and extending earlier characterisations of the equilibrium states of finite affine iterated function systems to the countably infinite case. As an application, under mild conditions, we prove that the affinity dimension of a countable affine iterated function system is equal to the supremum of the affinity dimensions of its finite subsystems. We deduce corollaries concerning the Hausdorff dimension of countably generated self-affine sets in dimensions $1$, $2$, and $3$ satisfying mild deterministic assumptions and in arbitrary dimension with generic translations.
\end{abstract}

\maketitle

\tableofcontents

\section{Introduction and statement of results} \label{sec:intro-results}

\subsection{Background}

An \emph{iterated function system} acting on $\mathbb{R}^d$ is defined to be a collection $(T_i)_{i \in \mathcal{I}}$ of transformations $T_i \colon \R^d \to \R^d$ which are contracting with respect to some fixed norm $\threebar{\,\cdot\,}$ on $\R^d$, uniformly with respect to $i \in \I$, such that the fixed points of $T_i$ form a bounded set. In this article the set $\mathcal{I}$, which we call the \emph{index set} for the iterated function system $(T_i)_{i \in \mathcal{I}}$, will always be a nonempty set which is either finite or countably infinite. Since the transformations $T_i$ contract uniformly with respect to $i \in \I$, the mapping $A \mapsto \overline{\bigcup_{i \in \I} T_i(A)}$ defined on nonempty compact subsets of $\mathbb{R}^d$ is strictly contractive in Hausdorff distance. Therefore, by Banach's contraction mapping theorem, there exists a unique nonempty compact set $K \subset \mathbb{R}^d$ which satisfies
\begin{equation} \label{eq:X-invariant}
  K=\overline{\bigcup_{i \in \mathcal{I}} T_i(K)}.
\end{equation}
The set $K$ is called the \emph{attractor} of the iterated function system $(T_i)_{i \in \I}$. 

If the index set $\mathcal{I}$ is finite, then it is classical (and easily demonstrated) that the attractor $K$ is characterised by the following property: a point $x \in \R^d$ belongs to $K$ if and only if it is a \emph{limit point} of $(T_i)_{i \in \I}$, that is, there exists $(i_n)_{n=1}^\infty \in \I^\N$ such that for every $v \in \R^d$ we have
\begin{equation}\label{eq:limit-definition}
  \lim_{n \to \infty} T_{i_1} \circ \cdots \circ T_{i_n} (v) = x.
\end{equation}
In the countably infinite case, using the facts that the transformations $T_i$ contract uniformly with respect to $i \in \I$ and the fixed points of $T_i$ form a bounded set, it is not difficult to show that for every $(i_n)_{n=1}^\infty \in \I^\N$ there exists an associated limit point $x \in \R^d$ satisfying \eqref{eq:limit-definition} for all $v \in \R^d$. The union
\begin{equation} \label{eq:limit-set-def}
  X=\bigcup_{(i_n)_{n=1}^\infty \in \I^\N} \lim_{n \to \infty} T_{i_1} \circ \cdots \circ T_{i_n} (v)
\end{equation}
of all limit points is called the \emph{limit set} of the iterated function system $(T_i)_{i \in \I}$. It is easy to see that the attractor $K$ is the closure of the limit set $X$, which satisfies $X=\bigcup_{i \in \mathcal{I}} T_i(X)$ but which in general need not be compact.

As an example, let us consider an iterated function system $(x \mapsto (i+x)^{-1})_{i \in \N}$ acting on $(0,1)$\footnote{Strictly speaking this example does not define an iterated function system on $(0,1)$ since the map $f_i(x)=1/(i+x)$ is not contracting for $i=1$, but this point of detail may be circumvented by considering the larger system of maps $(f_i \circ f_j)_{i,j=1}^\infty$ which is uniformly contracting.}. This countably infinite system arises from continued fraction expansions, and its limit set $X$ as defined in \eqref{eq:limit-set-def} is precisely the set of all irrational numbers in the unit interval. The attractor $K$ in this case is the unit interval, so the limit set better reflects the dynamical properties of the system. We are therefore interested in the limit set of an iterated function system, i.e.\ the set of all points $x \in \R^d$ which arise as limits of the form \eqref{eq:limit-definition} for a given $(T_i)_{i \in \I}$. It is worthwhile to emphasize that in this example the limit set is not topologically closed.

Throughout this article we will be concerned with the situation in which every transformation $T_i$ is invertible and affine, having the form $T_i(x)=A_ix+v_i$ for some invertible linear map $A_i \in \GL_d(\R)$ and vector $v_i \in \R^d$, and we will describe such iterated function systems simply as \emph{affine iterated function systems}. The limit set $X$ of an affine iterated function system is conventionally called a \emph{self-affine set} as it consists of affine images of itself. In this article we will prefer to say that a set is a \emph{finitely generated self-affine set} if it is the attractor of an affine iterated function system with a finite index set and an \emph{infinitely generated self-affine set} if it is the limit set of an affine iterated function system with a countably infinite index set. The dimension theory of finitely generated self-affine sets, and study of the natural measures on finitely generated self-affine sets, has been very substantially developed in the last two decades in works such as \cite{BaranyKaenmakiRossi2021, HochmanRapaport2022, BaranyJordanKaenmakiRams2021, BaranyKaenmakiMorris2020, BaranyKaenmakiKoivusalo2018, KaenmakiMorris2018, BaranyKaenmaki2017, Feng2023, FengKaenmaki2011, BaranyHochmanRapaport2019, Baranski2007, BochiMorris2018, DasSimmons2017, Feng2023, FengShmerkin2014, Fraser2012, Morris2016, Rossi2014, Morris2018, KaenmakiVilppolainen2010, KaenmakiReeve2014, JordanPollicottSimon2007, Kaenmaki2004, MorrisSert2019, MorrisSert2023preprint, Rapaport2022preprint, Rapaport2023preprint, BaranyKaenmakiYu2021, FengFeng2023preprint}. In this article, we will be concerned with the extension of these results to infinitely generated self-affine sets, continuing a project which was initiated in \cite{KaenmakiReeve2014}. We will be particularly interested in extending the thermodynamic formalism of finitely generated self-affine sets to the case of infinitely generated self-affine sets, and in the approximation of infinitely generated self-affine sets by their finitely generated self-affine subsets. In this respect the present work parallels the now-classic article \cite{MauldinUrbanski1996} which extended the theory of conformal iterated function systems from the finitely-generated to the infinitely-generated context in an analogous manner.

\subsection{Singular value pressure}
In this article we let $\|\cdot\|$ denote the Euclidean norm on $\R^d$ and its induced operator norm on $d \times d$ real matrices. We denote the set of all real $d  \times d$ matrices by $\mathrm{M}_d(\R)$. If $\threebar{\,\cdot\,}$ is any norm on $\R^d$ then the same symbol will likewise be used to denote the corresponding operator norm on $\mathrm{M}_d(\R)$. We recall that the \emph{singular values} of $A \in \mathrm{M}_d(\R)$ are defined to be the non-negative square roots of the eigenvalues of the positive-semidefinite matrix ${A^\top}A$ and are denoted $\sigma_1(A),\ldots,\sigma_d(A)$ in non-increasing order. The identities $\sigma_1(A)=\|A\|$ and $\prod_{i=1}^d \sigma_i(A)=|\det A|$ for all $A \in \mathrm{M}_d(\R)$ are standard, as is the identity $\sigma_d(A)=\|A^{-1}\|^{-1}$ in the case where $A$ is invertible. We now recall some further definitions arising in \cite{Falconer1988}. For each $A \in \mathrm{M}_d(\R)$ and $s \geq 0$ we define the \emph{singular value function} by
\begin{equation*}
  \varphi^s(A)=
  \begin{cases}
    \sigma_1(A)\cdots \sigma_{\lfloor s\rfloor}(A) \sigma_{\lceil s\rceil}(A)^{s-\lfloor s\rfloor}, &\text{if } 0 \leq s \leq d,\\
    |\det A|^{\frac{s}{d}}, &\text{if } s > d.
  \end{cases}
\end{equation*}
Note that $\sigma_d(A)^{s}\le\fii^s(A)=\|A\|^s$ for all $0 \le s \le 1$ and $\sigma_d(A)^{s} \le \fii^s(A) \le \|A\|^s$ for all $s > 1$. The inequality $\varphi^s(AB) \leq \varphi^s(A)\varphi^s(B)$ was demonstrated in \cite{Falconer1988} to hold for all $A,B \in \mathrm{M}_d(\R)$. Given a finite or countably infinite tuple $\A=(A_i)_{i \in \I} \in \GL_d(\R)^\I$ of invertible matrices, we define for each $s \geq 0$ the \emph{pressure} of $\A$ at $s$ by setting
\begin{equation}\label{eq:pressure-defn}
  P(\A,s) = \lim_{n \to \infty} \frac{1}{n}\log \sum_{(i_1,\ldots,i_n) \in \I^n} \varphi^s(A_{i_1}\cdots A_{i_n}) \in (-\infty,\infty].
\end{equation}
The sequence $(a_n)_{n=1}^\infty$, where
\[
  a_n= \log \sum_{(i_1,\ldots,i_n) \in \I^n} \varphi^s(A_{i_1}\cdots A_{i_n}),
\]
satisfies the subadditivity property $a_{n+m} \leq a_n+a_m$ for all $n,m \geq 1$ as a consequence of the aforementioned inequality. If every $a_n$ is finite (as is necessarily the case when $\I$ is a finite set) this property suffices to guarantee the existence of the limit \eqref{eq:pressure-defn} as an element of $[-\infty,\infty)$ by the classical subadditivity lemma of Fekete. On the other hand when some of the terms $a_n$ are allowed to equal $\infty$ the existence of the limit is no longer guaranteed by subadditivity alone and additional arguments are needed. (For example, if $b_n=1$ for even $n$ and $b_n=\infty$ for odd $n$ then the sequence $(b_n)_{n=1}^\infty$ is subadditive but the sequence $(b_n/n)_{n=1}^\infty$ is not convergent.) We will see in Lemma \ref{th:first-lemma} below that the limit \eqref{eq:pressure-defn} always exists in $(-\infty,\infty]$ and is equal to the infimum of the same sequence.

In the case where $\I$ is finite it is well-established that $s \mapsto P(\A,s)$ is a continuous function $[0,\infty) \to \R$ and satisfies $P(\A,0)=\log\#\I$. If additionally $\sup_{i \in \I}\threebar{A_i}<1$ for some norm $\threebar{\,\cdot\,}$ on $\R^d$ then the function is additionally strictly decreasing with $\lim_{s \to \infty} P(\A,s)=-\infty$ and as such it has a unique zero. In the case where $\I$ is infinite, on the other hand, the situation is slightly more subtle. Let $\A=(A_i)_{i \in \mathcal{I}} \in \GL_d(\R)^\I$ where $\I$ is either finite or countably infinite. We define the \emph{finiteness threshold} of the pressure to be the quantity
\begin{equation} \label{eq:theta-def}
  \theta_\A=\inf\left\{ s \geq 0 \colon P(\A,s) <\infty\right\}
\end{equation}
if $\I$ is infinite, and $\theta_\A=0$ if $\I$ is finite. We also write
\begin{equation*}
  \mathscr{I}_\A=\left\{s \geq 0 \colon P(\A,s) \in \R \right\} \subseteq [\theta_\A,\infty).
\end{equation*}
A tuple $\A=(A_i)_{i \in \I} \in \GL_d(\R)^\I$ will be called \emph{irreducible} if there is no nonzero proper subspace $V \subset \R^d$ such that $A_iV=V$ for every $i \in \I$; otherwise $\A$ is \emph{reducible}. We also say that $\A$ is \emph{completely reducible} if in some basis the matrices in $\A$ are block-diagonal with irreducible blocks of the same size; see Section \ref{sec:completely-reducible} for further details. Our first main result describes the behaviour of the pressure functional $s \mapsto P(\A,s)$:

\begin{theorem}\label{th:aff}
  Let $\A=(A_i)_{i \in \I} \in \GL_d(\R)^\I$, where $\I$ is either finite or countably infinite. Then the following four assertions hold:
  \begin{enumerate}[(i)]
  \item \label{it:finiteness}
  The set $\mathscr{I}_\A$ is equal to either $[\theta_\A,\infty)$ or $(\theta_\A,\infty)$, and satisfies the alternative characterisation 
  \[
    \mathscr{I}_\A=\Bigl\{s \geq 0 \colon \sum_{i \in \I} \varphi^s(A_i)<\infty\Bigr\}.
  \]
  In particular, if $\I$ is finite then we have $\mathscr{I}_\A=[0,\infty)$.
  \item\label{it:right-con}
  The pressure function $s \mapsto P(\A,s)$ defined on $\mathscr{I}_\A$ is continuous, and is convex when restricted to the intervals $[k,k+1] \cap \mathscr{I}_\A$ for all $k \in \{0,\ldots,d-1\}$ and when restricted to the interval $[d,\infty) \cap \mathscr{I}_\A$. In particular, if $P(\A,\theta_\A)<\infty$, then
  \[
    \lim_{s \downarrow \theta_\A} P(\A,s)=P(\A,\theta_\A).
  \]
  \item\label{it:sup}
  For all $s \in \mathscr{I}_\A$ we have
  \[
    P(\A,s) = \sup\{ P((A_i)_{i \in \J} ,s) \colon \J\text{ is a nonempty finite subset of }\I\}
  \]
and if $\A$ is completely reducible then the above relation holds for every $s \ge 0$.
  \item\label{it:strictly}
  Let $\threebar{\,\cdot\,}$ be any norm on $\R^d$ and define $\kappa=-\log \sup_{i \in \I} \threebar{A_i}$. Then we have $P(\A,s+t) \leq P(\A,s) - \kappa t$ for all $s \in\mathscr{I}_\A$ and $t\geq 0$. In particular, if $ \sup_{i \in \I} \threebar{A_i}<1$, then $\kappa > 0$ and $s \mapsto P(\A,s)$ is strictly decreasing with $\lim_{s\to \infty} P(\A,s)=-\infty$.
  \end{enumerate}
\end{theorem}

Clauses \eqref{it:finiteness} and \eqref{it:strictly} of Theorem \ref{th:aff} are direct and straightforward to prove, and we present the proofs without delay in Section \ref{sec:preli}. But the proofs of \eqref{it:right-con} and \eqref{it:sup} are surprisingly involved and they are presented in Section \ref{sec:conditional-pressure} which further depends on the results of Section \ref{sec:algebra-stuff}. Those parts of \eqref{it:right-con} and \eqref{it:sup} which deal with the endpoint case $s=\theta_\A \in \mathscr{I}_\A$ are particularly involved. In Remark \ref{rem:assumption-needed}, we show that the assumption $P(\A,\theta_\A)<\infty$ in \eqref{it:right-con} is required for the right-continuity of the pressure at $\theta_\A$. If the tuple $\A$ consists only of a constant multiple of orthogonal matrices (or, more generally, the transformations $T_i$ are conformal and satisfy the bounded distortion property), then the corresponding theorem is much simpler and it is proved in \cite[Proposition 3.3]{MauldinUrbanski1996}.

If for some norm $\threebar{\,\cdot\,}$ on $\R^d$ we have $\sup_{i \in \I} \threebar{A_i}<1$,  we  define the \emph{upper affinity dimension} of $\A = (A_i)_{i \in \I}$ to be the quantity
\begin{equation} \label{eq:udimaff-def}
  \udimaff \A = \inf\{s \geq 0 \colon P(\A,s)<0\}
\end{equation}
and the \emph{lower affinity dimension} of $\A$ to be the quantity
\begin{equation}\label{eq:ldimaff-def}
  \ldimaff \A = \sup\left\{ \udimaff (A_i)_{i \in \J} \colon \J\text{ is a nonempty finite subset of }\I\right\}.
\end{equation}
It is readily checked that
\begin{equation} \label{eq:dimaff-ineq}
  \max\{\theta_\A,\ldimaff\A\} \leq \udimaff\A.
\end{equation}
Indeed, since the infimum is monotone with respect to inclusion, we necessarily have $\theta_{\mathsf{A}} \le \udimaff\mathsf{A}$. If $s>\udimaff \A$ and $\J \subseteq \I$ is any nonempty finite set then by a straightforward examination of the definition of the pressure we have $P((A_i)_{i \in \J},s) \leq P(\A,s)<0$ and it follows that $\dimaff (A_i)_{i \in \J} \leq s$. Thus $\ldimaff \A \leq s$ for all $s>\udimaff \A$ and hence, $\ldimaff \A \leq \udimaff \A$ as required.

If $\ldimaff \A = \udimaff \A$, then we denote the common value by $\dimaff \A$ and call it the \emph{affinity dimension} of $\A$. The following result is obtained as a corollary of Theorem \ref{th:aff} and we present its proof in Section \ref{sec:conditional-pressure}.

\newcommand{\propaffinity}{
  Let $\A=(A_i)_{i \in \I} \in \GL_d(\R)^\I$, where $\I$ is either finite or countably infinite, be such that $\sup_{i \in \I} \threebar{A_i}<1$ for some norm $\threebar{\,\cdot\,}$ on $\R^d$. If at least one of the following four assumptions,
  \begin{enumerate}[(i)]
    \item \label{it:propaff1} $0 \leq P(\A,\udimaff\A)<\infty$,
    \item \label{it:propaff2} $\theta_\A<\udimaff\A$,
    \item \label{it:propaff3} $\A$ is completely reducible,
    \item \label{it:propaff4} $\I$ is finite,
  \end{enumerate}
  holds, then $\ldimaff\A=\udimaff\A$.
}

\begin{proposition} \label{thm:propaffinity}
  \propaffinity
\end{proposition}

The situation in which $\ldimaff \A < \udimaff \A$ thus requires that $\udimaff \A = \theta_\A$, and this can be realised  both with  $P(\A,\udimaff\A)=\infty$ and with $P(\A,\udimaff\A)<0$,  as is demonstrated by the following proposition which will be proved in Section \ref{sec:conditional-pressure}:

\newcommand{\proppathologies}{
  For all $\alpha, \beta \in (0,1)$ and $\gamma \in (\beta,1]$ there exists a tuple of matrices $\mathsf{A} = (A_i)_{i \in \N} \in \GL_2(\R)^\N$ such that $\sup_{i \in \N} \|A_i\|<\alpha$ and
  \begin{equation*}
    \ldimaff\mathsf{A} = \beta < \gamma = \theta_{\mathsf{A}} = \udimaff\mathsf{A}.
  \end{equation*}
  Furthermore, $\mathsf{A}$ may be chosen such that $P(\mathsf{A},\theta_{\mathsf{A}})$ is either negative or infinite, as desired.
}

\begin{proposition} \label{thm:proppathologies}
  \proppathologies
\end{proposition}

If the tuple $\A$ consists only of constant multiples of orthogonal matrices then the strict inequality $\ldimaff \A < \udimaff \A$ cannot hold; this follows from the fact that such a tuple is necessarily completely reducible, but follows also from the antecedent result \cite[Theorem 3.15]{MauldinUrbanski1996} (which also applies if the transformations $T_i$ are assumed only to be conformal transformations with an appropriate bounded distortion property). As such the outcome $\ldimaff \A < \udimaff \A$ demonstrated in Proposition \ref{thm:proppathologies} is a phenomenon which is specific to the case of non-conformal infinite iterated function systems, and which has to the best of our knowledge not previously been remarked.

\subsection{Equilibrium states} \label{sec:equilibrium-states}

Our second major result is a complete description of the equilibrium states of $\varphi^s$ with respect to the full shift over a countable alphabet. This result extends the description given in the finite-alphabet case in \cite{BochiMorris2018}. Whereas in the finite-alphabet case the existence of at least one equilibrium state follows from a weak$^*$ compactness argument (see \cite{Kaenmaki2004}), in the countable-alphabet case no general existence results were previously known.

The collection of all Borel probability measures on $\I^\N$ will be denoted by $\MM(\I^\N)$. We let $\MM_\sigma(\I^\N)$ denote the set of all $\sigma$-invariant measures in $\MM(\I^\N)$, where $\sigma \colon \I^\N \to \I^\N$ is the \emph{left shift} taking $(i_k)_{k=1}^\infty$ into $(i_k)_{k=2}^\infty$. Here and hereafter we denote $\iii = i_1 \cdots i_n = (i_k)_{k=1}^n \in \I^n$, $\iii = i_1i_2\cdots = (i_k)_{k=1}^\infty \in \I^\N$, $\jjj|_n = j_1 \cdots j_n$ for all $\jjj \in \I^\N$, and $[\iii] = \{\jjj \in \I^\N \colon \jjj|_n=\iii\} \subset \I^\N$ for all $\iii \in \I^n$. The set $[\iii]$ is called a \emph{cylinder set} at level $n$ whenever $\iii \in \I^n$. We also write $\I^* = \bigcup_{n \ge 1} \I^n$.

Let $\mu \in \MM_\sigma(\I^\N)$, $\A=(A_i)_{i \in \I} \in \GL_d(\R)^\I$, and $s \ge 0$ be such that $\sup_{i \in \I} \fii^s(A_i) < \infty$. Recall that the singular value function $\fii^s$ satisfies $\fii^s(A) \le \|A\|^s$ for all $A \in \GL_d(\R)$. Therefore, $\sup_{i \in \I} \|A_i\| < \infty$ implies $\sup_{i \in \I} \fii^s(A_i) < \infty$ for all $s \ge 0$. We define the \emph{energy} of $\A$ at $s \ge 0$ with respect to $\mu \in \MM_\sigma(\I^\N)$ by setting
\begin{equation} \label{eq:energy}
  \Lambda(\mu,\A,s) = \lim_{n \to \infty} \frac{1}{n} \int_{\I^\N} \log\fii^s(A_{\iii|_n}) \dd\mu(\iii) \in [-\infty,\log \sup_{i \in \I} \fii^s(A_i)],
\end{equation}
where $A_\iii = A_{i_1} \cdots A_{i_n}$ for all $\iii = i_1 \cdots i_n \in \I^n$. We will see in Lemma \ref{thm:energy-basic-prop}\eqref{it:energy1} below that the limit \eqref{eq:energy} exists in $[-\infty,\log \sup_{i \in \I} \fii^s(A_i)]$ and is equal to the infimum of the same sequence.

For $\mu \in \MM(\I^\N)$ and a finite Borel partition $\PP$ of $\I^\N$ we define the \emph{Shannon entropy} by
\begin{equation} \label{eq:shannon-entropy}
  H(\mu,\PP) = -\sum_{C \in \PP} \mu(C)\log\mu(C) \in [0,\#\PP/e].
\end{equation}
Here we adopt the usual convention according to which $0 \log 0 = 0$. Recall that the $n$-level refinement of the partition $\PP$, denoted by $\bigvee_{i=0}^{n-1} \sigma^{-i}(\PP)$, is the collection of sets of the form $\bigcap_{i=0}^{n-1} \sigma^{-i}(C_{j_i})$, where $C_{j_i} \in \PP$. Note that the refinements are finite Borel partitions of $\I^\N$. We write
\begin{equation*}
  h(\mu,\PP) = \lim_{n \to \infty} \frac{1}{n} H\biggl( \mu,\bigvee_{i=0}^{n-1} \sigma^{-i}(\PP) \biggr) \in [0,\infty)
\end{equation*}
and define the \emph{Kolmogorov-Sinai entropy} of $\mu \in \MM_\sigma(\I^\N)$ by setting
\begin{equation} \label{eq:entropy}
  h(\mu) = \sup\{h(\mu,\PP) \colon \PP \text{ is a finite Borel partition of } \I^\N \} \in [0,\infty].
\end{equation}
We will see in Lemma \ref{thm:entropy-basic-prop}\eqref{it:entropy1} below that the limit $h(\mu,\PP)$ always exists in $[0,\infty)$ and is equal to the infimum of the same sequence.

The following proposition shows how these quantities are related to the pressure. We postpone its proof until Section \ref{sec:conditional-equilibrium}.

\newcommand{\equilibriuminequality}{
  Let $\A=(A_i)_{i \in \I} \in \GL_d(\R)^\I$ and $s \ge 0$ be such that $\sup_{i \in \I} \fii^s(A_i) < \infty$, where $\I$ is either finite or countably infinite. If $\mu \in \MM_\sigma(\I^\N)$ is such that $h(\mu)<\infty$ or $\Lambda(\mu,\A,s)>-\infty$, then
  \begin{equation*}
    h(\mu) + \Lambda(\mu,\A,s) \le P(\A,s).
  \end{equation*} 
}

\begin{proposition} \label{thm:equilibriuminequality}
  \equilibriuminequality
\end{proposition}

If $\mu \in \MM_\sigma(\I^\N)$ is such that $h(\mu)<\infty$ or $\Lambda(\mu,\A,s)>-\infty$, then we say that $\mu$ is an \emph{$\fii^s$-equilibrium state} for $\A$ if it satisfies
\begin{equation*}
  h(\mu) + \Lambda(\mu,\A,s) = P(\A,s).
\end{equation*} 
If $\I$ is finite, then the existence of an $\fii^s$-equilibrium state is proved in \cite{Kaenmaki2004} and their complete description is given in \cite{BochiMorris2018}. The next theorem extends this information into the countably infinite case by completely classifying the structure of the family of all $\fii^s$-equilibrium states. A tuple $\A=(A_i)_{i \in \I} \in \GL_d(\R)^\I$ is \emph{strongly irreducible} if there does not exist a finite collection $\mathcal{V}$ of nonzero proper subspaces of $\R^d$ such that $A_i\mathcal{V}=\mathcal{V}$ for every $i \in \I$. A strongly irreducible tuple is clearly irreducible. Given $\A=(A_i)_{i \in \I} \in \GL_d(\R)^\I$, we define $\A^{\wedge k}$ to be the tuple $(A_i^{\wedge k})_{i \in \I}$, where $A_i^{\wedge k} \colon \wedge^k \R^d \to \wedge^k \R^d$ is the induced invertible linear map. Recall that $\theta_\A$ is the finiteness threshold of the pressure defined in \eqref{eq:theta-def} and the set $\mathscr{I}_\A$ is equal to either $[\theta_\A,\infty)$ or $(\theta_\A,\infty)$.

\begin{theorem} \label{thm:equilibrium-states-classification}
  Let $\A=(A_i)_{i \in \I} \in \GL_d(\R)^\I$ be such that $\sup_{i \in \I} \fii^s(A_i) < \infty$, where $\I$ is either finite or countably infinite. If $s \in \mathscr{I}_\A$, then 
  \begin{equation*}
    P(\A,s) = \sup\{h(\mu) + \Lambda(\mu,\A,s) \colon \mu\in \mathcal{M}_\sigma(\I^\N) \text{ is such that } h(\mu)<\infty \}.
  \end{equation*}
  Furthermore, if $s > \theta_\A$, then the following three assertions hold:
  \begin{enumerate}[(i)]
  \item\label{it:eq1}
  If $s \geq d$ then there is a unique $\fii^s$-equilibrium state for $\A$ and it is a Bernoulli measure.
  \item\label{it:eq2}
  If $s \in (0,d) \cap \Z$ then the number of distinct ergodic $\fii^s$-equilibrium states for $\A$ is at least one and is not more than $\binom{d}{s}$. If $\A^{\wedge s}$ is irreducible then there is a unique $\fii^s$-equilibrium state for $\A$, and if additionally $\A^{\wedge s}$ is strongly irreducible then this unique equilibrium state is mixing.
  \item\label{it:eq3}
  If $s \in (0,d) \setminus \Z$ then the number of distinct ergodic $\fii^s$-equilibrium states for $\A$ is at least one and is not more than $\binom{d}{\lfloor s \rfloor}\binom{d}{\lceil s \rceil}$. If one of $\A^{\wedge \lfloor s\rfloor}$ and $\A^{\wedge \lceil s\rceil}$ is irreducible and the other is strongly irreducible then there is a unique $\fii^s$-equilibrium state for $\A$, and if  both are strongly irreducible then this unique equilibrium state is mixing.
  \end{enumerate}
  In all cases every equilibrium state is fully supported on $\I^\N$.
\end{theorem}

The proof of the theorem is given in Section \ref{sec:conditional-equilibrium} which further relies on the results of Section \ref{sec:algebra-stuff}.
The ergodic equilibrium states admit a precise description which is essentially identical to that given in \cite{BochiMorris2018} and which will be given in detail later. In Example \ref{ex:no-theta-eq-state}, we demonstrate that the $\fii^{s}$-equilibrium state does not necessarily exist when $s = \theta_\A \in \mathscr{I}_\A$. The following result is a consequence of the fact that equilibrium states are always fully supported:

\newcommand{\proppressuredrop}{
  Let $\A=(A_i)_{i \in I} \in \GL_d(\R)^\I$ where $\I$ is either finite or countably infinite. Then the following two assertions hold:
  \begin{enumerate}[(i)]
  \item\label{it:drop1}
  For every $s > \theta_\A$, if $\J$ is a nonempty proper subset of $\I$, then $P((A_i)_{i \in \J},s)<P(\A,s)$.
  \item\label{it:drop2}
  If there exists a norm $\threebar{\,\cdot\,}$ on $\R^d$ such that $\sup_{i \in \I} \threebar{A_i}<1$, and if $\theta_\A < \udimaff\A$, then $\udimaff (A_i)_{i \in \J} < \dimaff \A$ for every nonempty proper subset $\J$ of $\I$.
  \end{enumerate}
}

\begin{proposition} \label{thm:proppressuredrop}
  \proppressuredrop
\end{proposition}

\subsection{Self-affine sets} \label{sec:self-affine-sets}

The third class of results proved in this article present some applications to the dimension theory of infinitely generated self-affine sets. Let $(T_i)_{i \in \I}$ be an affine iterated function system acting on $\R^d$. Recall that every transformation $T_i$ is invertible and affine, having the form $T_i(x)=A_ix+v_i$ for some invertible linear map $A_i \in \GL_d(\R)$ and translation vector $v_i \in \R^d$. Denote the associated self-affine set by $X \subset \R^d$. We are interested in determining $\dimh X$, the Hausdorff dimension of $X$. We use the convention that whenever we speak about a self-affine set, then it is automatically accompanied with a tuple of affine maps which defines it. Write $T_\iii = T_{i_1} \circ \cdots \circ T_{i_n}$ for all $\iii = i_1 \cdots i_n$. Relying on the equation $X = \bigcup_{i \in \I} T_i(X)$, the self-affine set $X$ can naturally be covered by the sets $T_\iii(B)$, where $B$ is a ball containing $X$. The singular value function $\fii^s(A_\iii)$ represents a measurement of the $s$-dimensional volume of the image of the Euclidean unit ball under $T_\iii$. For example, in the planar case, the set $T_\iii(B)$ can be covered by one ball of radius $\sigma_1(A_\iii)\diam(B)$ or by $O(\sigma_1(A_\iii)/\sigma_2(A_\iii))$ balls of radius $\sigma_2(A_\iii)\diam(B)$. This motivates the study of the limiting behavior of sums $\sum_{\iii \in \I^n} \fii^s(A_\iii)$ and hence, the pressure $P(\A,s)$ of $\A = (A_i)_{i \in \I} \in \GL_d(\R)^\I$. In particular, Theorem \ref{th:aff}\eqref{it:sup} introduces a way to approximate the dimension of infinitely generated self-affine sets by their finitely generated self-affine subsets.

The first proposition is a rather standard covering argument and it generalises the classical estimate \cite[Proposition 5.1]{Falconer1988} into the infinitely generated case. All the results announced below will be proved in Section \ref{sec:conditional-self-affine}.

\newcommand{\selfaffinedimupperbound}{
  Let $X \subset \R^d$ be a self-affine set. Then $\dimh X \le \udimaff\A$.
}

\begin{proposition} \label{thm:selfaffinedimupperbound}
  \selfaffinedimupperbound
\end{proposition}

Let $\|\cdot\|$ be a norm on the vector space of affine maps from $\R^d$ into itself. We say that $X$ satisfies the \emph{exponential separation condition} if for every finite $\J \subseteq \I$ there exists $c_\J>0$ such that $\|T_\iii-T_\jjj\| \ge c_\J^n$ for all $n \ge 1$ and distinct $\iii,\jjj \in \J^n$. It is not difficult to see that the self-affine set $X$ satisfies exponential separation when the defining iterated function system generates a free semigroup and is defined by algebraic parameters (i.e.\ when $T_\iii \ne T_\jjj$ for all distinct $\iii, \jjj \in \I^*$ and all the entries of $A_i$ and $v_i$ are algebraic numbers).

The remaining theorems introduce sufficient conditions for the Hausdorff dimension of the self-affine set equal the affinity dimension. Furthermore, in such cases Proposition \ref{thm:proppressuredrop}\eqref{it:drop2} translates into a result that removing one of the defining affine maps results in a strict reduction of the Hausdorff dimension, a property which was previously demonstrated for finite affine iterated function systems in \cite{BochiMorris2018,KaenmakiMorris2018}. The following result generalises the seminal theorem \cite[Corollary 1.2]{Hochman2014} into the infinitely generated case. As affine maps acting on the real line are similarities, the self-affine sets on $\R$ are often called \emph{self-similar}.

\begin{theorem} \label{thm:selfaffinedimone}
  Let $X \subset \R$ be a self-similar set satisfying the exponential separation condition. Then $\dimh X = \min\{1,\dimaff\A\}$.
\end{theorem}

We say that the self-affine set $X$ satisfies the \emph{fixed point condition} if the maps $T_i$ in the defining affine iterated function system do not have a common fixed point. Furthermore, we say that $\A = (A_i)_{i \in \I} \in \GL_d(\R)^\I$ is \emph{proximal} if there exist $\iii_1,\iii_2,\ldots \in \I^*$ and $\alpha_1,\alpha_2,\ldots \in \R$ such that the sequence $(\alpha_n A_{\iii_n})_{n \ge 1}$ converges to a rank one linear transformation. In the case where $\A$ is additionally irreducible, this is equivalent to the existence of $\iii \in \I^*$ such that $A_\iii$ has a simple leading eigenvalue. 
The next theorem generalises both \cite[Theorem 1.1]{HochmanRapaport2022} and \cite[Theorem 1.1]{BaranyHochmanRapaport2019}, whose proofs are based on \cite[Theorem 2.3]{BaranyKaenmaki2017} and \cite[Theorem 1.1]{MorrisShmerkin2019}, into the infinitely generated case.

\begin{theorem} \label{thm:selfaffinedimtwo}
  Let $X \subset \R^2$ be a self-affine set satisfying the fixed point condition and the exponential separation condition such that the associated matrix tuple $\A$ is strongly irreducible and proximal. Then $\dimh X = \min\{2,\dimaff\A\}$.
\end{theorem}

Recall that $X$ satisfies the \emph{strong open set condition} if there exists a nonempty bounded open set $U \subset \R^d$ intersecting $X$ such that $\bigcup_{i \in \I} T_i(U) \subseteq U$ with disjoint union. By \cite[\S 6.2]{BaranyHochmanRapaport2019}, the strong open set condition implies the exponential separation. The following theorem generalises \cite[Theorem 1.5]{MorrisSert2023preprint}, whose proof rely on \cite[Theorem 1.9]{Rapaport2022preprint} and \cite[Theorem 1.4]{Feng2023}, into the infinitely generated case.

\begin{theorem} \label{thm:selfaffinedimthree}
  Let $X \subset \R^3$ be a self-affine set satisfying the strong open set condition such that the associated matrix tuple $\A$ is strongly irreducible and proximal. Then $\dimh X = \min\{3,\dimaff\A\}$.
\end{theorem}

In the following result, which generalises \cite[Theorem 5.3]{Falconer1988} and \cite[Theorem B]{KaenmakiReeve2014}, the self-affine sets $X_{\mathsf{v}}$ are parametrised by the tuples of associated translation vectors $\mathsf{v} = (v_i)_{i \in \I} \in (\R^d)^\I$. Define $\mathbf{Q} = ([0,1]^d)^\I$ and note that by Kolmogorov extension theorem $\mathbf{Q}$ supports a natural probability measure $\LL_{\mathbf{Q}} = (\LL^d|_{[0,1]^d})^\I$, where $\LL^d$ is the Lebesgue measure on $\R^d$.

\begin{theorem} \label{thm:selfaffinedimd}
  Let $X_{\mathsf{v}} \subset \R^d$ be a self-affine set such that the associated matrix tuple $\A = (A_i)_{i \in \I}$ satisfies $\sup_{i \in \I} \|A_i\| < \tfrac12$. Then $\dimh X_{\mathsf{v}} \ge \min\{d,\ldimaff\A\}$ for $\LL_{\mathbf{Q}}$-almost all $\mathsf{v} \in \mathbf{Q}$.
\end{theorem}

Under any of the assumptions \eqref{it:propaff1}--\eqref{it:propaff4} in Proposition \ref{thm:propaffinity}, the above theorem improves into $\dimh X_{\mathsf{v}} = \min\{d,\dimaff\A\}$ for $\LL_{\mathbf{Q}}$-almost all $\mathsf{v} \in \mathbf{Q}$. Recalling Proposition \ref{thm:proppathologies}, it would be interesting to know if in the context of Theorem \ref{thm:selfaffinedimd} there exist infinitely generated self-affine sets $X$ with $\dimh X < \udimaff\A$. If not, then Proposition \ref{thm:proppathologies} shows that there are infinitely generated self-affine sets $X$ with $\dimh \bigcup_{\J \subset \I \text{ is finite}} X_\J < \dimh X$, where $X_\J \subseteq X$ is the self-affine set associated to $(T_i)_{i \in \J}$.

The remainder of the article is structured as follows. In the following section we prove the elementary clauses \eqref{it:finiteness} and \eqref{it:strictly} of Theorem \ref{th:aff}. In Section \ref{sec:preli-entropy-energy}, we introduce necessary preliminaries on the energy and entropy, and in Section \ref{sec:completely-reducible}, we verify an important reduction that it suffices to work with completely reducible matrices. All the remaining results stated in the introduction are proved in Section \ref{sec:conditional-proof} under the assumption of a technical result whose proof is more algebraic in flavor. Finally, Section \ref{sec:algebra-stuff} is dedicated to the proof of this technical result.

%
%

\section{Existence and finiteness of pressure} \label{sec:preli}

In this section, we collect some elementary properties of the pressure $P(\A,s)$ for countably infinite affine iterated function systems. In particular, we prove the clauses \eqref{it:finiteness} and \eqref{it:strictly} in Theorem \ref{th:aff}. Without further mentioning, we use notation concerning words introduced in Section \ref{sec:equilibrium-states}. We begin with a fundamental lemma concerning the sequence used to define the pressure.

\begin{lemma}\label{le:zeroth-lemma}
  Let $\I$ be countably infinite, $\A=(A_i)_{i \in \I} \in \GL_d(\R)^\I$, and $s \geq 0$. If $\sum_{i \in \I} \varphi^{s}(A_i)<\infty$, then for every $n \geq 1$ the series $\sum_{\iii \in \I^n} \varphi^t(A_\iii)$ converges uniformly on $[s,\infty)$.
\end{lemma}

\begin{proof}
  Suppose that $\sum_{i \in \I} \varphi^{s}(A_i)<\infty$ and define  $\I^+=\{i \in \I \colon \varphi^{s}(A_i) \geq 1\}$ and $\I^-=\{i \in \I \colon \varphi^{s}(A_i)<1\}$. Since by hypothesis $\sum_{i \in \I} \varphi^{s}(A_i)<\infty$, the set $\I^+$ is finite. For all $i \in \I^-$ we have
  \[
  \begin{cases}
    \sigma_{\lceil s\rceil}(A_i)^{s} \leq \varphi^{s}(A_i)<1, &\text{if } 0 \le s < d, \\ 
    |\det A_i| < 1, &\text{if } s \ge d,
  \end{cases}
  \]
  and hence,
  \[
  \varphi^{t}(A_i) \leq
  \begin{cases}
    \varphi^{s}(A_i) \sigma_{\lceil s\rceil}(A_i)^{t-s} \leq \varphi^{s}(A_i), &\text{if } 0 \le t < d, \\
    \varphi^{s}(A_i) |\det A_i|^{\frac{t-\max\{s,d\}}{d}} \leq \varphi^{s}(A_i), &\text{if } t \ge d
  \end{cases}
  \]
  for all $t \geq s$. Let $\varepsilon>0$ and choose a finite set $\mathcal{J}\subset \mathcal{I}$ such that $\sum_{i \in\mathcal{I}\setminus \mathcal{J}}\varphi^{s}(A_i)<\varepsilon$. It follows directly that if $t \geq s$, then
  \[
    \sum_{i \in \I \setminus (\I^+ \cup \J)}\varphi^t(A_i)=\sum_{i \in \I^- \setminus \J}\varphi^t(A_i) \leq \sum_{i \in \I^- \setminus \J}\varphi^{s}(A_i) \leq \sum_{i \in \I \setminus \J}\varphi^{s}(A_i)  <\varepsilon.
  \]
  We have proved that $\sum_{i \in \I}\varphi^t(A_i)$ converges uniformly on $[s,\infty)$ and this proves the lemma in the case $n=1$.

  To show the general case, observe first that if $A,B \in \GL_d(\R)$, then, by a well-known inequality,
  \begin{equation} \label{eq:svf-sub-multi}
    \fii^{s}(AB) \le \fii^{s}(A) \fii^{s}(B);    
  \end{equation}
  for example, see \cite[\S 3.4]{KaenmakiMorris2018}. If $n \ge 1$, then by summing we get
  \begin{equation} \label{eq:svf-sum-sub-multi}
    \sum_{\iii \in \I^{n+1}} \fii^{s}(A_\iii) \le \biggl( \sum_{\iii \in \I^n} \fii^{s}(A_\iii) \biggr)\biggl( \sum_{i \in \I} \fii^{s}(A_i) \biggr).
  \end{equation}
  Therefore, since $\sum_{\iii \in \I^n}\varphi^{s}(A_\iii)\le (\sum_{i \in \I}\varphi^{s}(A_i))^n$ by a simple induction, the convergence of $\sum_{i \in \I}\varphi^{s}(A_i)$ immediately implies the convergence of $\sum_{\iii \in \I^n}\varphi^{s}(A_\iii)$, and the result follows by applying the above arguments to $\I^n$ in place of $\I$.
\end{proof}

The following lemma verifies the existence of the pressure defined in \eqref{eq:pressure-defn} and proves Theorem \ref{th:aff}\eqref{it:finiteness}. The claim \eqref{it:l11} of the lemma extends \cite[Proposition 2.1.9]{MauldinUrbanski2003} into the self-affine setting.

\begin{lemma} \label{th:first-lemma}
  Let $\A=(A_i)_{i \in \I} \in \GL_d(\R)^\I$ where $\I$ is either finite or countably infinite. Then the following three assertions hold:
  \begin{enumerate}[(i)]
  \item\label{it:l11}
  For every $s \geq 0$ the limit
  \begin{equation}\label{eq:pressure-sequence}
    P(\mathsf{A},s) = \lim_{n\to\infty} \frac{1}{n}\log \sum_{\iii \in \I^n} \varphi^s(A_\iii) = \inf_{n \geq 1}\frac{1}{n}\log \sum_{\iii \in \I^n} \varphi^s(A_\iii)
  \end{equation}
  exists in $(-\infty,\infty]$ and is finite if and only if $\sum_{i \in \I}\varphi^s(A_i)<\infty$ is finite. In particular
  \[  
    \mathscr{I}_\A=\{s \ge 0 \colon P(\mathsf{A},s)<\infty\}=\Bigl\{s \geq 0 \colon \sum_{i \in \I} \varphi^s(A_i)<\infty\Bigr\}.
  \]
  \item\label{it:l12} The set $\mathscr{I}_\A$ is equal to either $(\theta_\A,\infty)$ or $[\theta_\A,\infty)$.
  \item\label{it:l13} For every $n \geq 1$ the function $s\mapsto \sum_{\iii \in \I^n} \varphi^s(A_\iii)$ defined on $\mathscr{I}_\A$ is continuous.
  \end{enumerate}
\end{lemma}

\begin{proof}
  Recall that if $\I$ is finite, then $\theta_\A=0$ and $\mathscr{I}_\A=[0,\infty)$. To prove the claims, we first observe that if $A,B \in \GL_d(\R)$ and $s \ge 0$, then \eqref{eq:svf-sub-multi} implies $\fii^s(A) = \fii^s(ABB^{-1}) \le \fii^s(AB)\fii^s(B^{-1})$. Since $\fii^s(B^{-1}) \le \sigma_1(B^{-1})^s = \sigma_d(B)^{-s}$, we thus get
  \[
    \fii^s(A)\sigma_d(B)^s \le \fii^s(AB).
  \]
  The inequality
  \begin{equation} \label{eq:svf-sum-super-multi}
    \biggl( \sum_{\iii \in \I^n} \fii^s(A_\iii) \biggr)\biggl( \sum_{i \in \I} \sigma_d(A_\iii)^s \biggr) \le \sum_{\iii \in \I^{n+1}} \fii^s(A_\iii)
  \end{equation}
  for all $n \geq 1$ follows immediately by summation.

  A simple induction using \eqref{eq:svf-sum-sub-multi} shows that if $\sum_{i \in \I} \varphi^s(A_i)$ is finite, then $\sum_{\iii \in \I^n} \varphi^s(A_\iii) \leq (\sum_{i \in \I} \varphi^s(A_i))^n<\infty$ for all $n \geq 1$; on the other hand if $\sum_{i \in \I} \varphi^s(A_i)$ is infinite then a similar induction using \eqref{eq:svf-sum-super-multi} implies that $\sum_{\iii \in \I^{n}}\varphi^s(A_\iii)$ is infinite for every $n \geq 1$. We conclude that the sequence in \eqref{eq:pressure-sequence} is either finite for every $n \geq 1$ or infinite for every $n \geq 1$. In the former case the existence of the limit \eqref{eq:pressure-sequence} and its identity with the claimed infimum follow from the subadditivity lemma; in the latter case the same properties follow trivially since the sequence is identically equal to $\infty$. We have proved \eqref{it:l11}. Applying this result together with Lemma \ref{le:zeroth-lemma} we see that
  \[
    \mathscr{I}_\A=\bigcup_{s \in \mathscr{I}_\A} [s,\infty)
  \]
  and this is equal to either $(\theta_\A,\infty)$ or $[\theta_\A,\infty)$ depending respectively on whether or not $\theta_\A \in \mathscr{I}_\A$, which yields \eqref{it:l12}.

  For every $n\geq 1$, by Lemma \ref{le:zeroth-lemma} the series $\sum_{\iii \in \I^n} \varphi^s(A_\iii)$ converges uniformly with respect to $s$ on closed subintervals of $\mathscr{I}_\A$. As the function $s \mapsto \varphi^s(A_\iii)$ is continuous for every $\iii \in \I^*$, this implies that each of the function
  \[
    s \mapsto \sum_{\iii \in \I^n} \fii^s(A_\iii)
  \]
  is continuous on $\mathscr{I}_\A$ proving \eqref{it:l13}.
\end{proof}

With elementary methods, we see that the pressure is always decreasing in $s$ and, if $\I$ is finite, also continuous. In fact, the following lemma proves Theorem \ref{th:aff}\eqref{it:strictly}.

\begin{lemma}\label{le:first-cty}
  Let $\I$ be finite or countably infinite, $\threebar{\,\cdot\,}$ be any norm on $\R^d$, and $\A=(A_i)_{i \in \I} \in \GL_d(\R)^\I$ be bounded. Then $P(\A,s+t) \leq P(\A,s)+Ct$ for all $s \in \mathscr{I}_\A$ and all $t \geq 0$, where $C=\log\sup_{i \in \I} \threebar{A_i}$. Furthermore, if $\I$ is finite then $s \mapsto P(\A,s)$ is continuous on $[0,\infty)$.
\end{lemma}

\begin{proof}
  Let $K>0$ be a constant such that $K^{-1}\|v\| \leq \threebar{v} \leq K\|v\|$ for all $v \in \R^d$, let $s \in \mathscr{I}_\A$ be arbitrary and let $t\geq0$. Using Lemma \ref{th:first-lemma}\eqref{it:l11} we have for all $n \geq 1$
  \[
    \exp\left( nP(\A,s+t)\right) \leq \sum_{\iii\in \I^n} \varphi^{s+t}(A_\iii) \leq \sum_{\iii\in \I^n} \|A_\iii\|^t \varphi^s(A_\iii) \leq Ke^{ntC} \sum_{\iii\in \I^n} \varphi^s(A_\iii)
  \]
  and the inequality $P(\A,s+t) \leq P(\A,s)+Ct$ follows easily. If $\I$ is finite then clearly
  \[
    \sum_{\iii\in \I^n} \varphi^{s+t}(A_\iii)\geq \sum_{\iii\in \I^n} \varphi^{s}(A_\iii)\sigma_d(A_\iii)^t \geq (\min_{i \in \I} \sigma_d(A_i))^{nt} \sum_{\iii\in \I^n} \varphi^s(A_\iii)
  \]
and so
  \[
    P(\A,s+t) \geq P(\A,s) + t\log \min_{i \in \I} \sigma_d(A_i)
  \]
  for all $s \in \mathscr{I}_\A$ and $t \geq 0$. The continuity of $s \mapsto P(\A,s)$ for finite $\I$ follows.
\end{proof}

To finish this section, let us verify the easy part of Theorem \ref{th:aff}\eqref{it:right-con} by showing that the pressure is convex between two consecutive integers.

\begin{lemma}\label{le:cvx}
  Let $\I$ be finite or countably infinite and let $\A=(A_i)_{i \in \I} \in \GL_d(\R)^\I$. Then the function $s \mapsto P(\A,s)$ is convex on $[k,k+1] \cap \mathscr{I}_\A$ for all $k \in \{0,\ldots,d-1\}$ and is also convex on $[d,\infty) \cap \mathscr{I}_\A$.
\end{lemma}

\begin{proof}
  We will prove the case of convexity on $[k,k+1] \cap \mathscr{I}_\A$; the case $[d,\infty) \cap \mathscr{I}_\A$ is similar and left to the reader. Fix $k \in \{0,\ldots,d-1\}$ and $t \in (0,1)$ and take $s_1,s_2 \in [k,k+1] \cap (\theta_\A,\infty)$ such that $s_1<s_2$. Since $\lfloor s_1 \rfloor = \lceil s_2 \rceil-1 = \lfloor ts_2+(1-t)s_1 \rfloor = k$, we have
  \begin{align*}
    \fii^{ts_2+(1-t)s_1}(A_\iii) &= \bigl(\sigma_1(A_\iii) \cdots \sigma_k(A_\iii)\sigma_{k+1}(A_\iii)^{s_2-k}\bigr)^t \\ 
    &\qquad\qquad\cdot\bigl(\sigma_1(A_\iii) \cdots \sigma_k(A_\iii)\sigma_{k+1}(A_\iii)^{s_1-k}\bigr)^{1-t} \\ 
    &= \fii^{s_2}(A_\iii)^t \fii^{s_1}(A_\iii)^{1-t}
  \end{align*}
  for all $\iii \in \I^n$ and $n \ge 1$. Therefore, by H\"older's inequality with H\"older conjugates $p=1/t$ and $q=1/(1-t)$,
  \begin{equation*}
    \sum_{\iii \in \I^n} \fii^{ts_2+(1-t)s_1}(A_\iii) \le \biggl( \sum_{\iii \in \I^n} \fii^{s_2}(A_\iii) \biggr)^t \biggl( \sum_{\iii \in \I^n} \fii^{s_1}(A_\iii) \biggr)^{1-t}
  \end{equation*}
  for all $n \ge 1$. The claim follows directly by taking the logarithm, dividing by $n$, and passing to the limit.
\end{proof}

Let $\A=(A_i)_{i \in \I} \in \GL_d(\R)^\I$ where $\I$ is either finite or countably infinite. If $\I$ is finite, then Lemma \ref{le:first-cty} shows the continuity of $s \mapsto P(\A,s)$ on $[0,\infty)$. Proving the continuity when $\I$ is countably infinite is more involved and is not fully addressed until in Section \ref{sec:conditional-proof}. The elementary results stated in this section give some information in this direction. At first, as a convex funtion defined on an open set is continuous, Lemma \ref{le:cvx} implies that $s \mapsto P(\A,s)$ is continuous on open intervals $(0,1) \cap (\theta_\A,\infty),\ldots,(d-1,d) \cap (\theta_\A,\infty)$, and $(d,\infty) \cap (\theta_\A,\infty)$.

For each $n \ge 1$, by Lemma \ref{th:first-lemma}\eqref{it:l13} we see that the function
\begin{equation*}
  s \mapsto \frac{1}{n} \log\sum_{\iii \in \I^n} \fii^s(A_\iii)
\end{equation*}
is continuous on $\mathscr{I}_\A$. Recalling \eqref{eq:pressure-sequence}, it follows that $s \mapsto P(\A,s)$ is upper semi-continuous on $\mathscr{I}_\A$ as it is the pointwise infimum of a sequence of continuous functions on that domain. In particular, we have
\begin{equation*}
  \limsup_{t \uparrow s} P(\A,t) \le P(\A,s)
\end{equation*}
for all $s \in (\theta_\A,\infty)$. If there exists a norm $\threebar{\,\cdot\,}$ on $\R^d$ such that $\sup_{i} \threebar{A_i} \le 1$, then $s \mapsto P(\A,s)$ is decreasing by Lemma \ref{le:first-cty} and, in particular, we have
\begin{equation*}
  \liminf_{t \uparrow s} P(\A,t) \ge P(\A,s)
\end{equation*}
for all $s \in (\theta_\A,\infty)$. Therefore, the function $s \mapsto P(\A,s)$ is in this case left-continuous at points in $\{1,\ldots,d\} \cap (\theta_\A,\infty)$. The task in Section \ref{sec:conditional-proof} is thus to prove the right-continuity at $\theta_\A$ and in $\{1,\ldots,d\} \cap (\theta_\A,\infty)$.

%
%

\section{Preliminaries on energy and entropy} \label{sec:preli-entropy-energy}

In this section, we examine the behaviour of the energy and entropy. Our first lemma verifies the basic properties of the energy. Recall that the singular value function $\fii^s$ is submultiplicative. We state the lemma for general submultiplicative potentials as we will later define the energy also for other functions. The lemma shows that the energy $\Lambda(\mu,\A,s)$ defined in \eqref{eq:energy} exists in $[-\infty,\sup_{i \in \I}\log\fii^s(A_i)]$ and is equal to the infimum of the same sequence over $n$.

\begin{lemma} \label{thm:energy-basic-prop}
  Let $\I$ be either finite or countably infinite set and $\psi \colon \I^* \to (0,\infty)$ be such that $\psi(\iii\jjj) \le \psi(\iii)\psi(\jjj)$ for all $\iii,\jjj \in \I^*$ and $\sup_{i \in \I} \psi(i) < \infty$. Then the following two assertions hold:
  \begin{enumerate}[(i)]
    \item\label{it:energy1} For each $\mu \in \MM_\sigma(\I^\N)$ the limit
    \begin{equation*}
      \Lambda(\mu,\psi) = \lim_{n \to \infty} \frac{1}{n} \int_{\I^\N} \log\psi(\iii|_n) \dd\mu(\iii)
    \end{equation*}
    exists in $[-\infty,\log\sup_{i \in \I}\psi(i)]$ and is equal to $\inf_{n \ge 1} \frac{1}{n} \int_{\I^\N} \log\psi(\iii|_n) \dd\mu(\iii)$.
    \item\label{it:energy2} The map $\mu \mapsto \Lambda(\mu,\psi)$ defined on $\MM_\sigma(\I^\N)$ is upper semicontinuous.
  \end{enumerate}
\end{lemma}

\begin{proof}
  Write $M = \sup_{i \in \I}\psi(i)$. By the submultiplicativity of $\psi$ and the $\sigma$-invariance of $\mu$, we have
  \begin{equation} \label{eq:energy-basic-prop1}
  \begin{split}
    \int_{\I^\N} \log\psi(\iii|_{m+n}) \dd\mu(\iii) &\le \int_{\I^\N} \log\psi(\iii|_m) \dd\mu(\iii) + \int_{\I^\N} \log\psi(\sigma^m\iii|_n) \dd\mu(\iii) \\
    &= \int_{\I^\N} \log\psi(\iii|_m) \dd\mu(\iii) + \int_{\I^\N} \log\psi(\iii|_n) \dd\mu(\iii).
  \end{split}
  \end{equation}
  In particular, we have
  \begin{equation*}
    \int_{\I^\N} \log\psi(\iii|_{n+1}) \dd\mu(\iii) \le \int_{\I^\N} \log\psi(\iii|_{n}) \dd\mu(\iii) + \log M
  \end{equation*}
  and therefore,
  \begin{equation*}
    \frac{1}{n}\int_{\I^\N} \log\psi(\iii|_{n}) \dd\mu(\iii) \le \log M
  \end{equation*}
  for all $n \ge 1$. It also follows that if $\int_{\I^\N} \log\psi(\iii|_{n}) \dd\mu(\iii) = -\infty$ for some $n$, then $\int_{\I^\N} \log\psi(\iii|_{m}) \dd\mu(\iii) = -\infty$ for all $m \ge n$. Hence, the limit $\Lambda(\mu,\psi)$ exists in $[-\infty,\log M]$ and is equal to the infimum of the same sequence over $n$ by the subadditivity \eqref{eq:energy-basic-prop1} as claimed in \eqref{it:energy1}. 

  Recall that, by the definition of the weak$^*$ convergence, a function $\mu \mapsto \int_{\I^\N} f \dd\mu$ is continuous whenever $f \colon \I^\N \to \R$ is bounded and continuous. Fix $n \ge 1$ and define for each $\eps > 0$ a function $\psi_\eps \colon \I^\N \to (0,M^n]$ by setting $\psi_\eps(\iii) = \max\{\eps, \psi(\iii)\}$ for all $\iii \in \I^\N$. It follows that $\mu \mapsto \int_{\I^\N} \log\psi_\eps(\iii|_n) \dd\mu(\iii)$ is continuous and $\int_{\I^\N} \log\psi(\iii|_n) \dd\mu(\iii) \le \int_{\I^\N} \log\psi_\eps(\iii|_n) \dd\mu(\iii)$ for all $\eps > 0$. By the monotone convergence theorem, we thus find that the map $\mu \mapsto \int_{\I^\N} \log\psi(\iii|_n) \dd\mu(\iii)$ is an infimum of continuous functions and hence upper semicontinuous for all $n \ge 1$. Therefore, by \eqref{it:energy1}, we see that $\mu \mapsto \Lambda(\mu,\psi)$ is an infimum of upper semicontinuous functions and the assertion \eqref{it:energy2} follows.
\end{proof}

Our second lemma verifies the existence of the entropy defined in \eqref{eq:entropy}. If $\I$ is finite, then $\I^\N$ is compact and the existence of the entropy and its basic properties follow immediately from \cite{Bowen}. To extend the results to case where $\I$ is countably infinite, we rely on the ultrametric structure of the non-compact set $\I^\N$. Let $\PP$ and $\QQ$ be either finite or countably infinite Borel partitions of $\I^\N$, extend the definition of the Shannon entropy \eqref{eq:shannon-entropy} to countably infinite partitions, and set
\begin{equation*}
  H(\mu,\PP\,\vert\,\QQ) = \sum_{D \in \QQ} \mu(D) H\biggl(\frac{\mu|_D}{\mu(D)},\PP\biggr) \in [0,\#\PP/e]
\end{equation*}
for all $\mu \in \MM(\I^\N)$.

\begin{lemma} \label{thm:entropy-basic-prop}
  Let $\I$ be either finite or countably infinite set and $\PP$ be either finite or countably infinite Borel partition of $\I^\N$. Then the following two assertions hold:
  \begin{enumerate}[(i)]
    \item\label{it:entropy1} For each $\mu \in \MM_\sigma(\I^\N)$ the limit
    \begin{equation*}
      h(\mu,\PP) = \lim_{n \to \infty} \frac{1}{n} H\biggl( \mu,\bigvee_{i=0}^{n-1} \sigma^{-i}(\PP) \biggr)
    \end{equation*}
    exists in $[0,\infty]$, is equal to $\inf_{n \ge 1} \frac{1}{n} H(\mu,\bigvee_{i=0}^{n-1} \sigma^{-i}(\PP))$, and is finite if $\PP$ is finite. Furthermore, either all terms of this sequence are finite, or all are infinite.
    \item\label{it:entropy2} If $\QQ$ is either finite or countably infinite Borel partition of $\I^\N$, then for each $\mu \in \MM_\sigma(\I^\N)$ we have
    \begin{equation*}
      h(\mu,\PP) \le h(\mu,\QQ) + H(\mu,\PP\,|\,\QQ).
    \end{equation*}
    In particular, if each set in $\PP$ is a union of sets in $\QQ$, then $h(\mu,\PP) \le h(\mu,\QQ)$.
  \end{enumerate}
\end{lemma}

\begin{proof}
  Let $\PP$ and $\QQ$ both be either finite or countably infinite Borel partitions of $\I^\N$. Recall that $\PP \vee \QQ$ is the collection of sets of the form $C \cap D$, where $C \in \PP$ and $D \in \QQ$. Note that
  \begin{equation} \label{eq:cond-entropy-as-difference}
  \begin{split}
    H(\mu,\PP\,|\,\QQ) + H(\mu,\QQ) &= -\sum_{D \in \QQ} \mu(D) \sum_{C \in \PP} \frac{\mu(C \cap D)}{\mu(D)} \log\frac{\mu(C \cap D)}{\mu(D)} \\ 
    &\qquad\qquad\qquad\qquad\qquad\qquad- \sum_{D \in \QQ} \mu(D)\log\mu(D) \\
    &= -\sum_{D \in \QQ} \sum_{C \in \PP} \mu(C \cap D) \log\mu(C \cap D) = H(\mu,\PP \vee \QQ)
  \end{split}
  \end{equation}
  and hence, $H(\mu,\QQ) \le H(\mu,\PP\vee\QQ)$ and
  \begin{equation} \label{eq:claim2-for-shannon}
  \begin{split}
    H(\mu,\PP) &= H(\mu,\PP\vee\QQ) - H(\mu,\QQ\,|\,\PP) \\ 
    &\le H(\mu,\PP\vee\QQ) = H(\mu,\QQ) + H(\mu,\PP\,|\,\QQ),
  \end{split}
  \end{equation}
  where, while calculating, we assume the quantities finite and then observe that the claimed inequality holds also when any of the quantities is infinite. In particular, if each set in $\QQ$ is a union of sets in $\PP$, then trivially $\PP \vee \QQ = \PP$ and
  \begin{equation} \label{eq:entropy-basic-prop12}
    H(\mu,\QQ) \le H(\mu,\PP \vee \QQ) = H(\mu,\PP).
  \end{equation}
  If $\RR$ is either finite or countably infinite Borel partition of $\I^\N$ such that each set in $\RR$ is a union of sets in $\QQ$, then Jensen's inequality on the concave function $x \mapsto -x\log x$ implies
  \begin{equation} \label{eq:cond-entropy-inequality}
  \begin{split}
    H(\mu,\PP\,|\,\QQ) &= \sum_{C \in \PP} \sum_{E \in \RR} \mu(E) \sum_{\atop{D \in \QQ}{D \subset E}} \frac{\mu(D)}{\mu(E)} \biggl( -\frac{\mu(C \cap D)}{\mu(D)} \log\frac{\mu(C \cap D)}{\mu(D)} \biggr) \\ 
    &\le \sum_{C \in \PP} \sum_{E \in \RR} \mu(E) \biggl( -\sum_{\atop{D \in \QQ}{D \subset E}} \frac{\mu(C \cap D)}{\mu(E)} \log\frac{\mu(C \cap D)}{\mu(E)} \biggr) \\ 
    &= -\sum_{E \in \RR} \mu(E) \sum_{C \in \PP} \frac{\mu(C \cap E)}{\mu(E)} \log\frac{\mu(C \cap E)}{\mu(E)} = H(\mu,\PP\,|\,\RR).
  \end{split}
  \end{equation}
  Note that choosing $\RR = \{\I^\N\}$ above gives $H(\mu,\PP\,|\,\QQ) \le H(\mu,\PP)$ and therefore, by \eqref{eq:cond-entropy-as-difference},
  \begin{equation} \label{eq:entropy-basic-prop11}
    H(\mu,\PP \vee \QQ) \le H(\mu,\PP) + H(\mu,\QQ).
  \end{equation}
  Furthermore, if $\RR$ is either finite or countably infinite Borel partition of $\I^\N$, then it follows from \eqref{eq:cond-entropy-as-difference} and \eqref{eq:cond-entropy-inequality} that
  \begin{equation} \label{eq:cond-entropy-finer}
  \begin{split}
    H(\mu,\PP\vee\QQ\,|\,\RR) &= H(\mu,\PP\vee\QQ\vee\RR) - H(\mu,\QQ\vee\RR) + H(\mu,\QQ\vee\RR) - H(\mu,\RR) \\ 
    &= H(\mu,\PP\,|\,\QQ\vee\RR) + H(\mu,\QQ\,|\,\RR) \le H(\mu,\PP\,|\,\RR) + H(\mu,\QQ\,|\,\RR),
  \end{split}
  \end{equation}
  where, while calculating, we assume the quantities finite and then observe that the claimed inequality holds also when any of the quantities is infinite. Write $\PP^n = \bigvee_{i=0}^{n-1} \sigma^{-i}(\PP)$ and note that $\PP^{m+n} = \PP^m \vee \bigvee_{i=m}^{m+n-1} \sigma^{-i}(\PP) = \PP^m \vee \sigma^{-m}(\PP^n)$. By \eqref{eq:entropy-basic-prop11} and the $\sigma$-invariance of $\mu$, we get
  \begin{equation*}
    H(\mu,\PP^{m+n}) \le H(\mu,\PP^m) + H(\mu,\sigma^{-m}(\PP^n)) = H(\mu,\PP^m) + H(\mu,\PP^n)
  \end{equation*}
  for all $m,n \ge 1$. Since each set in $\PP^n$ is a union of sets in $\PP^{n+1}$, we see that, by \eqref{eq:entropy-basic-prop12}, $H(\mu,\PP^n)$ is increasing with respect to $n$. Therefore, by the subadditivity, the limit $\lim_{n \to \infty} \frac{1}{n} H(\mu,\PP^n)$ exists in $[0,\infty]$, equals the infimum $\inf_{n \ge 1} \frac{1}{n} H(\mu,\PP^n)$, and either all terms of this sequence are finite, or all are infinite. The claim \eqref{it:entropy1} follows.

  To prove the claim \eqref{it:entropy2}, we first note that, as $\mu$ is $\sigma$-invariant, \eqref{eq:cond-entropy-as-difference} implies
  \begin{equation} \label{eq:cond-entropy-invariant}
  \begin{split}
    H(\mu,\sigma^{-i}(\PP)\,|\,\sigma^{-i}(\QQ)) &= H(\mu,\sigma^{-i}(\PP)\vee\sigma^{-i}(\QQ)) - H(\mu,\sigma^{-i}(\QQ)) \\ 
    &= H(\mu,\PP\vee\QQ) - H(\mu,\QQ) = H(\mu,\PP\,|\,\QQ)
  \end{split}
  \end{equation}
  for all $i \ge 1$. Write $\QQ^n = \bigvee_{i=0}^{n-1} \sigma^{-i}(\QQ)$ and observe that, by \eqref{eq:claim2-for-shannon}, \eqref{eq:cond-entropy-finer}, \eqref{eq:cond-entropy-inequality}, and \eqref{eq:cond-entropy-invariant},
  \begin{align*}
    H(\mu,\PP^n) &\le H(\mu,\QQ^n) + H(\mu,\PP^n\,|\,\QQ^n) \le H(\mu,\QQ^n) + \sum_{i=0}^{n-1} H(\mu,\sigma^{-i}(\PP)\,|\,\QQ^n) \\ 
    &\le H(\mu,\QQ^n) + \sum_{i=0}^{n-1} H(\mu,\sigma^{-i}(\PP)\,|\,\sigma^{-i}(\QQ)) = H(\mu,\QQ^n) + nH(\mu,\PP\,|\,\QQ).
  \end{align*}
  Dividing by $n$ before letting $n \to \infty$, the assertion \eqref{it:entropy1} shows that $h(\mu,\PP) \le h(\mu,\QQ) + H(\mu,\PP\,|\,\QQ)$. If each set in $\PP$ is a union of sets in $\QQ$, then trivially $\PP\vee\QQ = \QQ$ and $H(\mu,\PP\,|\,\QQ) = H(\mu,\PP \vee \QQ) - H(\mu,\QQ) = 0$ by \eqref{eq:cond-entropy-as-difference}.
\end{proof}

For each nonempty subset $\J$ of $\I$ let
\begin{equation*}
  \mathcal{I}_\J =
  \begin{cases}
    \{\{i\} \colon i \in \J\} \cup \{\I \setminus \J\}, &\text{if } \I \setminus \J \ne \emptyset, \\ 
    \{\{i\} \colon i \in \J\}, &\text{if } \I \setminus \J = \emptyset
  \end{cases}
\end{equation*}
be a partition of $\I$ and let
\begin{equation*}
  \PP_\J = \biggl\{ \bigcup_{i \in I} [i] \colon I \in \mathcal{I}_\J \biggr\}
\end{equation*}
be a Borel partition of $\I^\N$. Note that if $\J$ is finite, then $\PP_\J$ is a finite partition. If $\iii = (i_k)_{k=1}^n \in \I^n$, then the corresponding cylider is $[\iii] = \{\jjj \in \I^\N \colon \jjj|_n=\iii\}$. We extend this definition as follows: If $\mathtt{I} = I_1 \cdots I_n = (I_k)_{k=1}^n \in \mathcal{I}_\J^n$, then the corresponding \emph{generalized cylinder} is
\begin{equation*}
  [\mathtt{I}] = \bigcup_{i_1 \in I_1} \cdots \bigcup_{i_n \in I_n} [i_1 \cdots i_n].
\end{equation*}
Note that for each $\mathtt{I} = I_1 \cdots I_n \in \mathcal{I}_\I^n$ there is $\iii = i_1\cdots i_n \in \I^n$ such that $I_k=\{i_k\}$ for all $k \in \{1,\ldots,n\}$ and $[\mathtt{I}] = [\iii]$. Notice that $\PP_\J = \{[I] \colon I\in \mathcal{I}_\J\}$ and, since $\I^\N$ and all the first level cylinders $[i]$ are open and closed, the partition $\PP_\J$ consists of sets which are open and closed. The $n$-level refinement of the partition $\PP_\J$ is
\begin{equation} \label{eq:refinement-gen-cyl}
  \bigvee_{i=0}^{n-1} \sigma^{-i}(\PP_\J) = \{[\mathtt{I}] \colon \mathtt{I} \in \mathcal{I}_\J^n\}
\end{equation}
for all nonempty subsets $\J$ of $\I$. In particular, $\bigvee_{i=0}^{n-1} \sigma^{-i}(\PP_\I) = \{[\iii] \colon \iii \in \I^n\}$. Since $\sigma$ is continuous, the set $\sigma^{-i}([I])$ is open and closed for all $I \in \I_\J$ and $i \ge 1$. Since $[\mathtt{I}]$, where $\mathtt{I} \in \mathcal{I}_\J^n$, is a finite intersection of such sets, we conclude that all the generalized cylinders are open and closed. Let $\mu \in \MM_\sigma(\I^\N)$ and observe that, by \eqref{eq:refinement-gen-cyl} and Lemma \ref{thm:entropy-basic-prop}\eqref{it:entropy1}, we have
\begin{equation} \label{eq:entropy-I-def}
\begin{split}
  h(\mu,\PP_\J) &= -\lim_{n \to \infty} \frac{1}{n} \sum_{\mathtt{I} \in \mathcal{I}_\J^n} \mu([\mathtt{I}])\log\mu([\mathtt{I}]) \\ 
  &= \inf_{n \ge 1} -\frac{1}{n} \sum_{\mathtt{I} \in \mathcal{I}_\J^n} \mu([\mathtt{I}])\log\mu([\mathtt{I}]) \in [0,\infty)
\end{split}
\end{equation}
for all nonempty finite subsets $\J$ of $\I$ and, in particular,
\begin{equation} \label{eq:entropy-i-def}
\begin{split}
  h(\mu,\PP_\I) &= -\lim_{n \to \infty} \frac{1}{n} \sum_{\mathtt{i} \in \mathcal{I}^n} \mu([\mathtt{i}])\log\mu([\mathtt{i}]) \\ 
  &= \inf_{n \ge 1} -\frac{1}{n} \sum_{\mathtt{i} \in \mathcal{I}^n} \mu([\mathtt{i}])\log\mu([\mathtt{i}]) \in [0,\infty].
\end{split}
\end{equation}
Note that if $\mu$ is a Bernoulli measure, then $h(\mu,\PP_\I) = H(\mu,\PP_\I)$.

\begin{lemma} \label{thm:entropy-usc}
  Let $\I$ be either finite or countably infinite set and $\J$ be a nonempty finite subset of $\I$. Then the map $\mu \mapsto h(\mu,\PP_\J)$ defined on $\MM_\sigma(\I^\N)$ is upper semicontinuous.
\end{lemma}

\begin{proof}
  Let $\mu \in \MM_\sigma(\I^\N)$ and $(\mu_k)_{k=1}^\infty$ be a sequence of measures in $\MM_\sigma(\I^\N)$ such that $\mu_k \to \mu$ in the weak$^*$ topology and fix $\eps>0$. By \eqref{eq:entropy-I-def}, let $n \in \N$ be such that
  \begin{equation} \label{eq:entropy-basic-prop1}
    -\frac{1}{n} \sum_{\mathtt{I} \in \mathcal{I}_\J^n} \mu([\mathtt{I}])\log\mu([\mathtt{I}]) \le h(\mu,\PP_\J) + \frac{\eps}{2}.
  \end{equation}
  Since the generalized cylinders $[\mathtt{I}]$ are open and closed, the weak$^*$ convergence of the sequence $(\mu_k)_{k=1}^\infty$ implies $\mu_k([\mathtt{I}]) \to \mu([\mathtt{I}])$ for all $\mathtt{I} \in \I_\J^n$. Hence, by \eqref{eq:entropy-I-def}, the continuity of $x \mapsto -x\log x$ on $[0,1]$, and \eqref{eq:entropy-basic-prop1}, there exists $k_0 \in \N$ such that
  \begin{align*}
    h(\mu_k,\PP_\J) &\le -\frac{1}{n} \sum_{\mathtt{I} \in \mathcal{I}_\J^n} \mu_k([\mathtt{I}])\log\mu_k([\mathtt{I}]) \\ 
    &\le -\frac{1}{n} \sum_{\mathtt{I} \in \mathcal{I}_\J^n} \mu([\mathtt{I}])\log\mu([\mathtt{I}]) + \frac{\eps}{2} \le h(\mu,\PP_\J) + \eps
  \end{align*}
  for all $k \ge k_0$. It follows that $\limsup_{k \to \infty} h(\mu_k,\PP_\J) \le h(\mu,\PP_\J)$ and the claim follows.
\end{proof}

The following example recalls \cite[Remark 3.11]{GodofredoToddVelozo2020} and shows that if $\I$ is countably infinite, then the map $\mu \mapsto h(\mu,\PP_\I)$ defined on $\MM_\sigma(\I^\N)$ is not upper semicontinuous.

\begin{example} \label{ex:entropy-not-usc}
  In this example, we exhibit a sequence $(\mu_k)_{k=1}^\infty$ of measures in $\MM_\sigma(\I^\N)$, where $\I$ is countably infinite, such that $\mu_k \to \mu \in \MM_\sigma(\I^\N)$ in the weak$^*$ topology and
  \begin{equation*}
    \limsup_{k \to \infty} h(\mu_k,\PP_\I) > h(\mu,\PP_\I).
  \end{equation*}
  Let $\mu_k$ be the Bernoulli measure obtained from the probability vector
  \begin{equation*}
    \Bigl( 1-\frac{1}{\log k},\frac{1}{k\log k},\frac{1}{k\log k},\ldots,\frac{1}{k\log k},0,0,\ldots \Bigr),
  \end{equation*}
  where the term $\frac{1}{k\log k}$ appears $k$ times. Note that $\mu_k \to \delta_1$ in the weak$^*$ topology, where $\delta_1$ is the Dirac mass at $111\cdots \in \I^\N$. As $\mu_k$ is Bernoulli, we have
  \begin{align*}
    h(\mu_k,\PP_\I) &= H(\mu_k,\PP_\I) = -\Bigl( 1-\frac{1}{\log k} \Bigr)\log\Bigl( 1-\frac{1}{\log k} \Bigr) - k\frac{1}{k\log k}\log\frac{1}{k\log k} \\ 
    &= -\Bigl( 1-\frac{1}{\log k} \Bigr)\log\Bigl( 1-\frac{1}{\log k} \Bigr) - \frac{1}{\log k}\log\frac{1}{\log k} + \frac{1}{\log k}\log k \to 1
  \end{align*}
  as $k \to \infty$. Since $h(\delta_1,\PP_\I) = H(\delta_1,\PP_\I) = 0$, we conclude that $\lim_{k \to \infty}h(\mu_k,\PP_\I)=1>0=h(\delta_1,\PP_\I)$ as wished.
\end{example}

The partitions $\PP_\J$ can be used to define the entropy. Recall that, by Lemma \ref{thm:entropy-basic-prop}\eqref{it:entropy1}, $h(\mu,\PP_\J) \in [0,\infty)$ for all $\mu \in \MM_\sigma(\I^\N)$ whenever $\J$ is a nonempty finite subset of $\I$.

\begin{lemma} \label{thm:entropy-via-cylinders}
  Let $\I$ be either finite or countably infinite set. For each $\mu \in \MM_\sigma(\I^\N)$ we have $h(\mu) = \sup\{h(\mu,\PP_\J) \colon \J \text{ is a finite subset of } \I\}$.
\end{lemma}

\begin{proof}
  Let $\mu \in \MM_\sigma(\I^\N)$. Since trivially $\sup\{h(\mu,\PP_\J) \colon \J$ is a finite subset of $\I\} \le h(\mu)$, it suffices to prove that $h(\mu) \le \sup\{h(\mu,\PP_\J) \colon \J$ is a finite subset of $\I\}$. Fix a finite Borel partition $\PP$ and $\eps > 0$. It is enough to show that there exists a finite subset $\J$ of $\I$ such that
  \begin{equation} \label{eq:entropy-via-cylinders1}
    h(\mu,\PP) \le h(\mu,\PP_\J) + \eps.
  \end{equation}
  Write $\PP = \{C_1,\ldots,C_p\}$ and choose $0<\delta<1/e$ such that $-p^2 \delta\log\delta \le \eps/2$ and, relying on uniform continuity,
  \begin{equation*}
    -\sum_{i=1}^p y_i \log y_i \le -\sum_{i=1}^p x_i\log x_i + \frac{\eps}{2}
  \end{equation*}
  for all $x_i,y_i \in [0,1]$ with $x_i-\delta \le y_i \le x_i+\delta$. Since $\I^\N$ is complete and separable and $\mu$ is a Borel probability measure on $\I^\N$, there are compact sets $K_1,\ldots,K_p \subset \I^\N$ such that $K_i \subset C_i$ and $\mu(C_i \setminus K_i) < \delta/(p+1)$ for all $i \in \{1,\ldots,p\}$. Let $\eta = \min_{i \ne j} \dist(K_i,K_j) > 0$ and choose $m \ge 1$ such that $\diam([\iii]) < \eta/2$ for all $\iii \in \I^m$. Since the union of $K_i$'s is compact, it can be covered by finitely many $m$-level cylinders. In other words, there exists a finite collection $\{\iii^k = i^k_1 \cdots i^k_m\}_{k=1}^\ell \subset \I^m$ such that $\bigcup_{i=1}^p K_i \subseteq \bigcup_{k=1}^\ell [\iii^k]$. We may assume that each $[\iii^k]$ intersects $\bigcup_{i=1}^p K_i$. The choice of $m \ge 1$ now guarantees that for every $k$ there exists unique $i$ such that $[\iii^k] \cap K_i \ne \emptyset$. Choose a finite subset $\J \subseteq \I$ so large that all of these $m$-level cylinders $[\iii^k]$ appear as generalized $m$-level cylinders, i.e.\ $\{i^k_1\} \cdots \{i^k_m\} \in \I_\J^m$ for all $k \in \{1,\ldots,\ell\}$. Define
  \begin{equation*}
    D_i = \bigcup_{[\iii^k] \cap K_i \ne \emptyset} [\iii^k]
  \end{equation*}
  for all $i \in \{1,\ldots,p-1\}$ and $D_p = \I^\N \setminus \bigcup_{i=1}^{p-1} D_i$. We see that $\mathcal{D} = \{D_1,\ldots,D_p\}$ is a partition of $\I^\N$ such that for every $i \in \{1,\ldots,p\}$ we have $K_i \subseteq D_i$ and $D_i$ is a union of elements in $\PP_\J^m = \bigvee_{i=0}^{m-1} \sigma^{-i}(\PP_\J)$, i.e.\ generalized cylinders $[\mathtt{I}]$ with $\mathtt{I} \in \I_\J^m$. Furthermore,
  \begin{equation} \label{eq:entropy-via-cylinders2}
  \begin{split}
    \mu(C_i \setminus D_i) + \mu(D_i \setminus C_i) &\le \mu(C_i \setminus K_i) + \mu(D_i \setminus K_i) \\ 
    &\le \frac{\delta}{p+1} + \mu\biggl( \I^\N \setminus \bigcup_{i=1}^p K_i \biggr) \le \delta.
  \end{split}
  \end{equation}
  By Lemma \ref{thm:entropy-basic-prop}\eqref{it:entropy2}, we have
  \begin{equation*}
    h(\mu,\PP) \le h(\mu,\mathcal{D}) + H(\mu,\PP\,|\,\mathcal{D}).
  \end{equation*}
  Therefore, in order to obtain \eqref{eq:entropy-via-cylinders1}, it suffices to show that $h(\mu,\mathcal{D}) \le h(\mu,\PP_\J)$ and $H(\mu,\PP\,|\,\mathcal{D}) \le \eps$. Recall that, by \eqref{eq:cond-entropy-as-difference}, $H(\mu,\PP\,|\,\mathcal{D}) = H(\mu,\PP \vee \mathcal{D}) - H(\mu,\mathcal{D})$. By \eqref{eq:entropy-via-cylinders2}, we have $\mu(D_i)-\delta \le \mu(C_i \cap D_i) \le \mu(D_i)+\delta$ for all $i$ and $\mu(C_i \cap D_j) \le \mu(C_i \setminus D_i) \le \delta$ whenever $i \ne j$. Therefore, by the choice of $\delta>0$, we have
  \begin{align*}
    H(\mu,\PP \vee \mathcal{D}) &= -\sum_{i,j=1}^p \mu(C_i \cap D_j) \log\mu(C_i \cap D_j) \\ 
    &= -\sum_{i=1}^p \mu(C_i \cap D_i) \log\mu(C_i \cap D_i) - \sum_{i \ne j} \mu(C_i \cap D_j) \log\mu(C_i \cap D_j) \\ 
    &\le H(\mu,\mathcal{D}) + \frac{\eps}{2} - p^2 \delta\log\delta \le H(\mu,\mathcal{D}) + \eps
  \end{align*}
  and hence, $H(\mu,\PP\,|\,\mathcal{D}) \le \eps$. Furthermore, by Lemma \ref{thm:entropy-basic-prop}\eqref{it:entropy2}, we have
  \begin{align*}
    h(\mu,\mathcal{D}) &\le h(\mu,\PP_\J^m) = \lim_{n \to \infty} \frac{1}{n} H\biggl( \mu,\bigvee_{i=0}^{n-1} \sigma^{-i}(\PP_\J^m) \biggr) \\ 
    &= \lim_{n \to \infty} \frac{n+m-1}{n} \frac{1}{n+m-1} H(\mu,\PP_\J^{n+m-1}) = h(\mu,\PP_\J)
  \end{align*}
  and the proof is finished.
\end{proof}

If the index set $\I$ is finite, then Lemma \ref{thm:entropy-via-cylinders}, \eqref{eq:entropy-I-def}, and Lemma \ref{thm:entropy-basic-prop}\eqref{it:entropy1} imply the well-known fact that
\begin{equation} \label{eq:entropy-finite-def}
  h(\mu) = h(\mu,\PP_\I).
\end{equation}

\begin{example} \label{rem:entropy-def-explanation}
  In this example, we demonstrate that in the countable infinite case we can have $h(\mu)<h(\mu,\PP_\I)$. Hence the entropy cannot in general be defined by using the finest partition given by the first level cylinders as in the finite case. Let $\I$ be countably infinite and $\mu \in \MM_\sigma(\I^\N)$. Write
  \begin{equation*}
    a_n = -\sum_{\iii \in \I^n} \mu([\iii])\log \mu([\iii])
  \end{equation*}
  for all $n \ge 1$ and recall that, by Lemma \ref{thm:entropy-basic-prop}\eqref{it:entropy1}, either all terms of the sequence $(a_n)_{n=1}^\infty$ are finite, or all are infinite. Note that, by Lemmas \ref{thm:entropy-basic-prop}\eqref{it:entropy2} and \ref{thm:entropy-via-cylinders}, we have
  \begin{equation} \label{eq:entropy-bad-def-ineq}
    h(\mu) \le h(\mu,\PP_\I).
  \end{equation}
  To illustrate that $h(\mu,\PP_\I)$ does not in general give us a good definition for the entropy in the countable case, we construct a measure $\mu \in \MM_\sigma(\I^\N)$ for which
  \begin{equation} \label{eq:infinity-entropy-bad-def}
    h(\mu) = 0 < \infty = h(\mu,\PP_\I).
  \end{equation}
  Let $\I = \N \setminus \{1\}$ and $\delta_i$ be the Dirac mass at $iii\cdots \in \I^\N$. Define
  \begin{equation*}
     \mu = c\sum_{i \in \I}\frac{\delta_i}{i(\log i)^2},
  \end{equation*} 
  where the constant $c$ is chosen such that $c\sum_{i \in \I} i^{-1}(\log i)^{-2} = 1$. Since $\mu$ is a linear combination of ergodic measures on the compact set $(\I \cup \{\infty\})^\N$ each of which has zero entropy, it follows from \cite[Theorem 8.1]{Walters1982} that $h(\mu)=0$. On the other hand,
  \begin{equation*}
    a_1 = -\sum_{i \in \I} \mu([i])\log\mu([i]) = c\sum_{i \in \I} \frac{\log i + 2\log\log i - \log c}{i(\log i)^2} = \infty.
  \end{equation*}
  Therefore, as all terms in the sequence $(a_n)_{n=1}^\infty$ are either finite or infinite, we have $a_n = \infty$ for all $n \in \N$ and \eqref{eq:infinity-entropy-bad-def} holds.
\end{example}

The following lemma shows that the strict inequality $h(\mu) < h(\mu,\PP_\I)$ is possible only when $h(\mu,\PP_\I) = \infty$. Under the assumption $h(\mu,\PP_\I) < \infty$ the entropy can equivalently be defined by using the finest partition given by the first level cylinders.

\begin{lemma} \label{thm:entropy-finite}
  Let $\I$ be either finite or countably infinite set and $\mu \in \MM_\sigma(\I^\N)$. If $h(\mu,\PP_\I) < \infty$, then
  \begin{equation*}
    h(\mu) = h(\mu,\PP_\I).
  \end{equation*}
\end{lemma}

\begin{proof}
  By Lemmas \ref{thm:entropy-basic-prop}\eqref{it:entropy2} and \ref{thm:entropy-via-cylinders}, we have $h(\mu) \le h(\mu,\PP_\I)$. Therefore, it suffices to prove that for every $\eps>0$ there exists a finite subset $\J$ of $\I$ such that $h(\mu,\PP_\I) \le h(\mu,\PP_\J) + \eps$. Fix $\eps>0$ and recall that, by Lemma \ref{thm:entropy-basic-prop}\eqref{it:entropy2},
  \begin{equation*}
    h(\mu,\PP_\I) \le h(\mu,\PP_\J) + H(\mu,\PP_\I\,\vert\,\PP_\J).
  \end{equation*}
  It is thus enough to show that $H(\mu,\PP_\I\,\vert\,\PP_\J) \le \eps$. Since $h(\mu,\PP_\I)<\infty$, Lemma \ref{thm:entropy-basic-prop}\eqref{it:entropy1} implies that $-\sum_{i \in \I} \mu([i])\log\mu([i]) < \infty$. Let $\J$ be a finite subset of $\I$ such that $-\sum_{i \in \I \setminus \J} \mu([i])\log\mu([i]) \le \eps$. Then
  \begin{align*}
    H(\mu,\PP_\I\,\vert\,\PP_\J) &= -\sum_{I \in \I_\J} \mu([I]) \sum_{i \in \I} \frac{\mu([I]\cap[i])}{\mu([I])} \log\frac{\mu([I]\cap[i])}{\mu([I])} \\ 
    &= -\mu([\I \setminus \J]) \sum_{i \in \I \setminus \J} \frac{\mu([i])}{\mu([\I \setminus \J])} \log\frac{\mu([i])}{\mu([\I \setminus \J])} \\ 
    &= -\sum_{i \in \I \setminus \J} \mu([i]) \log\mu([i]) + \sum_{i \in \I \setminus \J} \mu([i])\log\mu([\I \setminus \J]) \\ 
    &\le -\sum_{i \in \I \setminus \J} \mu([i]) \log\mu([i]) \le \eps
  \end{align*}
  as required.
\end{proof}

%
%

\section{Reduction to completely reducible matrices} \label{sec:completely-reducible}

In this section, we verify an important reduction according to which, to study the pressure and equilibrium states, it suffices to work with completely reducible matrices. This reduction serves as a basis in our analysis.

Let $V$ be a finite-dimensional real vector space and $\mathsf{A} \subseteq \GL(V)$ an arbitrary subset. We say that $\mathsf{A}$ is \emph{reducible} if there exists a proper nonzero subspace of $V$ which is preserved by every element of $\mathsf{A}$. When this is not the case we call $\mathsf{A}$ \emph{irreducible}. We furthermore say that $\mathsf{A}$ is \emph{completely reducible} if there exists a splitting $V=\bigoplus_{j=1}^\ell V_j$ such that $AV_j =V_j$ for all $A \in \mathsf{A}$ and $j\in\{1,\ldots,\ell\}$, and such that additionally $\{A|_{V_j} \colon A \in \mathsf{A}\}\subseteq \GL(V_j)$ is irreducible for every $j\in\{1,\ldots,\ell\}$. In other words, if $\A \subseteq \GL_d(\R)$, then completely reducibility means that the matrices in $\A$ are block-diagonal with irreducible blocks of the same size. Note that irreducibility implies complete reducibility since in this case we can take $\ell=1$. By a slight abuse of notation we will say that a tuple $(A_i)_{i \in \I} \in \GL(V)^\I$ is reducible, irreducible or completely irreducible if the corresponding set has the stated property.

The following technical result forms the first step in analysing a countably infinite affine iterated function system. It extends earlier work of the authors \cite[Theorem 5]{KaenmakiMorris2018} in the finite case, and also extends the antecedent result \cite[Proposition 1.4]{FengKaenmaki2011} which applies for finite affine iterated function systems in the parameter range $s \in (0,1] \cup [d-1,\infty)$. In effect it reduces the study of the pressure and equilibrium states to the case of tuples which are block diagonalisable with irreducible blocks.

\begin{theorem}\label{th:detriangularisation}
  Let $\I$ be either finite or countaby infinite and $\A=(A_i)_{i \in \I} \in \GL_d(\R)^\I$. Then there exist $X \in \GL_d(\R)$, $k \in \N$, and positive integers $d_1,\ldots,d_k$ such that we may write
  \[
    A_i = X^{-1}
    \begin{pmatrix}
      B_i^{11} &B_i^{12} & \cdots &B_i^{1k} \\
      0 & B_i^{22} & \cdots & B_i^{2k}\\
      \vdots & \vdots &\ddots &\vdots\\
      0 & 0& \cdots & B_i^{kk}
    \end{pmatrix}X
  \]
  for all $i \in \I$, where each matrix $B_i^{t_1t_2}$ is a real matrix with dimensions $d_{t_1} \times d_{t_2}$ and the family $(B_i^{tt})_{i \in \I}$ is irreducible for all $t \in \{1,\ldots,k\}$. If $\A'=(A_i')_{i \in \I}$ is defined by
  \[
    A_i' =
    \begin{pmatrix}
      B_i^{11} &0 & \cdots &0 \\
      0 & B_i^{22} & \cdots & 0\\
      \vdots & \vdots &\ddots &\vdots\\
      0 & 0& \cdots & B_i^{kk}
    \end{pmatrix}
  \]
  for all $i \in \I$, then $P(\A,s)=P(\A',s)$ for all $s \in \mathscr{I}_\A$ and in particular $\theta_{\A'} \leq \theta_{\A}$. Additionally, for every $s \in \mathscr{I}_\A$ the set of $\fii^s$-equilibrium states for $\A$ is identical to that of $\A'$.
\end{theorem}

Before going into the proof of Theorem \ref{th:detriangularisation}, let us study further properties of the singular value function. The following feature of $\varphi^s$ does not seem to have been previously noted (except in \cite[Lemma 6.1]{KaenmakiMorris2018} which covers the case $k=2$) and we believe it is original with this article:

\begin{proposition} \label{thm:not-previously-noted}
Let $d_1,\ldots,d_k \in \N$ and $d=\sum_{t=1}^k d_t$. If
\[
M_1=\begin{pmatrix} 
B^{11} & 0 & 0 & \cdots & 0\\
0& B^{22} & 0 & \cdots &0 \\
0&0& B^{33} & \cdots & 0 \\
\vdots &\vdots &\vdots & \ddots & \vdots \\
0&0&0& \cdots &B^{kk}
\end{pmatrix},
\qquad 
M_2=\begin{pmatrix} 
B^{11} & B^{12} & B^{13} & \cdots & B^{1k}\\
0& B^{22} & B^{23} & \cdots &B^{2k} \\
0&0& B^{33} & \cdots & B^{3k} \\
\vdots &\vdots &\vdots & \ddots & \vdots \\
0&0&0& \cdots &B^{kk}
\end{pmatrix},
\]
where $B^{tt} \in \mathrm{M}_{d_t}(\R)$ for every $t\in\{1,\ldots,k\}$ and $B^{t_1t_2} \in \mathrm{M}_{d_{t_1} \times d_{t_2}}(\R)$ for every $t_1,t_2$ such that $1 \leq t_1 < t_2 \leq k$, then 
\begin{equation*}
  \varphi^s(M_1) \leq \varphi^s(M_2)
\end{equation*}
for all $s \geq 0$.
\end{proposition}

\begin{proof}
We will first prove the proposition in the case where $s$ is one of the integers $1,\ldots,d$; the general case can then be easily deduced. For each $A \in \mathrm{M}_d(\R)$ let $\lambda_1(A),\ldots,\lambda_d(A)$ denote the absolute values of the eigenvalues of $A$, listed in non-increasing order. An inequality due to Weyl (see e.g. \cite[Theorem 3.3.2]{HornJohnson1991}) asserts that
\begin{equation}\label{eq:weyl}
  \prod_{i=1}^\ell \lambda_i(A) \leq \prod_{i=1}^\ell \sigma_i(A).
\end{equation}
for all $\ell\in\{1,\ldots,d\}$. For each $t\in\{1,\ldots,k\}$ let $B^{tt}=U_t D_t V_t^\top$ be a singular value decomposition of $B^{tt}$, where $D_t$ is a diagonal matrix with entries $\sigma_1(B^{tt}),\ldots,\sigma_{d_t}(B^{tt})$ and $U_t,V_t \in O(d_t)$. Define
\[
U=\begin{pmatrix} 
U_1 & 0 & 0 & \cdots & 0\\
0& U_2 & 0 & \cdots &0 \\
0&0&U_3 & \cdots & 0 \\
\vdots &\vdots &\vdots & \ddots & \vdots \\
0&0&0& \cdots &U_k
\end{pmatrix},
\qquad 
V=\begin{pmatrix} 
V_1 & 0 & 0 & \cdots & 0\\
0& V_2 & 0 & \cdots &0 \\
0&0&V_3 & \cdots & 0 \\
\vdots &\vdots &\vdots & \ddots & \vdots \\
0&0&0& \cdots &V_k
\end{pmatrix}
\]
and notice that $U,V \in O(d)$. We then have
\[
U^\top M_2V= \begin{pmatrix} 
D_1 & U_1^\top B^{12}V_2 & U_1^\top B^{13}V_3 & \cdots & U_1^\top B^{1k}V_k\\
0& D_2 & U_2^\top B^{23}V_3 & \cdots &U_2^\top B^{2k}V_k \\
0&0&D_3 & \cdots & U_3^\top B^{3k}V_k \\
\vdots &\vdots &\vdots & \ddots & \vdots \\
0&0&0& \cdots &D_k
\end{pmatrix}.
\]
The diagonal matrices of $U^\top M_2V$ are precisely the singular values of the matrices $B^{11},\ldots,B^{kk}$, which together form the singular values of $M_1$; but since $U^\top M_2V$ is upper triangular its diagonal entries are its eigenvalues, so the eigenvalues of $U^\top M_2 V$ are precisely the singular values of $M_1$. On the other hand, the singular values of $U^\top M_2 V$ are precisely the singular values of $M_2$ since the matrices $U$ and $V$ are orthogonal, and singular values are invariant with respect to pre- or post-multiplication by an orthogonal matrix. Thus
\[
\prod_{i=1}^\ell \sigma_i(M_1) = \prod_{i=1}^\ell \lambda_i(U^\top M_2 V) \leq \prod_{i=1}^\ell \sigma_i(U^\top M_2 V) =\prod_{i=1}^\ell \sigma_i(M_2)
\]
using \eqref{eq:weyl}. This yields $\varphi^\ell(M_1) \leq \varphi^\ell(M_2)$ for every integer $\ell\in\{1,\ldots,d\}$, and the case $\ell=0$ is also obvious. Now if $s=[0,d)$ let us write $s=\ell+\delta$ where $\ell=\lfloor s\rfloor$ and $\delta\in [0,1)$. The identity $\varphi^s(A)=\varphi^\ell(A)^{1-\delta} \varphi^{\ell+1}(A)^\delta$ for all $A \in \mathrm{M}_d(\R)$ follows easily from inspection of the definition of $\varphi^s$. We therefore have
\[
\varphi^s(M_1)=\varphi^\ell(M_1)^{1-\delta} \varphi^{\ell+1}(M_1)^\delta \leq \varphi^\ell(M_2)^{1-\delta} \varphi^{\ell+1}(M_2)^\delta = \varphi^s(M_2)
\]
as required. Since the case $s=d$ was already established this leaves only those cases where $s>d$, but in this case we clearly have
\[
\varphi^s(M_1)= \varphi^d(M_1)^{\frac{s}{d}} \leq \varphi^d(M_2)^{\frac{s}{d}} =\varphi^s(M_2).
\]
The proof is complete.
\end{proof}

We are now ready to prove Theorem \ref{th:detriangularisation}. We remark that the proof is related to the argument used to show \cite[Proposition 6.2]{KaenmakiMorris2018} but, from the theoretical point of view, is significantly simpler as it does not rely on measures.

\begin{proof}[Proof of Theorem \ref{th:detriangularisation}]
Since the inequality
\[
P((A_i')_{i \in \I},s) \leq P((A_i)_{i \in \I},s)
\]
for all $s \in \mathscr{I}_\A$ and therefore also the fact that $\theta_{\A'} \leq \theta_{\A}$ follow directly from Proposition \ref{thm:not-previously-noted}, it suffices to prove the inequality in the other way around. We first observe that if $A, X \in \GL_d(\R)$ and $s \geq 0$ then
\[
\varphi^s(X^{-1}AX) \leq \varphi^s(X^{-1}) \varphi^s(A) \varphi^s(X) \leq \|X\|^s \|X^{-1}\|^s \varphi^s(A)
\]
and by considering $X^{-1}AX$ in place of $A$
\[
\varphi^s(A) = \varphi^s(XX^{-1}AXX^{-1}) \leq \|X\|^s \|X^{-1}\|^s \varphi^s(X^{-1}AX).
\]
In particular, if $(A_i)_{i \in \I} \in \GL_d(\R)$, $X \in \GL_d(\R)$, and $s \geq 0$ are specified, then
\[
\biggl| \frac{1}{n} \log \sum_{\iii \in \I^n} \varphi^s(X^{-1}A_{\iii} X)  - \frac{1}{n} \log \sum_{\iii \in \I^n} \varphi^s(A_{\iii})  \biggr| \leq \frac{1}{n} \log (\|X\|^s \|X^{-1}\|^s)
\]
for every $n \geq 1$, and it follows that $P((A_i)_{i \in \I},s) = P((X^{-1}A_iX)_{i \in \I},s)$ for every $s \geq 0$. 

By \cite[Proposition 1.4]{FengKaenmaki2011}, we see that the tuple $\A$ is conjugated to a tuple of block upper-triangular matrices as the first displayed equation in the formulation claims. Furthermore, the above analysis shows that both tuples have the same pressure. Therefore, it suffices to consider two families $(A_i)_{i \in \I}$ and $(A_i')_{i \in \I}$, where 
\[
A_i=\begin{pmatrix} B_i^{11} &B_i^{12} & \cdots &B_i^{1k} \\
0 & B_i^{22} & \cdots & B_i^{2k}\\
\vdots & \vdots &\ddots &\vdots\\
0 & 0& \cdots & B_i^{kk}
\end{pmatrix},
\qquad
A_i'=\begin{pmatrix} B_i^{11} &0 & \cdots &0 \\
0 & B_i^{22} & \cdots & 0\\
\vdots & \vdots &\ddots &\vdots\\
0 & 0& \cdots & B_i^{kk}
\end{pmatrix}
\]
for all $i \in \I$. For each $\varepsilon \in (0,1]$ we define
\[
X_\varepsilon=\begin{pmatrix} \varepsilon I^1 &0 & \cdots &0 \\
0 & \varepsilon^2 I^2 & \cdots & 0\\
\vdots & \vdots &\ddots &\vdots\\
0 & 0& \cdots & \varepsilon^k I^k
\end{pmatrix},
\]
where $I^{t} \in \mathrm{M}_{d_t}(\R)$ is an identity matrix for every $t\in\{1,\ldots,k\}$. We have
\[
X_\varepsilon^{-1} A_iX_\varepsilon =\begin{pmatrix} B_i^{11} & \varepsilon B_i^{12} & \cdots & \varepsilon^{k-1} B_i^{1k} \\
0 & B_i^{22} & \cdots & \varepsilon^{k-2} B_i^{2k}\\
\vdots & \vdots &\ddots &\vdots\\
0 & 0& \cdots & B_i^{kk}
\end{pmatrix}
\]
for all $i \in \I$ and therefore, $\lim_{\varepsilon \downarrow 0} X_\varepsilon^{-1} A_\iii X_\varepsilon = A_\iii'$ for every $\iii \in \bigcup_{n \in \N} \I^n$. It follows that
\begin{align*}
  P((A_i')_{i \in \I},s) &= \inf_{n \geq 1}\liminf_{\varepsilon \downarrow 0}\frac{1}{n}\log \sum_{\iii \in \I^n} \varphi^s(X_\varepsilon^{-1} A_\iii X_\varepsilon) \\ 
  &\geq  \liminf_{\varepsilon \downarrow 0}\inf_{n \geq 1}\frac{1}{n}\log \sum_{\iii \in \I^n} \varphi^s(X_\varepsilon^{-1} A_\iii X_\varepsilon) \\ 
  &= \lim_{\varepsilon \downarrow 0} P((X_\varepsilon^{-1} A_i X_\varepsilon)_{i \in \I},s)  = P((A_i)_{i \in \I},s)
\end{align*}
for all $s \in \mathscr{I}_\A$.
\end{proof}

%
%

\section{Conditional proofs of the results} \label{sec:conditional-proof}

In this section, conditioned on the following technical result, we prove all the claims presented in Section \ref{sec:intro-results}. The proof of the following result is postponed until Section \ref{sec:algebra-stuff} in hoping to clarify the presentation as it is more algebraic in flavor. 

\begin{theorem} \label{thm:conc}
  Let $(A_i)_{i \in \I} \in \GL_d(\R)^\I$ be completely reducible where $\I$ is either finite or countably infinite. Then for each integer $k \in \{0,\ldots,d-1\}$ there exist an integer $p$ such that
  \begin{equation*}
    \begin{cases}
      1 \le p \le \binom{d}{k}, &\text{if } s=k, \\ 
      1 \le p \le \binom{d}{k}\binom{d}{k+1}, &\text{if } k < s \le k+1,
    \end{cases}
  \end{equation*}
  with functions $\Phi^{(1)}_{(\cdot)},\ldots,\Phi^{(p)}_{(\cdot)} \colon [k,k+1]\times \I^* \to (0,\infty)$, a constant $K>0$, and a finite set $F \subset \I^*$ such that the following three properties hold:
  \begin{enumerate}[(i)]
  \item\label{it:conc-1}
  For every $s \in [k,k+1]$ we have
  \[
    K^{-1}\varphi^s(A_\iii) \leq \max_{j \in \{1,\ldots,p\}} \Phi_s^{(j)}(\iii) \leq K\varphi^s(A_\iii)
  \]
  for all $\iii \in \I^*$.
  \item\label{it:conc-2}
  For every $s \in [k,k+1]$ and $j \in \{1,\ldots,p\}$ we have
  \[
    \Phi_s^{(j)}(\iii\jjj) \leq \Phi_s^{(j)}(\iii)\Phi_s^{(j)}(\jjj) \leq K\max_{\kkk \in F}  \Phi_s^{(j)}(\iii\kkk\jjj)
  \]
  for all $\iii,\jjj \in \I^*$. This property is called \emph{quasi-multiplicativity} of $\Phi_s^{(j)}$.
  \item\label{it:conc-3}
  For every $j \in \{1,\ldots,p\}$ and $\iii \in \I^*$ the function $s \mapsto \Phi_s^{(j)}(\iii)$ defined on $[k,k+1]$ is continuous. 
  \end{enumerate}
  If $s=k$ and $\A^{\wedge k}$ is irreducible then we may take $p=1$. We may do so also if $s \in (k,k+1]$ and both $\A^{\wedge k}$ and $\A^{\wedge(k+1)}$ are irreducible and at least one of them is strongly irreducible.
\end{theorem}

Without further mentioning, this section uses notation and the definitions introduced in Theorem \ref{thm:conc}. We remark that the upper bound $p \le \binom{d}{k}\binom{d}{k+1}$ is unlikely to be sharp. In \cite{KaenmakiMorris2018}, we conjectured that the natural lower bound $p \ge (d-k)\binom{d}{k}$ for $s \in (k,k+1)$ serves also as an upper bound.

\subsection{Behaviour of the pressure} \label{sec:conditional-pressure}
Recall that the clauses \eqref{it:finiteness} and \eqref{it:strictly} in Theorem \ref{th:aff} were verified already in Section \ref{sec:preli}. Conditioned on Theorem \ref{thm:conc}, we now prove the remaining clauses \eqref{it:right-con} and \eqref{it:sup}. Notice that, by Theorem \ref{th:detriangularisation}, it suffices to work with completely reducible families of matrices. We also present the proofs for Propositions \ref{thm:propaffinity} and \ref{thm:proppathologies}.

Let $\A=(A_i)_{i \in \I} \in \GL_d(\R)^\I$ where $\I$ is either finite or countably infinite, let $k \in \{0,\ldots,d-1\}$ and let $p \geq 1$ and $\Phi_{(\cdot)}^{(1)},\ldots,\Phi_{(\cdot)}^{(p)}$ be as in Theorem \ref{thm:conc}. For each $j \in \{1,\ldots,p\}$ and $s \in [k,k+1]$, we define
\begin{equation} \label{eq:def-block-pressure}
  P^{(j)}(\A,s)=\lim_{n \to \infty} \frac{1}{n}\log \sum_{\iii \in \I^n} \Phi_s^{(j)}(\iii).
\end{equation}
It follows from Theorem \ref{thm:conc}\eqref{it:conc-1} and Lemma \ref{th:first-lemma}\eqref{it:l11} that $\sum_{i \in \I} \Phi_s^{(j)}(i)<\infty$ for every $j \in \{1,\ldots,p\}$ and $s \in [k,k+1]\cap \mathscr{I}_\A$, and therefore by subadditivity
\begin{equation} \label{eq:inf-block-pressure}
  P^{(j)}(\A,s)=\inf_{n \geq 1} \frac{1}{n}\log \sum_{\iii \in \I^n} \Phi_s^{(j)}(\iii) \in[-\infty,\infty)
\end{equation}
for every $s \in [k,k+1]\cap \mathscr{I}_\A$, where subadditivity itself follows from Theorem \ref{thm:conc}\eqref{it:conc-2}. In particular the limit in \eqref{eq:def-block-pressure} exists. Now, it follows directly from Theorem \ref{thm:conc}\eqref{it:conc-1} that for every $s \in [k,k+1]\cap \mathscr{I}_\A$ we have
\[
  K^{-1}\sum_{\iii \in \I^n} \varphi^s(A_\iii) \leq \sum_{j=1}^p\sum_{\iii \in \I^n} \Phi_s^{(j)}(\iii) \leq p\max_{j \in \{1,\ldots,p\}} \sum_{\iii \in \I^n} \Phi_s^{(j)}(\iii) \leq pK\sum_{\iii \in \I^n} \varphi^s(A_\iii)
\]
for every $n \geq 1$ and consequently
\begin{equation} \label{eq:pressure-maximum}
  P(\A,s) = \max_{j \in \{1,\ldots,p\}} P^{(j)}(\A,s)
\end{equation}
for all $s \in [k,k+1]\cap \mathscr{I}_\A$.

Together with Lemmas \ref{le:first-cty} and \ref{le:cvx}, and recalling Theorem \ref{th:detriangularisation}, the following proposition proves Theorem \ref{th:aff}\eqref{it:right-con} conditioned on Theorem \ref{thm:conc}.

\begin{proposition} \label{thm:pressure-continuity}
  Let $\A=(A_i)_{i \in \I} \in \GL_d(\R)^\I$ be completely reducible, where $\I$ is countably infinite. Then the function $s\mapsto P(\A,s)$ is continuous on $\mathscr{I}_\A$. In particular, if $P(\A,\theta_\A)<\infty$, then
  \[
    \lim_{s \downarrow \theta_\A} P(\A,s)=P(\A,\theta_\A).
  \]
\end{proposition}

\begin{proof}
  Since the determinant is multiplicative, the continuity on $[d,\infty)$ is straightforward. Therefore, we may suppose without loss of generality that $\theta_\A<d$. It is sufficient to show that $s \mapsto P(\A,s)$ is continuous on $[k,k+1]\cap \mathscr{I}_\A$ for every integer $k$ such that $\lfloor\theta_\A\rfloor \leq k <d$, so we fix such an integer $k$ for the remainder of the proof and demonstrate the result in this form. Furthermore, by \eqref{eq:pressure-maximum}, it is enough to prove that the function $s \mapsto P^{(j)}(\A,s)$ defined on $[k,k+1]\cap \mathscr{I}_\A$ is continuous for every $j \in \{1,\ldots,p\}$. We will show in due course that this function does indeed take real values only, i.e. that it cannot take the value $-\infty$. We therefore fix such an integer $j$ for the remainder of the proof and demonstrate the continuity of $s \mapsto P^{(j)}(\A,s)$ on $[k,k+1]\cap \mathscr{I}_\A$.

  For every $n\geq 1$, by Lemma \ref{le:zeroth-lemma} the series $\sum_{\iii \in \I^n} \varphi^s(A_\iii)$ converges uniformly with respect to $s$ on closed subintervals of  $[k,k+1]\cap \mathscr{I}_\A$. In view of Theorem \ref{thm:conc}\eqref{it:conc-1} it follows that $\sum_{\iii \in \I^n} \Phi_s^{(j)}(\iii) $ also converges uniformly with respect to $s$ on closed subintervals of  $[k,k+1]\cap \mathscr{I}_\A$. Since by Theorem \ref{thm:conc}\eqref{it:conc-3} each function $s\mapsto \Phi_s^{(j)}(\iii)$ is continuous with respect to $s \in [k,k+1]$, this implies that each of the functions
  \[
    s \mapsto \sum_{\iii \in \I^n} \Phi_s^{(j)}(\iii)
  \]
  is continuous with respect to $s \in [k,k+1]\cap \mathscr{I}_\A$. We deduce in particular that $s \mapsto P^{(j)}(\A,s)$ is an upper semi-continuous function $[k,k+1]\cap \mathscr{I}_\A \to [-\infty,\infty)$, being the pointwise infimum of a sequence of continuous functions on that domain.

  We next apply quasi-multiplicativity of Theorem \ref{thm:conc}\eqref{it:conc-2} to show that it is also the pointwise supremum of a similar sequence of functions, which also serves to demonstrate that it takes real values only. Let $t\geq 1$ be a natural number such that $F\subseteq \bigcup_{k=1}^t \I^k$. Using Theorem \ref{thm:conc}\eqref{it:conc-2} for every $\iii,\jjj\in\I^*$ we have
  \[
    \Phi^{(j)}_s(\iii)\Phi^{(j)}_s(\jjj) \leq K \max_{\kkk \in F} \Phi^{(j)}_s(\iii\kkk\jjj)  \leq K \max_{1 \leq k \leq t} \max_{\kkk \in \I^k} \Phi^{(j)}_s(\iii\kkk\jjj) \leq  K\sum_{k=1}^t\sum_{\kkk \in \I^k} \Phi^{(j)}_s(\iii\kkk\jjj)
  \]
  and therefore for every $n,m \geq 1$
  \begin{align*}
    \biggl(\sum_{\iii \in \I^n} \Phi^{(j)}_s(\iii) \biggr)&\biggl(\sum_{\jjj \in \I^m} \Phi^{(j)}_s(\jjj) \biggr) \leq  K\sum_{k=1}^t\sum_{\kkk \in \I^k}\sum_{\iii \in \I^n} \sum_{\jjj \in \I^m} \Phi^{(j)}_s(\iii\kkk\jjj)\\
    &= K\sum_{k=1}^t \sum_{\lll \in \I^{n+m+k}}  \Phi^{(j)}_s(\lll)
    = K\sum_{k=1}^t\sum_{\kkk \in \I^k} \sum_{\lll \in \I^{n+m}}  \Phi^{(j)}_s(\kkk\lll) \\
    &\leq K\sum_{k=1}^t \biggl(\sum_{\kkk \in \I^k}\Phi^{(j)}_s(\kkk) \biggr)\biggl( \sum_{\lll \in \I^{n+m}} \Phi^{(j)}_s(\lll)\biggr)
    \leq \hat{K}(s) \sum_{\lll \in \I^{n+m}} \Phi^{(j)}_s(\lll),
  \end{align*}
  say, where
  \begin{equation} \label{eq:K-hat}
    \hat{K}(s)=K\sum_{k=1}^t \biggl(\sum_{\kkk \in \I^k}\Phi^{(j)}_s(\kkk) \biggr) \in (0,\infty)
  \end{equation}
  which depends continuously on $s \in [k,k+1]\cap \mathscr{I}_\A$ since we have already established that $s\mapsto \sum_{\kkk \in \I^k}\Phi^{(j)}_s(\kkk) $ is continuous for every natural number $k$ . This inequality is precisely what is required to demonstrate that the sequence
  \[
    \log \biggl(\hat{K}(s)^{-1}\sum_{\iii \in \I^n}\Phi^{(j)}_s(\iii) \biggr)
  \]
  is superadditive for each $s \in [k,k+1]\cap \mathscr{I}_\A$. Consequently 
  \begin{equation} \label{eq:j-pressure-sup}
    P^{(j)}(\A,s)= \sup_{n \geq 1} \frac{1}{n}\log\biggl(\hat{K}(s)^{-1} \sum_{\iii \in \I^n} \Phi_s^{(j)}(\iii)\biggr)
  \end{equation}
  for every  $s \in [k,k+1]\cap \mathscr{I}_\A$. Thus $s \mapsto P^{(j)}(\A,s)$ is also a lower semi-continuous function $[k,k+1]\cap \mathscr{I}_\A \to (-\infty,\infty)$ as it is a pointwise supremum of a sequence of continuous functions and we conclude that it depends continuously on $s$ and takes real values only. The result follows by recalling \eqref{eq:pressure-maximum} and the fact that the maximum of finitely many continuous functions is continuous.
\end{proof}

Recalling Theorem \ref{th:detriangularisation}, the following proposition proves Theorem \ref{th:aff}\eqref{it:sup} conditioned on Theorem \ref{thm:conc}.

\begin{proposition} \label{le:dense-inner-approx}
  Let $\A=(A_i)_{i \in \I} \in \GL_d(\R)^\I$ be completely reducible, where $\I$ is countably infinite. Then
  \[
    P(\A,s) = \sup\{ P((A_i)_{i \in \J} ,s) \colon \J\text{ is a nonempty finite subset of }\I\}
  \]
  for all $s \ge 0$.
\end{proposition}

\begin{proof}
  To simplify notation, we write $Q(\A,s) = \sup\{ P((A_i)_{i \in \J} ,s) \colon \J$ is a nonempty finite subset of $\I\} \in (-\infty,\infty]$ for all $s \ge 0$. Let $s \ge 0$ be such that $s \in [k,k+1)$ for some $k \in \{0,\ldots,d-1\}$. The case $s \ge d$ is relatively trivial and is left to the reader. Recalling \eqref{eq:pressure-maximum}, let $j \in \{1,\ldots,p\}$ be such that $P(\A,s) = P^{(j)}(\A,s)$. Let $(\J_N)_{N \ge 1}$ be an increasing sequence of nonempty finite subsets of $\I$ such that $\bigcup_{N \ge 1} \J_N = \I$ and $F \subseteq \J_1^*$, where $F \subset \I^*$ is as in Theorem \ref{thm:conc}. Recall that, by \eqref{eq:inf-block-pressure},
  \begin{equation*}
    P^{(j)}(\A,s) \le \frac{1}{n}\log\sum_{\iii \in \I^n} \Phi_s^{(j)}(\iii)
  \end{equation*}
  for all $n \ge 1$ and, by \eqref{eq:j-pressure-sup},
  \begin{equation} \label{eq:K-hat-j-estimate}
    \frac{1}{n}\log\biggl( \hat{K}(s)^{-1}\sum_{\iii \in \J_N^n}\Phi_s^{(j)}(\iii) \biggr) \le P^{(j)}((A_i)_{i \in \J_N},s)
  \end{equation}
  for all $n \ge 1$, where $\hat{K}(s) \in (0,\infty)$ is as in \eqref{eq:K-hat}. Fix $\eps > 0$ and observe that for every $n \ge 1$ we may choose $N(n) \ge 1$ such that
  \begin{equation*}
    P^{(j)}(\A,s) - \eps < \frac{1}{n}\log\sum_{\iii \in \J_{N(n)}^n} \Phi_s^{(j)}(\iii).
  \end{equation*}
  Therefore, it follows from \eqref{eq:K-hat-j-estimate} that
  \begin{equation*}
    P^{(j)}(\A,s) - \eps < P^{(j)}((A_i)_{i \in \J_{N(n)}},s) + \frac{1}{n}\log\hat{K}(s) \le Q(\A,s) + \frac{1}{n}\log\hat{K}(s).
  \end{equation*}
  By letting $n \to \infty$ and $\eps \downarrow 0$, this gives the claim.
\end{proof}

Assuming Theorem \ref{thm:conc}, we have now finished the proof of Theorem \ref{th:aff}. Let us next use Theorem \ref{th:aff} to show Proposition \ref{thm:propaffinity} which we repeat below.

\begin{propositionno}{\ref{thm:propaffinity}}
  \propaffinity
\end{propositionno}

\begin{proof}
To see that $0 \leq P(\A,\udimaff\A)<\infty$ implies $\ldimaff\A=\udimaff\A$ we argue as follows. Suppose $0 \leq P(\A,\udimaff\A)<\infty$ and define $\kappa=-\log \sup_{i \in \I} \threebar{A_i}>0$ as in Theorem \ref{th:aff}\eqref{it:strictly}. If $\J \subseteq \I$ is any nonempty finite set, applying Theorem \ref{th:aff}\eqref{it:strictly} to $(A_i)_{i \in \J}$ we find that
\begin{align*}
  P((A_i)_{i \in \J}, \udimaff \A) &\leq P((A_i)_{i \in \J}, \ldimaff \A) - \kappa(\udimaff \A - \ldimaff \A) \\ 
  &\leq - \kappa(\udimaff \A - \ldimaff \A),
\end{align*}
where the inequality  $P((A_i)_{i \in \J}, \ldimaff \A)\leq0$ follows from the definition of $\ldimaff \A$. Since by hypothesis $\udimaff \A \in \mathscr{I}_\A$, taking the supremum with respect to $\J$ and using Theorem \ref{th:aff}\eqref{it:sup} it follows that
\[
  0 \leq P(\A,\udimaff \A) \leq - \kappa(\udimaff \A - \ldimaff \A)
\]
and we deduce that $\udimaff \A \leq \ldimaff \A$ as required. If $\theta_\A<\udimaff\A$ then by Theorem \ref{th:aff}\eqref{it:right-con} together with the definition \eqref{eq:udimaff-def} of $\udimaff\A$ we must have $P(\A,\udimaff\A)=0$, and, by recalling Prosposition \ref{le:dense-inner-approx}, the same reasoning applies when $\A$ is completely reducible. Finally, if $\A$ is indexed over a finite set then the supremum \eqref{eq:ldimaff-def} is trivially attained by $\I$ itself and it follows that $\ldimaff \A = \udimaff \A$ whenever $\A$ is indexed over a finite set.
\end{proof}

Finally, let us prove Proposition \ref{thm:proppathologies} which is repeated below.

\begin{propositionno}{\ref{thm:proppathologies}}
  \proppathologies
\end{propositionno}

\begin{proof}
 Fix $\alpha, \beta \in (0,1)$ and $\gamma \in [\beta,1]$. Since $t \mapsto e^{\beta t}-1$ is a continuous and surjective function $(0,\infty) \to (0,\infty)$, there exists $t>0$ such that $e^{\beta t}-1 = \alpha^\beta$. Observe that then
  \begin{equation*}
    \sum_{k \in \N} (\alpha e^{-tk})^\beta = \sum_{k \in \N} e^{-\beta t(k-1)}-e^{-\beta tk} = 1.
  \end{equation*}
  In the case where we wish to have $P(\mathsf{A},\theta_{\mathsf{A}})=\infty$, let us define $a_k = k^{-\frac{1}{\gamma}}$ for all $k \in \N$ so that the series $\sum_{k \in \N} a_k^s = \sum_{k \in \N} k^{-\frac{s}{\gamma}}$ is finite if and only if $s > \gamma$. If on the other hand we wish to have $P(\mathsf{A},\theta_{\mathsf{A}}) < 0$ then we choose $a_k = k^{-\frac{1}{\gamma}}(\log(k+1))^{-\frac{2}{\gamma}}$ for all $k\in \N$ so that the sum $\sum_{k \in \N} a_k^s = \sum_{k \in \N} k^{-\frac{s}{\gamma}}(\log(k+1))^{-\frac{2s}{\gamma}}$ is finite if and only if $s \ge \gamma$. Finally, define
  \begin{equation*}
    A_k =
    \begin{pmatrix}
      \alpha e^{-tk} & \eps a_k \\
      0 & \alpha e^{-tk}
    \end{pmatrix}
  \end{equation*}
  for all $k \in \N$, where $\eps>0$ is chosen such that $\sup_{k \in \N} \|A_k\|<\alpha$. 

  Let us first show that $\ldimaff \A=\beta$. Define $\A' = (A_k')_{k \in \N}$, where
  \[
    A_k' =
    \begin{pmatrix}
      \alpha e^{-tk} & 0 \\
      0 & \alpha e^{-tk}
    \end{pmatrix}
  \]
  for all $k \in \N$. If $\J \subset \N$ is a nonempty finite set then, by Theorem \ref{th:detriangularisation}, we have $P((A_k)_{k \in \J},s)=P((A_k')_{k \in \J},s)$ for all $s \geq 0$, so it suffices to show that $\ldimaff (A_k')_{k \in \N}=\beta$. Since each $A_k'$ is conformal we simply have $P((A_k')_{k \in \J},s)=\log \sum_{k \in \J} (\alpha e^{-tk})^s$ for every finite set $\J \subset \N$. Since clearly $P((A_k')_{k \in \J},\beta) < 0 = \log\sum_{k \in \N} (\alpha e^{-tk})^\beta = P(\A',\beta)$ for all finite sets $\J \subset \N$ by the choice of $t>0$, we conclude that $\ldimaff (A_k)_{k \in \N} \leq \beta$. On the other hand, for every $s \in (0,\beta)$ we have $\sum_{k \in \N} (\alpha e^{-tk})^s>1$ and this implies the existence of a finite set $\J \subset \N$ such that $\sum_{k \in \J} (\alpha e^{-tk})^s>1$, and such a set necessarily has $\dimaff (A_k)_{k \in \J}>s$. It follows that $\ldimaff (A_k)_{k \in \N}>s$ for every $s \in (0,\beta)$ and this completes the proof.

  Let us next demonstrate that $\gamma=\theta_\A$ and that $P(\A,\theta_\A)$ is either finite or infinite. By Theorem \ref{th:aff}(\ref{it:finiteness}), it is sufficient to show that the series $\sum_{k\in \N} \varphi^s(A_k)=\sum_{k \in \N}\|A_k\|^s$ is infinite for all $s \in [0,\gamma)$, finite for all $s \in (\gamma,2]$, and either finite or infinite at $s=\gamma$ as appropriate. For all $s \in [0,2]$, we have the estimate
  \begin{align*}
    \varepsilon^s \sum_{k \in \N} a_k^s &\leq \sum_{k \in \N} \|A_k\|^s \leq   \sum_{k \in \N} (2\alpha e^{-t k} + \varepsilon a_k)^s \le \sum_{k \in \N} 2^{s-1}(2^s \alpha^s e^{-s tk} + \varepsilon^s a_k^s) \\
    &= \sum_{k \in \N} 2^{2s-1} \alpha^s e^{-s tk} + 2^{s-1}\varepsilon^s \sum_{k \in \N} a_k^s = \frac{2^{2s-1} \alpha^s}{e^{s t}-1} + 2^{s-1}\varepsilon^s \sum_{k \in \N} a_k^s,
  \end{align*}
  where we have used the fact that $\|A\|$ is bounded above by the total of the absolute values of the entries of $A$ and have also used H\"older's inequality which gives $(x+y)^s \leq 2^{s-1}(x^s + y^s)$ for all $x,y \geq 0$. It follows in particular that $\sum_{k\in \N} \|A_k\|^s$ is finite if and only if $\sum_{k \in \N} a_k^s$ is finite. The claim follows by the choice of the sequence $(a_k)_{k \in \N}$.

  Finally, let us verify that $\theta_\A = \udimaff\A$. By \eqref{eq:dimaff-ineq}, we have $\theta_\A \le \udimaff\A$. If it was $\theta_\A < \udimaff\A$, then, by Proposition \ref{thm:propaffinity}\eqref{it:propaff2}, we would have $\ldimaff\A=\udimaff\A$. This is a contradiction as $\A$ is chosen such that $\ldimaff\A = \beta < \theta_\A$. Furthermore, since now we have $\ldimaff\A < \udimaff\A$, Proposition \ref{thm:propaffinity}\eqref{it:propaff1} implies that $P(\A,\theta_\A)<0$ or $P(\A,\theta_\A)=\infty$.
\end{proof}

\begin{remark} \label{rem:assumption-needed}
  Let us demonstrate that the assumption $P(\A,\theta_\A)<\infty$ in Theorem \ref{th:aff}\eqref{it:right-con} is required for the right-continuity of the pressure at $\theta_\A$. We begin by choosing in Proposition \ref{thm:proppathologies} the tuple $\A$ of upper-triangular matrices such that $P(\A,\theta_\A)=\infty$ and $\beta = \theta_\A$. If $\A'$ is the corresponding tuple of diagonal matrices, then the choice of $t > 0$ in the proof of Proposition \ref{thm:proppathologies} implies $P(\A',\theta_\A)=0$. As $\A'$ is completely reducible, it follows from Theorem \ref{th:detriangularisation} that $P(\A,s)=P(\A',s)$ for all $s \in (\theta_\A,\infty)$ and, in particular,
  \begin{equation*}
    \lim_{s \downarrow \theta_\A} P(\A,s) = \lim_{s \downarrow \theta_\A} P(\A',s) = P(\A',\theta_\A) < P(\A,\theta_\A)
  \end{equation*}
  as wished.
\end{remark}

%
%

\subsection{Description of the equilibrium states} \label{sec:conditional-equilibrium}
Let us next verify Proposition \ref{thm:equilibriuminequality} and use Theorem \ref{thm:conc} to prove Theorem \ref{thm:equilibrium-states-classification}. We also present the proof for Proposition \ref{thm:proppressuredrop} which in fact is just a simple application of Theorem \ref{thm:equilibrium-states-classification}.

\begin{propositionno}{\ref{thm:equilibriuminequality}}
  \equilibriuminequality
\end{propositionno}

\begin{proof}
  Let us first assume that $h(\mu)<\infty$. If $\Lambda(\mu,\A,s)=-\infty$ or $P(\A,s)=\infty$, then there is nothing to prove. We may therefore assume that $\Lambda(\mu,\A,s)>-\infty$ and $P(\A,s)<\infty$. Let $\J$ be a nonempty finite subset of $\I$ and write $\PP_\J^n = \bigvee_{i=0}^{n-1} \sigma^{-i}(\PP_\J)$. Note that $h(\mu,\PP_\J)<\infty$. Fix $n \ge 1$ and for each $C \in \PP_\J^n$ choose $\jjj_C \in \I^n$ such that
  \begin{equation} \label{eq:equilibrium-estimate1}
    \int_C \log\fii^s(A_{\iii|_n}) \dd\mu(\iii) \le \mu(C) \log\fii^s(A_{\jjj_C}).
  \end{equation}
  Recall that the elements of $\PP_\J^n = \{[\mathtt{I}] \colon \mathtt{I} \in \mathcal{I}_\J^n\}$ are unions of the elements in $\{[\iii] \colon \iii \in \I^n\}$. Therefore, we may choose $\jjj_C \in \I^n$ such that $[\jjj_C] \subset C$ and the choice $C \mapsto \jjj_C$ is necessarily injective. By applying Lemma \ref{thm:entropy-basic-prop}\eqref{it:entropy1}, Lemma \ref{thm:energy-basic-prop}, \eqref{eq:equilibrium-estimate1}, and Jensen's inequality on logarithm, we get
  \begin{align*}
    h(\mu,\PP_\J) + \Lambda(\mu,\A,s) &\le \frac{1}{n}\biggl( H(\mu,\PP_\J^n) + \int_{\I^\N} \log\fii^s(A_{\iii|_n}) \dd\mu(\iii) \biggr) \\ 
    &\le \frac{1}{n} \sum_{C \in \PP_\J^n} \mu(C)\log\frac{\fii^s(A_{\jjj_C})}{\mu(C)} \\ 
    &\le \frac{1}{n} \log\sum_{C \in \PP_\J^n} \fii^s(A_{\jjj_C}) 
    \le \frac{1}{n} \log\sum_{\iii \in \I^n} \fii^s(A_\iii).
  \end{align*}
  The proof follows by letting $n \to \infty$, noticing that the choice of the finite set $\J \subseteq \I$ is free, and recalling the definition of the pressure from \eqref{eq:pressure-defn} and the characterisation of the entropy in Lemma \ref{thm:entropy-via-cylinders}.

  Let us then assume that $\Lambda(\mu,\A,s)>-\infty$. If $h(\mu)<\infty$, then we are in the situation we already have covered. We may thus assume that $h(\mu)=\infty$. Then for every $M>0$ there is a nonempty finite subset $\J$ of $\I$ such that $M \le h(\mu,\PP_\J) < \infty$ by Lemma \ref{thm:entropy-basic-prop}\eqref{it:entropy1}. But now we are again in the situation we have studied. It follows from the first part of the proof that
  \begin{equation*}
    M + \Lambda(\mu,\A,s) \le h(\mu,\PP_\J) + \Lambda(\mu,\A,s) \le \frac{1}{n} \log\sum_{\iii \in \I^n} \fii^s(A_\iii).
  \end{equation*}
  By letting $n \to \infty$ and then $M \to \infty$, we see that $P(\A,s)=\infty$ and the proof is complete.
\end{proof}

Let $\A=(A_i)_{i \in \I} \in \GL_d(\R)^\I$ be such that $\sup_{i \in \I} \fii^s(A_i) < \infty$, where $\I$ is either finite or countably infinite, and $\mu \in \MM_\sigma(\I^\N)$. For each $k \in \{1,\ldots,d-1\}$, $s \in [k,k+1]$, and $j \in \{1,\ldots,p\}$, we define
\begin{equation*}
  \Lambda^{(j)}(\mu,\A,s) = \lim_{n \to \infty} \frac{1}{n} \int_{\I^\N} \log\Phi^{(j)}_s(\iii|_n) \dd\mu(\iii) \in [-\infty,\log K\sup_{i \in \I} \fii^s(A_i)]
\end{equation*}
and notice that, by recalling Theorem \ref{thm:conc}\eqref{it:conc-1}--\eqref{it:conc-2} and Lemma \ref{thm:energy-basic-prop}, these quantities are well-defined and are equal to the infima of the same sequence over $n$. Furthermore, if $\mu$ is ergodic, then, by Theorem \ref{thm:conc}\eqref{it:conc-1} together with the subadditive ergodic theorem, we have
\begin{equation} \label{eq:measure-pressure-est}
\begin{split}
  \max_{j \in \{1,\ldots,p\}} \Lambda^{(j)}(\mu,\A,s) &= \max_{j \in \{1,\ldots,p\}} \lim_{n \to \infty} \frac{1}{n} \log\Phi^{(j)}_s(\iii|_n) \\ 
  &=\lim_{n \to \infty}\frac{1}{n} \log\varphi^s(\iii|_n) = \Lambda(\mu,\A,s)
\end{split}
\end{equation}
where the limit is taken almost everywhere with respect to $\mu$. We also define for each $s \in [k,k+1] \cap \mathscr{I}_\A$ and $j \in \{1,\ldots,p\}$ the \emph{measure-theoretical pressure} of $\A$ at $s$ with respect to $\mu$ by setting
\begin{equation*}
  P^{(j)}(\mu,\A,s) = \lim_{n \to \infty} \frac{1}{n} \sum_{\iii \in \I^n} \mu([\iii]) \log\frac{\Phi^{(j)}_s(\iii)}{\mu([\iii])}.
\end{equation*}
The main advantage in using the measure-theoretical pressure is that it can have a finite value even if $h(\mu,\PP_\I)=\infty$ and $\Lambda^{(j)}(\mu,\A,s)=-\infty$. The following lemma verifies the existence of the limit $P^{(j)}(\mu,\A,s)$.

\begin{lemma} \label{thm:measure-theoretical-pressure-prop}
  Let $\A=(A_i)_{i \in \I} \in \GL_d(\R)^\I$, where $\I$ is either finite or countably infinite. If $s \in \mathscr{I}_\A$, then the following two assertions hold:
  \begin{enumerate}[(i)]
    \item\label{it:m-pressure1} For each
    $\mu \in \MM_\sigma(\I^\N)$ the limit
    \begin{equation*}
      P^{(j)}(\mu,\A,s) = \lim_{n \to \infty} \frac{1}{n} \sum_{\iii \in \I^n} \mu([\iii]) \log\frac{\Phi^{(j)}_s(\iii)}{\mu([\iii])}.
    \end{equation*}
    exists in $[-\infty,P^{(j)}(\A,s)]$ and is equal to $\inf_{n \ge 1} \frac{1}{n} \sum_{\iii \in \I^n} \mu([\iii]) \log\frac{\Phi^{(j)}_s(\iii)}{\mu([\iii])}$. In particular, the map $\mu \mapsto P^{(j)}(\mu,\A,s)$ defined on $\MM_\sigma(\I^\N)$ is upper semicontinuous.
    \item\label{it:m-pressure2} If $\mu \in \MM_\sigma(\I^\N)$ is such that $h(\mu,\PP_\I) < \infty$ or $\Lambda^{(j)}(\mu,\A,s)>-\infty$, then
    \begin{equation*}
      P^{(j)}(\mu,\A,s) = h(\mu) + \Lambda^{(j)}(\mu,\A,s).
    \end{equation*} 
  \end{enumerate}
\end{lemma}

\begin{proof}
  Fix $j \in \{1,\ldots,p\}$ and write
  \begin{equation*}
    a_n = \sum_{\iii \in \I^n} \mu([\iii]) \log\frac{\Phi^{(j)}_s(\iii)}{\mu([\iii])}
  \end{equation*}
  for all $n \ge 1$. Since $s \in \mathscr{I}_\A$, the submultiplicativity of $\Phi_s^{(j)}$ given by Theorem \ref{thm:conc}\eqref{it:conc-2} together with Theorem \ref{thm:conc}\eqref{it:conc-1} and Lemma \ref{th:first-lemma}\eqref{it:l11} show that
  \begin{equation} \label{eq:phi-sum-finite}
    \sum_{\iii \in \I^n} \Phi_s^{(j)}(\iii) \le \biggl(\sum_{i \in \I} \Phi_s^{(j)}(i)\biggr)^n \le K^n\biggl(\sum_{i \in \I} \fii^s(A_i)\biggr)^n < \infty
  \end{equation}
  for all $n \ge 1$. Therefore, by Jensen's inequality on the concave function $x \mapsto -x\log x$, we have
  \begin{align*}
    \sum_{\iii \in \I^n} \mu([\iii])\biggl( &\log\frac{\Phi_s^{(j)}(\iii)}{\mu([\iii])} - \log\sum_{\jjj \in \I^n} \Phi_s^{(j)}(\jjj) \biggr) \\ 
    &= \sum_{\iii \in \I^n} \frac{\Phi_s^{(j)}(\iii)}{\sum_{\jjj \in \I^n} \Phi_s^{(j)}(\jjj)} \biggl( -\frac{\mu([\iii])\sum_{\jjj \in \I^n} \Phi_s^{(j)}(\jjj)}{\Phi_s^{(j)}(\iii)} \log\frac{\mu([\iii])\sum_{\jjj \in \I^n} \Phi_s^{(j)}(\jjj)}{\Phi_s^{(j)}(\iii)} \biggr) \\ 
    &\le -\biggl( \sum_{\iii \in \I^n} \mu([\iii]) \biggr) \log\biggl( \sum_{\iii \in \I^n} \mu([\iii]) \biggr) = 0
  \end{align*}
  and hence, $a_n \le \log\sum_{\iii \in \I^n} \Phi_s^{(j)}(\iii) < \infty$ for all $n \ge 1$. In particular, dividing by $n$ before letting $n \to \infty$ shows that $P^{(j)}(\mu,\A,s) \le P^{(j)}(\A,s)$ provided that the limit $P^{(j)}(\mu,\A,s)$ exists. To show that $P^{(j)}(\mu,\A,s)$ exists, observe that, by the submultiplicativity of $\Phi_s^{(j)}$, Jensen's inequality on the concave function $x \mapsto -x\log x$ and the $\sigma$-invariance of $\mu$ imply
  \begin{equation} \label{eq:measure-pressure-subadd}
  \begin{split}
    a_{m+n} - a_m &\le \sum_{\iii \in \I^m} \sum_{\jjj \in \I^n} \mu([\iii\jjj])\log\frac{\Phi_s^{(j)}(\iii)\Phi_s^{(j)}(\jjj)}{\mu([\iii\jjj])} - \sum_{\iii \in \I^m} \mu([\iii])\log\frac{\Phi_s^{(j)}(\iii)}{\mu([\iii])} \\ 
    &= \sum_{\jjj \in \I^n} \Phi_s^{(j)}(\jjj) \sum_{\iii \in \I^m} \mu([\iii])\biggl( -\frac{\mu([\iii\jjj])}{\mu([\iii])\Phi_s^{(j)}(\jjj)}\log\frac{\mu([\iii\jjj])}{\mu([\iii])\Phi_s^{(j)}(\jjj)} \biggr) \\ 
    &\le \sum_{\jjj \in \I^n} \Phi_s^{(j)}(\jjj) \biggl( -\sum_{\iii \in \I^m} \frac{\mu([\iii\jjj])}{\Phi_s^{(j)}(\jjj)} \log\sum_{\iii \in \I^m} \frac{\mu([\iii\jjj])}{\Phi_s^{(j)}(\jjj)} \biggr) = a_n
  \end{split}
  \end{equation}
  for all $m,n \ge 1$. It follows that if $a_n = -\infty$ for some $n \ge 1$, then $a_m = -\infty$ for all $m \ge n$. Hence, the limit $P^{(j)}(\mu,\A,s)$ exists in $[-\infty,P^{(j)}(\A,s)]$ and is equal to the infimum of the same sequence over $n$ by the subadditivity \eqref{eq:measure-pressure-subadd} as claimed in \eqref{it:m-pressure1}. Since the map $\mu \mapsto \sum_{\iii \in \I^n} \mu([\iii]) \log\frac{\Phi^{(j)}_s(\iii)}{\mu([\iii])}$ is continuous for all $n \in \N$, we see that $\mu \mapsto P^{(j)}(\mu,\A,s)$ is an infimum of continuous functions and the claim \eqref{it:m-pressure1} follows.

  To show the assertion \eqref{it:m-pressure2}, notice first that the assumption $s \in \mathscr{I}_\A$ necessarily implies $\sup_{i \in \I} \fii^s(A_i) < \infty$ since $\sum_{i \in \I} \fii^s(A_i) < \infty$ by \eqref{eq:phi-sum-finite}. Therefore the energy $\Lambda^{(j)}(\mu,\A,s)$ defined in \eqref{eq:energy} exists in $[-\infty,\log K\sup_{i \in \I} \fii^s(A_i)]$ by Lemma \ref{thm:energy-basic-prop}\eqref{it:energy1}. Let us first assume that $\Lambda^{(j)}(\mu,\A,s)>-\infty$ in which case
  \begin{equation*}
    -\infty < \Lambda^{(j)}(\mu,\A,s) \le \frac{1}{n}\int_{\I^\N} \log\Phi_s^{(j)}(\iii|_n) \dd\mu(\iii) = \frac{1}{n}\sum_{\iii \in \I^n}\mu([\iii])\log\Phi_s^{(j)}(\iii)
  \end{equation*}
  for all $n \ge 1$ again by Lemma \ref{thm:energy-basic-prop}\eqref{it:energy1}. By the assertion \eqref{it:m-pressure1}, there exists $n_0 \ge 1$ such that
  \begin{equation*}
    \frac{1}{n}\sum_{\iii \in \I^n} \mu([\iii])\log\frac{\Phi_s^{(j)}(\iii)}{\mu([\iii])} \le P^{(j)}(\mu,\A,s)+1 < \infty
  \end{equation*}
  for all $n \ge n_0$. Thus,
  \begin{equation*}
    h(\mu,\PP_\I) \le P^{(j)}(\mu,\A,s)+1-\Lambda^{(j)}(\mu,\A,s) < \infty
  \end{equation*}
  and Lemma \ref{thm:entropy-finite} shows that $P^{(j)}(\mu,\A,s) = h(\mu)+\Lambda^{(j)}(\mu,\A,s)$.

  Let us then assume that $h(\mu,\PP_\I) < \infty$ in which case, by Lemma \ref{thm:entropy-finite}, there exists $n_0 \ge 1$ such that
  \begin{equation*}
    \frac{1}{n} \sum_{\iii \in \I^n} \mu([\iii]) \log\mu([\iii]) \le h(\mu)+1 < \infty
  \end{equation*}
  for all $n \ge n_0$. If $\Lambda^{(j)}(\mu,\A,s)>-\infty$, then we are in the situation already covered above. Therefore, it suffices to show that $\Lambda^{(j)}(\mu,\A,s) = -\infty$ implies $P^{(j)}(\mu,\A,s) = -\infty$. To that end, suppose that $P^{(j)}(\mu,\A,s) > -\infty$ in which case
  \begin{equation*}
    -\infty < P^{(j)}(\mu,\A,s) \le \frac{1}{n}\sum_{\iii \in \I^n} \mu([\iii])\log\frac{\Phi_s^{(j)}(\iii)}{\mu([\iii])}
  \end{equation*}
  for all $n \ge 1$ by the assertion \eqref{it:m-pressure1}. It is now evident that $\Lambda^{(j)}(\mu,\A,s) = P^{(j)}(\mu,\A,s)-h(\mu,\PP_\I) > -\infty$. The proof of \eqref{it:m-pressure2} is now finished by recalling Lemma \ref{thm:entropy-finite}.
\end{proof}

Let us now turn to prove Theorem \ref{thm:equilibrium-states-classification}. As the assertion in Theorem \ref{thm:equilibrium-states-classification}\eqref{it:eq1} can be treated by existing methods, we will cover it in the next lemma.

\begin{lemma}
  Let $\A=(A_i)_{i \in \I} \in \GL_d(\R)^\I$ be such that $\sup_{i \in \I} \|A_i\| < \infty$, where $\I$ is either finite or countably infinite. If $s > \theta_\A$ and $s \ge d$, then there is a unique $\fii^s$-equilibrium state for $\A$ and it is a Bernoulli measure.
\end{lemma}

\begin{proof}
  Since the singular value function in this regime is a power of a determinant and the determinant is multiplicative, the result follows from \cite[\S 3]{MauldinUrbanski1996} and \cite[\S 2]{Urbanski1998}.
\end{proof}

Conditioned on Theorem \ref{thm:conc}, the following proposition proves all the remaining claims in Theorem \ref{thm:equilibrium-states-classification} which do not assume irreducibility. The proof in the countably infinite case is more complicated than in the finite case as we cannot rely on the upper semicontinuity of the entropy. Showing even the existence of an $\fii^s$-equilibrium state requires more delicate approach. The idea in the proof is to use the fact that for each $j \in \{1,\ldots,p\}$ the quasi-multiplicativity of $\Phi^{(j)}_s$ implies the existence of a unique $\mu^{(j)} \in \MM_\sigma(\I^\N)$ for which $P^{(j)}(\mu^{(j)},\A,s) = P^{(j)}(\A,s)$. Then, by proving $h(\mu^{(j)},\PP_\I)<\infty$ and considering only pressure maximising indices $j \in \{1,\ldots,p\}$, Lemmas \ref{thm:measure-theoretical-pressure-prop}\eqref{it:m-pressure2} and \ref{thm:entropy-finite} show that $\mu^{(j)}$ is an $\fii^s$-equilibrium state.

\begin{proposition}
  Let $\A=(A_i)_{i \in \I} \in \GL_d(\R)^\I$ be such that $\sup_{i \in \I} \|A_i\| < \infty$, where $\I$ is either finite or countably infinite, and $p\geq 1$ as in Theorem \ref{thm:conc}. If $s \in \mathscr{I}_\A$, then
  \begin{equation*}
    P(\A,s) = \sup\{h(\mu) + \Lambda(\mu,\A,s) \colon \mu\in \mathcal{M}_\sigma(\I^\N) \text{ is such that } h(\mu)<\infty\}.
  \end{equation*}
  Furthermore, if $s > \theta_\A$, then the number of distinct ergodic $\fii^s$-equilibrium states $\mu$ for $\A$ is at least one and is not more than $p$. Finally, all the equilibrium states are fully supported on $\I^\N$.
\end{proposition}

\begin{proof}
  Recalling Theorem \ref{th:detriangularisation}, we may assume that $\A=(A_i)_{i \in \I} \in \GL_d(\R)^\I$ is completely reducible. Fix $s \in \mathscr{I}_\A$ and $j \in \{1,\ldots,p\}$, let $F \subset \I^*$ be as in Theorem \ref{thm:conc}, and notice that, by Theorem \ref{thm:conc}\eqref{it:conc-2}, the function $\Phi^{(j)}_s \colon \I^* \to (0,\infty)$ is quasimultiplicative. Note that, by \eqref{eq:j-pressure-sup} and \eqref{eq:pressure-maximum}, we have $-\infty < P^{(j)}((A_i)_{i \in \J},s) \le P(\A,s) < \infty$. It follows from \cite[Proposition 3.4]{KaenmakiReeve2014} that there exists a constant $C \ge 1$ such that for every nonempty finite subset $\J$ of $\I$ with $F \subset \J^*$ there is a measure $\mu_\J^{(j)} \in \MM_\sigma(\I^\N)$ supported on $\J^\N$ for which
  \begin{equation} \label{eq:quasimulti-equilibrium-state-finite}
    C^{-1} \le \frac{\mu_\J^{(j)}([\iii])}{\exp(-nP^{(j)}((A_i)_{i \in \J},s)) \Phi^{(j)}_s(\iii)} \le C
  \end{equation}
  for all $\iii \in \J^n$. Note that $h(\mu^{(j)}_\J)=h(\mu^{(j)}_\J,\PP_\J) < \infty$ by \eqref{eq:entropy-finite-def} and Lemma \ref{thm:entropy-basic-prop}\eqref{it:entropy1}. Therefore, \eqref{eq:quasimulti-equilibrium-state-finite} gives
  \begin{equation} \label{eq:finite-eq-state}
    \Lambda^{(j)}(\mu^{(j)}_\J,\A,s) = P^{(j)}((A_i)_{i \in \J},s) - h(\mu^{(j)}_\J) > -\infty.
  \end{equation}
  Therefore, by \eqref{eq:pressure-maximum} and Proposition \ref{le:dense-inner-approx},
  \begin{align*}
    \sup_{\atop{\J \subseteq \I \text{ is finite}}{k \in \{1,\ldots,p\}}} \{h(\mu^{(k)}_\J) + \Lambda^{(k)}(\mu^{(k)}_\J,\A,s)\} &= \sup_{\J \subseteq \I \text{ is finite}} \max_{k \in \{1,\ldots,p\}} P^{(k)}((A_i)_{i \in \J},s) \\ 
    &= \sup_{\J \subseteq \I \text{ is finite}} P((A_i)_{i \in \J},s) = P(\A,s).
  \end{align*}
  Since, by Theorem \ref{thm:conc}\eqref{it:conc-1} and Proposition \ref{thm:equilibriuminequality},
  \begin{equation*}
    h(\mu^{(j)}_\J) + \Lambda^{(j)}(\mu^{(j)}_\J,\A,s) \le h(\mu^{(j)}_\J) + \Lambda(\mu^{(j)}_\J,\A,s) \le P(\A,s),
  \end{equation*}
  we have shown that
  \begin{equation*}
    P(\A,s) = \sup\{h(\mu) + \Lambda(\mu,\A,s) \colon \mu\in \mathcal{M}_\sigma(\I^\N) \text{ and }h(\mu)<\infty\}
  \end{equation*}
  as claimed.

  Since $s \in \mathscr{I}_\A$, \cite[Theorem 3.5]{KaenmakiReeve2014} guarantees the existence of a fully supported measure $\mu^{(j)} \in \MM_\sigma(\I^\N)$ which is an accumulation point of $\{\mu_\J^{(j)} \colon \J$ is a nonempty finite subset of $\I$ such that $F \subset \J^*\}$ in the weak$^*$ topology and satisfies
  \begin{equation} \label{eq:quasimulti-equilibrium-state}
    C^{-1} \le \frac{\mu^{(j)}([\iii])}{\exp(-nP^{(j)}(\A,s)) \Phi^{(j)}_s(\iii)} \le C
  \end{equation}
  for all $\iii \in \I^n$. By \cite[Theorem 3.6]{KaenmakiReeve2014}, we see that $\mu^{(j)}$ is ergodic. Furthermore, \cite[Lemmas 3.7 and 3.8]{KaenmakiReeve2014} show that $\mu^{(j)}$ is the unique ergodic measure satisfying
  \begin{equation*}
    P^{(j)}(\A,s) = P^{(j)}(\mu^{(j)},\A,s).
  \end{equation*}
  Since, by Lemma \ref{thm:measure-theoretical-pressure-prop}\eqref{it:m-pressure1} and \eqref{eq:pressure-maximum},
  \begin{equation*}
    P^{(j)}(\mu^{(j)},\A,s) \le \max_{k \in \{1,\ldots,p\}} P^{(k)}(\mu^{(j)},\A,s) \le \max_{k \in \{1,\ldots,p\}} P^{(k)}(\A,s) = P(\A,s),
  \end{equation*}
  there are at least one and not more than $p$ distinct indices $j \in \{1,\ldots,p\}$ for which
  \begin{equation} \label{eq:k-reeve-equilibrium}
    P(\A,s) = \max_{k \in \{1,\ldots,p\}} P^{(k)}(\mu^{(j)},\A,s).
  \end{equation}
  To finish the proof it suffices to show that if $s > \theta_\A$, then
  \begin{equation} \label{eq:entropy-finite-goal}
    h(\mu^{(j)},\PP_\I) < \infty
  \end{equation}
  for all $j \in \{1,\ldots,p\}$. Indeed, if this was the case, then Lemma \ref{thm:measure-theoretical-pressure-prop}\eqref{it:m-pressure2} and \eqref{eq:measure-pressure-est} would show that
  \begin{align*}
    \max_{k \in \{1,\ldots,p\}} P^{(k)}(\mu^{(j)},\A,s) &= h(\mu^{(j)}) + \max_{k \in \{1,\ldots,p\}}\Lambda^{(k)}(\mu^{(j)},\A,s) \\ 
    &= h(\mu^{(j)}) + \Lambda(\mu^{(j)},\A,s)
  \end{align*}
  for $s>\theta_\A$ and \eqref{eq:k-reeve-equilibrium} finishes the proof. To prove \eqref{eq:entropy-finite-goal}, fix $j \in \{1,\ldots,p\}$ and notice that, by Lemma \ref{thm:entropy-basic-prop}\eqref{it:entropy1}, it suffices to show
  \begin{equation*}
    -\sum_{i \in \I} \mu^{(j)}([i])\log\mu^{(j)}([i]) < \infty
  \end{equation*}
  when $\I$ is countably infinite. Since $s > \theta_\A$, we have $P^{(j)}(\A,s)<\infty$ by \eqref{eq:pressure-maximum} and there is $\delta>0$ such that $s-\delta>\theta_\A$ and $\lceil s \rceil = \lceil s-\delta \rceil$. By Lemma \ref{th:first-lemma}\eqref{it:l11}, we have $\sum_{i \in \I} \fii^{s-\delta}(A_i) < \infty$. Since $\sigma_{\lceil s \rceil}(A_i)^s \le \fii^s(A_i)$ and $\sum_{i \in \I} \fii^s(A_i) < \infty$ by Lemma \ref{le:zeroth-lemma}, there exists a finite set $\K \subset \I$ such that
  \begin{equation*}
    -\log\fii^s(A_i) \le -s\log\sigma_{\lceil s \rceil}(A_i) \le \sigma_{\lceil s \rceil}(A_i)^{-\delta}
  \end{equation*}
  and
  \begin{equation*}
    CKe^{-P^{(j)}(\A,s)}\fii^s(A_i) < \frac{1}{e}
  \end{equation*}
  for all $i \in \I \setminus \K$, where $C \ge 1$ is as in \eqref{eq:quasimulti-equilibrium-state} and $K>0$ is as in Theorem \ref{thm:conc}. Therefore, as $\fii^s(A_i) \le \fii^{s-\delta}(A_i)\sigma_{\lceil s \rceil}(A_i)^\delta$, we see that
  \begin{equation*}
    -\sum_{i \in \I \setminus \K} \fii^s(A_i) \log\fii^s(A_i) \le \sum_{i \in \I \setminus \K} \fii^s(A_i) \sigma_{\lceil s \rceil}(A_i)^{-\delta} \le \sum_{i \in \I \setminus \K} \fii^{s-\delta}(A_i) < \infty.
  \end{equation*}
  Since the function $x \mapsto -x\log x$ is increasing on $[0,\frac{1}{e}]$, we get by recalling \eqref{eq:quasimulti-equilibrium-state} and Theorem \ref{thm:conc}\eqref{it:conc-1} that
  \begin{align*}
    -\sum_{i \in \I \setminus \K} \mu^{(j)}([i])\log\mu^{(j)}([i]) &\le -\sum_{i \in \I \setminus \K} CKe^{-P^{(j)}(\A,s)} \fii^s(A_i) \log CKe^{-P^{(j)}(\A,s)} \fii^s(A_i) \\ 
    &\le CKe^{-P^{(j)}(\A,s)}\biggl( -\sum_{i \in \I \setminus \K} \fii^s(A_i) \log \fii^s(A_i) \\ 
    &\qquad\qquad\quad+ (P^{(j)}(\A,s)-\log CK) \sum_{i \in \I \setminus \K} \fii^s(A_i) \biggr) < \infty
  \end{align*}
  and hence, $-\sum_{i \in \I} \mu^{(j)}([i])\log\mu^{(j)}([i]) < \infty$ as wished. The proof is finished.
\end{proof}

The next proposition proves all the remaining claims in Theorem \ref{thm:equilibrium-states-classification} which assume irreducibility and hence, also finishes the proof of Theorem \ref{thm:equilibrium-states-classification} under the assumption of Theorem \ref{thm:conc}.

\begin{proposition} \label{thm:eq-mixing}
  Let $\A=(A_i)_{i \in \I} \in \GL_d(\R)^\I$ be such that $\sup_{i \in \I} \|A_i\| < \infty$, where $\I$ is either finite or countably infinite. If $s > \theta_\A$, then the following two assertions hold:
  \begin{enumerate}[(i)]
    \item\label{it:eka} If $s \in (0,d) \cap \Z$ and $\A^{\wedge s}$ is irreducible then there is a unique $\fii^s$-equilibrium state for $\A$, and if additionally $\A^{\wedge s}$ is strongly irreducible then this unique equilibrium state is mixing.
    \item\label{it:toka} If $s \in (0,d) \setminus \Z$ and one of $\A^{\wedge \lfloor s\rfloor}$ and $\A^{\wedge \lceil s\rceil}$ is irreducible and the other is strongly irreducible then there is a unique $\fii^s$-equilibrium state for $\A$, and if both are strongly irreducible then this unique equilibrium state is mixing.
  \end{enumerate}
\end{proposition}

\begin{proof}
  We only prove the assertion \eqref{it:toka} as the proof of \eqref{it:eka} is similar and slightly easier. Write $k = \lfloor s \rfloor$. If one of $\A^{\wedge k}$ and $\A^{\wedge(k+1)}$ is irreducible and the other is strongly irreducible then the first claim follows immediately from \cite[Theorem 3]{KaenmakiMorris2018}. Let us thus assume that both $\A^{\wedge k}$ and $\A^{\wedge(k+1)}$ are strongly irreducible. Let $\mu \in \MM_\sigma(\I^\N)$ be the unique $\fii^s$-equilibrium state for $\A$. Recall that, by \eqref{eq:quasimulti-equilibrium-state}, \eqref{eq:pressure-maximum}, and Theorem \ref{thm:conc}\eqref{it:conc-1}, $\mu$ is ergodic and it satisfies
  \begin{equation} \label{eq:mu-mixing-gibbs}
    C^{-1} \le \frac{\mu([\iii])}{\exp(-nP(\A,s))\fii^s(A_\iii)} \le C
  \end{equation}
  for all $\iii \in \I^n$.

  Define $\A_n=(A_\iii)_{\iii \in \I^n}$ for each $n \geq 1$. We claim that $\A^{\wedge k}_n$ and $\A^{\wedge(k+1)}_n$ are both strongly irreducible for each $n \geq 1$. To see this, suppose for a contradiction that there is a finite collection of nonzero subspaces $U_1,\ldots,U_m\subset \wedge^k\R^d$ which is preserved by $\A_n^{\wedge k}$, say. Since $\A^{\wedge k}$ is strongly irreducible the set $\{A_\iii^{\wedge k} U_1 \colon \kkk \in \I^*\}$ must be infinite, and this set is clearly contained in the set $\{A_\iii^{\wedge k} U_j \colon 1 \leq |\iii| \leq n\text{ and }1\leq j \leq m\}$ by writing any arbitrary word $\kkk$ in the form $\iii\jjj$ where $n$ divides $|\jjj|$ and where $1 \leq |\iii|\leq n$. The latter set is therefore also infinite,  so by the pigeonhole principle there exist integers $\ell \in \{1,\ldots,n\}$ and $j_0 \in \{1,\ldots,m\}$ such that $\{A_\iii^{\wedge k}U_{j_0} \colon |\iii|=\ell\}$ is infinite; but $\{(A_1^{\wedge k})^{n-\ell}A_\iii^{\wedge k}U_{j_0} \colon |\iii|=\ell\}\subseteq \{A_\iii^{\wedge k} U_{j_0} \colon |\iii|=n\}\subseteq \{U_1,\ldots,U_m\}$ is finite, so $(A_1^{\wedge k})^{n-\ell}$ acts non-injectively on subspaces of $\wedge^k\R^d$, which is impossible since $A_1$ is invertible. The claim follows.

  Let us show that $\mu$ is totally ergodic for which we use the argument of \cite[Theorem 5(i)]{Morris2018}. Let $\iota_n \colon \I^\N \to (\I^n)^\N$ denote the natural identification of elements of $\I^\N$ with elements of $(\I^n)^\N$ given by $\iota_n((i_k)_{k=1}^\infty) = ( i_{(k-1)n+1}\cdots i_{kn})_{k=1}^\infty$, and observe that $\sigma \circ \iota_n = \iota_n \circ \sigma^n$. In particular, $(\iota_n)_*\mu$ defines an element of $\mathcal{M}_\sigma((\I^n)^\N)$ for every $n \geq 1$. By a straightforward calculation, we see that $(\iota_n)_*\mu$ is an $\fii^s$-equilibrium state for $\A_n$. By the strong irreduciblity of  $\A^{\wedge k}_n$ and $\A^{\wedge(k+1)}_n$ there exists a unique $\fii^s$-equilibrium state for $\A_n$ and that measure is ergodic with respect to $\sigma \colon (\I^n)^\N \to (\I^n)^\N$, so $(\iota_n)_*\mu$ is ergodic with respect to $\sigma \colon (\I^n)^\N \to (\I^n)^\N$. This implies via the relation $\sigma \circ \iota_n = \iota_n \circ \sigma^n$ that $\mu$ is ergodic with respect to $\sigma^n \colon \I^\N \to \I^\N$. Since $n$ is arbitrary this demonstrates that $\mu$ is totally ergodic.

  By \eqref{eq:mu-mixing-gibbs} and the sub-multiplicativity of the singular value function \eqref{eq:svf-sub-multi}, we have
  \begin{equation*}
    \mu([\iii\kkk\jjj]) \le C^4\mu([\iii])\mu([\kkk])\mu([\jjj])
  \end{equation*}
  for all $\iii,\kkk,\jjj \in \I^*$ and, consequently,
  \begin{equation*}
    \mu([\iii] \cap \sigma^{-|\iii|-n}([\jjj])) = \sum_{\kkk \in \I^n} \mu([\iii\kkk\jjj]) \le C^4\mu([\iii])\mu([\jjj])
  \end{equation*}
  for all $n \ge 1$. By arguments ultimately originating with \cite[Theorem 2.1]{Ornstein1972} and which are expressed in language more convenient for our purposes in the proof of \cite[Theorem 5(ii)]{Morris2018} this inequality together with the total ergodicity shows that $\mu$ is mixing.
\end{proof}

We remark that \cite[Proposition 6]{Morris2018} demonstrates that an $\fii^s$-equilibrium state for an irreducible $\A$ is not necessarily mixing when $0<s<1$. For non-integer parameters $s \in (1,d-1)$ it should also be possible to construct examples such that one of $\A^{\wedge \lfloor s\rfloor}$ and $\A^{\wedge \lceil s\rceil}$ is irreducible and the other is strongly irreducible, while the equilibrium state is not mixing. No example with these features has yet been noted explicitly   in the literature, but we believe that such an example could likely be constructed by a suitable modification of the examples in \cite{MorrisSert2019}. In the finite case, the results of \cite{Morris2021} extend Proposition \ref{thm:eq-mixing} to demonstrate that the unique equilibrium state is $\psi$-mixing and hence is measurably isomorphic to a Bernoulli shift in its natural extension. It is likely that this result can also be obtained in the infinite case, but we do not attempt this here. 

Let us next show that the $\fii^{\theta_\A}$-equilibrium state may not exist when $P(\A,\theta_\A)<\infty$. Our example mimics previous constructions such as \cite[Example 5.3]{MauldinUrbanski1996} and \cite[Example 4.5]{MauldinWilliams1986}.

\begin{example} \label{ex:no-theta-eq-state}
  Consider a tuple $\A=(a_i)_{i \in \I}$ of $1$-dimensional positive matrices, where $\I$ is countably infinite. We identify each real matrix $a_i$ with the corresponding positive real number. We suppose that $\sum_{i \in \I}a_i=1$ and $\sum_{i \in \I} a_i \log a_i=-\infty$, which holds for example in the case where $\I=\N$ and $a_n\equiv C/n(\log(n+1))^2$ for a suitable constant $C>0$. Straightforward calculations show that $\theta_\A=1$ and $P(\A,\theta_\A)=0$. We suppose for a contradiction that there exists a $\fii^{\theta_\A}$-equilibrium state $\mu \in \MM_\sigma(\I^\N)$, which by definition must have finite entropy. If $\nu \in  \MM_\sigma(\I^\N)$ is the Bernoulli measure satisfying $\nu([i])=\mu([i])$ for every $i \in \I$ then it is clear by direct calculations that $\Lambda(\nu,\A,\theta_\A)=\Lambda(\mu,\A,\theta_\A)$ and $h(\nu)\geq h(\mu)$. Since $\mu$ is an equilibrium state this implies that $\nu$ is also an equilibrium state and satisfies $h(\nu)=h(\mu)<\infty$, and this in turn gives $\sum_{i \in \I}-\nu([i])\log \nu([i])<\infty$. By Jensen's inequality and the fact that $\nu$ is an equilibrium state,
  \begin{align*}
    P(\A,\theta_\A) &= \log \sum_{i \in \I} a_i = \log\sum_{i \in \I} \nu([i])\exp(\log a_i - \log \nu([i])) \\
    &\geq \sum_{i \in \I}\nu([i])(\log a_i - \log \nu([i]))=\Lambda(\nu,\A,\theta_\A) + h(\nu)=P(\A,\theta_\A)=0.
  \end{align*}
  Since $\sum_{i \in \I}a_i=1$, in order for $\log\sum_{i\in\I} \nu([i])\exp(\log a_i - \log \nu([i]))$ to equal zero it must be the case that $\nu([i])>0$ for every $i \in \I$. Exact equality in Jensen's inequality now implies that $\log a_i - \log \nu([i])$ must be constant for all $i\in\I$, and we deduce that necessarily $a_i = \nu([i])$ for all $i\in \I$. This contradicts the requirement that $\sum_{i \in \I}-\nu([i])\log \nu([i])<\infty$ and we conclude that $\mu$ cannot be an equilibrium state.
\end{example}

Finally, let us use Theorem \ref{thm:equilibrium-states-classification} to prove Proposition \ref{thm:proppressuredrop} which we repeat below.

\begin{propositionno}{\ref{thm:proppressuredrop}}
  \proppressuredrop
\end{propositionno}

\begin{proof}
If $\J$ is a nonempty proper subset of $\I$ then the inequality $P((A_i)_{i \in \J},s)\leq P(\A,s)$ is clear from the definition of the pressure, so we need only show that $P((A_i)_{i \in \J},s)\neq P(\A,s)$. By Theorem \ref{thm:equilibrium-states-classification}, there exists an $\varphi^s$-equilibrium state $\mu \in \mathcal{M}_\sigma(\J^\N)$ for $(A_i)_{i \in \J}$. It is clear that we may identify $\mu$ with a measure on $\I^\N$ which has support equal to $\J^\N$, and if $P((A_i)_{i \in \J},s)= P(\A,s)$, this measure satisfies the definition of an $\varphi^s$-equilibrium state for $\A$. Therefore $\A$ has an $\varphi^s$-equilibrium state which is not fully supported on $\I^\N$, contradicting Theorem \ref{thm:equilibrium-states-classification}. This proves \eqref{it:drop1}.

Now suppose that for some norm $\threebar{\,\cdot\,}$ on $\R^d$ we have $\sup_{i \in \I} \threebar{A_i}<1$. Since $\theta_\A < \udimaff\A$, we see that $\dimaff\A$ exists by Proposition \ref{thm:propaffinity}\eqref{it:propaff2}. Note that, by the assertion \eqref{it:drop1}, we have $P((A_i)_{i \in \J},s) < P(\A,s)$ for every $s > \theta_\A$. By Theorem \ref{th:aff}\eqref{it:strictly} and \eqref{it:right-con}, the pressure as a function of $s$ is strictly decreasing and continuous at every $s > \theta_\A$. It follows that $\udimaff (A_i)_{i \in \J} < \dimaff \A$ which is the claim \eqref{it:drop2}.
\end{proof}

%
%

\subsection{Dimension of infinitely generated self-affine sets} \label{sec:conditional-self-affine}

In this section, we prove all the results announced in Section \ref{sec:self-affine-sets} by applying Theorem \ref{th:aff} which, at this stage, depends on Theorem \ref{thm:conc}. Let $X \subset \R^d$ be a self-affine set and $(T_i)_{i \in \I}$ its defining affine iterated function system associated such that $T_i(x)=A_ix+v_i$ for all $x \in \R^d$ and $i \in \I$. Write $\A = (A_i)_{i \in \I} \in \GL_d(\R)^\I$.

\begin{propositionno}{\ref{thm:selfaffinedimupperbound}}
  \selfaffinedimupperbound
\end{propositionno}

\begin{proof}
  If $\dimh X = 0$ then there is nothing to prove, so fix $0 < s < \dimh X$. Choose $k \in \{0,\ldots,d-1\}$ so that $s \in (k,k+1]$ and let $B \subset \R^d$ be a closed ball such that $T_i(B) \subseteq B$ for all $i \in \I$. It follows from the definition of the singular values that for each $\iii \in \I^*$ we may cover $T_\iii(B)$ with at most a constant times
  \begin{equation*}
    \frac{\sigma_1(A_\iii)}{\sigma_{k+1}(A_\iii)} \frac{\sigma_2(A_\iii)}{\sigma_{k+1}(A_\iii)} \cdots \frac{\sigma_k(A_\iii)}{\sigma_{k+1}(A_\iii)}
  \end{equation*}
  balls of radius $\sigma_{k+1}(A_\iii)$. Writing $\delta_n = \sup_{\iii \in \I^n}\|A_\iii\|$ for each $n \geq 1$ we thus see that there exists $c \ge 1$ so that
  \begin{equation*}
    \HH_{\delta_n}^s(X) \le \sum_{\iii \in \I^n} \HH_{\delta_n}^s(T_\iii(B)) \le c\sum_{\iii \in \I^n} \varphi^s(A_\iii)
  \end{equation*}
  for all $n \ge 1$, where $\HH^s$ is the $s$-dimensional Hausdorff measure. Since $\sup_{n \ge 1}\HH^s_{\delta_n}(X) = \HH^s(X) = \infty$, it follows that $\sum_{\iii \in \I^n} \varphi^s(A_\iii) \ge 1$ for all $n \ge 1$ large enough. Thus $P(\A,s) \ge 0$ and, by Theorem \ref{th:aff}\eqref{it:strictly} and the definition \eqref{eq:udimaff-def} of the upper affinity dimension, $s \le \udimaff\A$ which finishes the proof.
\end{proof}

If $\J$ is a nonempty finite subset of $\I$, then we denote the self-affine set associated to the finite affine iterated function system $(T_i)_{i \in \J}$ by $X_\J$. Observe that
for any sequence $(\J_n)_{n \ge 1}$ of finite subsets of $\I$ we have
\begin{equation} \label{eq:union-of-finite-subsets}
  \bigcup_{n=1}^\infty X_{\J_n} \subseteq X.
\end{equation}
If $\A = (A_i)_{i \in \I} \in \GL_d(\R)^\I$ is strongly irreducible, it does not automatically follow that there exists a finite a nonempty finite subset $\J$ of $\I$ such that $(A_i)_{i \in \J}$ is strongly irreducible: for example, if $\A=(A_i)_{i \in \N} \in \GL_2(\R)^\N$ where each $A_i$ is the matrix corresponding to rotation by $\pi/2^n$, then $\A$ is strongly irreducible but has no finite subsystem which is strongly irreducible. To circumvent this problem we use the following result:

\begin{proposition}\label{pr:finite-upwards}
Let $\A=(A_i)_{i \in \I} \in \GL_d(\R)^\I$, where $d \leq 3$ and $\I$ is countably infinite, and suppose that $\A$ is proximal and strongly irreducible. Then there exists a finite set $\J\subset \I$ such that $(A_i)_{i \in \J}$ is proximal and strongly irreducible.
\end{proposition}

The proof of Proposition \ref{pr:finite-upwards} requires the following lemma, which exploits the fact that proximality together with the failure of strong irreducibility in dimension two or three implies the existence of an invariant finite set of lines. (This implication becomes false in dimension four.)

\begin{lemma}\label{le:three-lines}
Let $\A =(A_i)_{i \in \I} \in  \GL_d(\R)^\I$ be proximal and irreducible, where $d \leq 3$. If $\A$ is not strongly irreducible then it preserves a union of $d$ one-dimensional subspaces.
\end{lemma}

\begin{proof}
We begin with the more difficult case $d=3$. Suppose that $\A$ is proximal, irreducible, and not strongly irreducible, in which case it preserves a finite set of proper subspaces $U_1,\ldots,U_m$. We suppose without loss of generality that these subspaces have the same dimension and we note that $m\neq 1$ by irreducibility. If the dimension of these subspaces is $2$ then $\A$ also preserves the nonempty, finite collection of all one-dimensional subspaces of the form $U_i \cap U_j$, so without loss of generality we suppose that each of $U_1,\ldots,U_m$ has dimension $1$. If $m<3$ then $\Span (U_1 \cup\cdots \cup U_m)$ is an invariant proper subspace of $\R^3$, contradicting irreducibility, so to prove the lemma it suffices for us to show that $m$ cannot be greater than $3$. 

Suppose for a contradiction that $m \geq 4$. Fix a  product $A_\iii$ with a simple leading eigenvalue, let $v \in \R^3$ be a leading eigenvector of $A_\iii$, and let $V$ be the unique $A_\iii$-invariant plane complementary to $v$. 
If one of the subspaces $U_i$ is not contained in $V$ and is also not parallel to $v$, then it is straightforward to verify that $\lim_{n \to \infty} A_\iii^n U$ is the leading eigenspace of $A_\iii$ and that $\{A_\iii^n U_i \colon n \geq 1\}$ is infinite, which is a contradiction since this set must be a subset of $\{U_1,\ldots,U_m\}$. Suppose instead that every subspace $U_i$ which is not parallel to $v$ is contained in $V$. Since at most one of these subspaces can be parallel to $v$, at least three subspaces must be contained in $V$. Let $U_1, U_2, U_3\subset V$ and let $u_1, u_2, u_3$ be nonzero vectors which span these respective spaces. Since $V$ is two-dimensional the vectors $u_1, u_2, u_3$ are linearly dependent. By irreducibility the smallest $\A$-invariant subspace which contains $u_1$ is $\R^3$ itself, and for this to be possible we must be able to choose a product $A_\jjj$ such that $A_\jjj u_1$ does not lie in $V$. By linear dependence either $A_\jjj u_2$ or $A_\jjj u_3$ \emph{also} does not lie in $V$. Thus at least two of the subspaces $A_\jjj U_1$, $A_\jjj U_2$ and $A_\jjj U_3$ are not contained in $V$, and by the pigeonhole principle at least one of those two is not parallel to $v$; but this subspace is necessarily equal to one of the $m$ subspaces $U_1,\ldots,U_m$, and this contradicts our supposition. We have arrived at a contradiction and we conclude that our earlier hypothesis $m \geq 4$ was impossible. It follows that $m=3$ and we have proved the lemma in the case $d=3$.

The remaining cases are much easier. The case $d=1$ is vacuous. In the case $d=2$, suppose that $\A$ preserves a finite union of one-dimensional subspaces $U_1,\ldots,U_m$, say, with $m\neq 2$. If $m<2$ then irreducibility is contradicted, so suppose instead that $m\geq 3$. Let $A_\iii$ have a simple leading eigenvalue and observe that one of the three subspaces $U_1, U_2, U_3$ is not an eigenspace of $A_\iii$, which implies that one of the sets $\{A_\iii^nU_i \colon n\geq 1\}$ is infinite, a contradiction. The lemma is proved.
\end{proof}

\begin{proof}[Proof of Proposition \ref{pr:finite-upwards}]
Choose arbitrarily an increasing sequence of finite sets $\J_1\subseteq \J_2 \subseteq \cdots$ whose union is $\I$. Since $\A$ is irreducible and proximal, there exists a product $A_\iii$ which has a simple leading eigenvalue. Clearly if $n$ is large enough that every symbol of $\iii$ belongs to $\J_n$ then $(A_i)_{i \in \J_n}$ is also proximal, so  $(A_i)_{i \in \J_n}$ is proximal for all sufficiently large $n$. Let $\mathcal{V}_n$ denote the set of all nonzero proper subspaces $V \subset \R^d$ which are preserved by $(A_i)_{i \in \J_n}$. Clearly each $\mathcal{V}_n$ is a closed subset of the Grassmannian manifold of $\R^d$, and $\mathcal{V}_{n+1}\subseteq \mathcal{V}_n $ for every $n \geq 1$. Any element of $\bigcap_{n=1}^\infty \mathcal{V}_n$ is an invariant subspace for $\A$, and since $\A$ is irreducible this intersection must be empty, which is only possible if $\mathcal{V}_n=\emptyset$ for all large enough $n$. We conclude that  $(A_i)_{i \in \J_n}$  is both proximal and irreducible for all large enough $n$.

Finally, it is clear that either $(A_i)_{i \in \J_n}$ is strongly irreducible for all large enough $n$, or for every $n\geq 1$ it fails to be strongly irreducible. In the latter case, let $\mathcal{W}_n$ denote the set of all  $(A_i)_{i \in \J_n}$-invariant $d$-tuples of lines in $\R^d$, which is a compact subset of $(\mathbb{RP}^{d-1})^d$.  By Lemma \ref{le:three-lines} the set $\mathcal{W}_n$ is nonempty for all large enough $n$, and this implies that it is nonempty for every $n\geq 1$. By a similar compactness argument, $\bigcap_{n=1}^\infty \mathcal{W}_n$ contains a tuple of $d$ lines (not necessarily all distinct) which are permuted by the matrices in $\A$, and this contradicts the strong irreducibility of $\A$. We conclude that $(A_i)_{i \in \J_n}$ is strongly irreducible for all large enough $n$, which proves the lemma.
\end{proof}

Recall that a strongly irreducible tuple is completely reducible. In view of \cite[Corollary 1.2]{Hochman2014}, \cite[Theorem 1.1]{HochmanRapaport2022}, \cite[Theorem 1.5]{MorrisSert2023preprint}, and \cite[Theorem 5.3]{Falconer1988}, the next theorem, together with Propositions \ref{thm:propaffinity}, \ref{thm:selfaffinedimupperbound}, and \ref{pr:finite-upwards}, proves Theorems \ref{thm:selfaffinedimone}--\ref{thm:selfaffinedimd}.

\begin{theorem} \label{thm:hausdorff-lower-affinity}
  Let $X \subset \R^d$ be a self-affine set and $(\J_n)_{n \ge 1}$ be an increasing sequence of finite subsets of $\I$ such that $\bigcup_{n=1}^\infty \J_n = \I$ and $\dimh X_{\J_n} = \min\{d,\dimaff(A_i)_{i \in \J_n}\}$ for all $n \ge 1$. Then $\dimh X \ge \min\{d,\ldimaff\A\}$.
\end{theorem}

\begin{proof}
  The monotonicity and the countable stability of the Hausdorff dimension applied with \eqref{eq:union-of-finite-subsets}, the assumption on the Hausdorff dimension of finitely generated self-affine subsets, and the definition \eqref{eq:ldimaff-def} of the lower affinity dimension immediately imply
  \begin{equation*}
    \dimh X \ge \sup_{n \ge 1} \dimh X_{\J_n} = \min\{d,\sup_{n \ge 1}\dimaff(A_i)_{i \in \J_n}\} = \min\{d,\ldimaff\A\}
  \end{equation*}
  as wished.
\end{proof}

%
%

\section{Algebraic arguments: Proof of Theorem \ref{thm:conc}} \label{sec:algebra-stuff}

In this section, we prove Theorem \ref{thm:conc} and hence, also conclude the proofs of Theorems \ref{th:aff} and \ref{thm:equilibrium-states-classification}.

\subsection{Linear algebraic groups and the Zariski topology}

A function $p \colon \GL_d(\mathbb{R}) \to \mathbb{R}$ is called a \emph{polynomial} if it maps each matrix $A=[a_{ij}]_{i,j=1}^d$ to the same polynomial function of the $d^2+1$ variables $a_{11},\ldots,a_{dd}$ and $1/\det A$. The \emph{Zariski topology} on $\GL_d(\mathbb{R})$ is then defined to be the smallest topology in which every set of the form $\{A \in \GL_d(\mathbb{R}) \colon p(A)=0\}$ is closed. The Zariski topology has the following important property, called the \emph{descending chain condition}: if $(Z_n)_{n=1}^\infty$ is a sequence of Zariski-closed sets such that $Z_{n+1} \subseteq Z_n$ for every $n \ge 1$, then $(Z_n)_{n=1}^\infty$ is eventually constant. This property implies that a set is Zariski closed if and only if it is the intersection of the zero loci of a finite collection of polynomials.

We recall that if $G$ is a group or semigroup then a \emph{representation} of $G$ is a homomorphism $\phi \colon G \to \GL(V)$ for some vector space $V$ over a field $\mathbb{K}$. In this article it will always be the case that $V$ is finite-dimensional and $\mathbb{K}$ is $\mathbb{R}$, although all our results also hold over the complex field without modification. The representation  $\phi \colon G \to \GL(V)$ is called \emph{irreducible} if the only subspaces of $V$ which are preserved by every element of $\phi(G)$ are $\{0\}$ and $V$; equivalently, $\phi$ is an irreducible representation if $\phi(G)$ is irreducible in the sense of Section \ref{sec:completely-reducible}. A representation is called \emph{faithful} if it is injective. A representation $\phi \colon G \to \GL(V)$ is called \emph{semisimple} if there exists a splitting $V=\bigoplus_{j=1}^\ell V_j$ such that each $V_j$ is preserved by every element of $\phi(G)$, and such that additionally each of the representations $\phi_j \colon G \to \GL(V_j)$ defined by $\phi_j(g)=\phi(g)|_{V_j}$ is an irreducible representation. One of the simplest examples of a representation which is not semisimple is the representation $\mathbb{R} \to \GL_2(\R)$ defined by
\[
  \phi(t)=\begin{pmatrix}1&t\\ 0 &1\end{pmatrix}.
\]

\subsection{Key results}
We require two fundamental results in order to prove Theorem \ref{thm:conc}, the first dealing principally with part \eqref{it:conc-1} of that theorem. This result is now standard in the theory of reductive linear groups.
\begin{proposition}\label{pr:cr}
Let $V$ be a finite-dimensional real or complex vector space, let $\I$ be a nonempty set, let $(A_i)_{i \in \I} \in \GL(V)^{\I}$ be completely reducible, and let $\ell \in \{1,\ldots,\dim V\}$. Then $(A_i^{\wedge \ell})_{i \in \I} \in \GL(\wedge^\ell V)^\I$ is completely reducible.
\end{proposition}
\begin{proof}
Let $G \leq \GL(V)$ denote the Zariski closure of the semigroup $\{A_\iii \colon \iii \in \I^*\}$. By hypothesis the inclusion representation $\iota \colon G \to \GL(V)$ is faithful and semisimple, which by Clifford's theorem implies that the restriction of $\iota$ to $G^0$ is also faithful and semisimple. By \cite[Theorem 22.42]{Milne2017} this implies that $G^0$ is reductive, hence by \cite[Corollary 22.43]{Milne2017} every representation from $G$ to a finite-dimensional real vector space is semisimple. In particular the representation $G \to \GL(\wedge^\ell V)$ defined by $g\mapsto g^{\wedge \ell}$ is semisimple. Since $(A_i)_{i \in \I}$ is Zariski dense in $G$, the invariant subspaces of $G$ acting on $\GL(\wedge^\ell V)$ are precisely the invariant subspaces of $(A_i^{\wedge \ell})_{i \in \I}$, so $(A_i^{\wedge \ell})_{i \in \I} \in \GL(\wedge^\ell V)^\I$ is completely reducible as required.
\end{proof}

The second result forms the core of Theorem \ref{thm:conc}\eqref{it:conc-2} and is similar in content to the results of \cite[\S 4]{BochiMorris2018}.

\begin{proposition}\label{pr:qm}
Let $\Gamma$ be a semigroup, let $\mathbb{K}$ be either $\mathbb{R}$ or $\mathbb{C}$, let $k \geq 1$ and for each $j \in \{1,\ldots,k\}$ let $V_j$ be a finite-dimensional inner product space  over $\mathbb{K}$ and $\phi_j \colon \Gamma \to \GL(V_j)$ a representation. For each $j \in \{1,\ldots,k\}$ let $U_j$ be a nonzero subspace of $V_j$ which has finite orbit under $\phi_j(\Gamma)$ and has the least possible dimension of any such subspace. Define
\[
  \mathcal{W}=\{ (\phi_j(g)U_j)_{j=1}^k \colon g \in \Gamma\} \subseteq \prod_{j=1}^k \Gr(V_j),
\]
which is necessarily a finite set. Then there exist a finite set $F \subseteq \Gamma$ and a real number $\kappa>0$ such that for every $(W_j)_{j=1}^k, (W_j')_{j=1}^k \in \mathcal{W}$  we may choose $h \in F$ such that for every $g_1,g_2 \in \Gamma$ we have
\[
  \|\phi_j(g_1h  g_2)|_{W_j}\| \geq \kappa \|\phi_j(g_1)|_{W_j'}\| \|\phi_j(g_2)|_{W_j}\|
\]
simultaneously for all $j \in \{1,\ldots,k\}$.
\end{proposition}

The first proposition is used to show that under the main hypothesis of Theorem \ref{thm:conc} also the exterior powers of $(A_i)_{i \in \I}$ can be block diagonalised with irreducible diagonal blocks. This allows us to write the singular value function $\varphi^s$ directly as a maximum of potentials arising from restrictions to the blocks. The second proposition is used to show that these potentials can in turn be written as maxima of finite collections of quasi-multiplicative potentials. We remark that in \cite{BochiMorris2018} it was only possible to say that the singular value pressure is equal to the maximum of the pressures of potentials arising from the blocks. Under complete reducibility we can say that the potentials are also equal.

\subsection{Proof of Proposition \ref{pr:qm}}

To prove the proposition it is sufficient for us to fix an arbitrary pair $(W_j)_{j=1}^k, (W_j')_{j=1}^k \in \mathcal{W}$ and construct a finite set $F' \subset \Gamma$ and real constant $\kappa'>0$ having the claimed properties only with respect to that specific choice of $(W_j)_{j=1}^k, (W_j')_{j=1}^k \in \mathcal{W}$. We may then define $F$ to be the union of the finite sets $F'$ thus constructed for different pairs $(W_j)_{j=1}^k, (W_j')_{j=1}^k \in \mathcal{W}$ and likewise take $\kappa$ to be the minimum of the finitely many constants $\kappa'$ in order to deduce the conclusion of the proposition. We therefore fix $(W_j)_{j=1}^k, (W_j')_{j=1}^k \in \mathcal{W}$ throughout the proof and prove the proposition in this simpler form.

Define $V=\bigoplus_{j=1}^k V_j$ and $\phi(h)=\bigoplus_{j=1}^k \phi_j(h) \in \GL(V)$ for all $h \in \Gamma$. Let $G\leq \GL(V)$ denote the Zariski closure of $\phi(\Gamma)$ in $\GL(V)$ and for every $j \in \{1,\ldots,k\}$ define a representation $\psi_j \colon G \to \GL(V_j)$ by setting $\psi_j(g)=g|_{V_j}$. We have $\phi_j = \psi_j \circ \phi$ for every $j \in \{1,\ldots,k\}$ and $\psi_j \colon G \to \GL(V_j)$ is clearly Zariski continuous. We may write $G$ as the disjoint union of finitely many Zariski-connected components $G^0,\ldots,G^m$ each of which is an irreducible variety. We let $G^0$ denote the unique component which contains the identity. For each $j$ let $U_j^1,\ldots,U_j^{r_j}$ denote the orbit of $U_j$ under $\phi_j(\Gamma)$. For fixed $j$ the sets
\begin{equation}\label{eq:etc}
\{g \in G \colon \psi_j(g)U_j = U_j^i\}
\end{equation}
for $i \in \{1,\ldots,r_j\}$ clearly partition $G$. Each such set is Zariski closed since the condition $\psi_j(g)U_j = U_j^i$ is equivalent to the statement that $\psi_j(g)$ takes every element of a basis for $U_j$ into the orthogonal complement of a basis for $(U_j^i)^\perp$; thus each of these sets is the common zero locus of some finite collection of polynomial functions of the matrix entries of $g$. Since the sets $\{g \in G \colon \psi_j(g)U_j = U_j^i\}$ are Zariski closed and partition $G$, they are clopen in the Zariski topology and hence each is equal to the union of a finite collection of connected components of $G$. The set
\begin{equation}\label{eq:its-a-set}
\{g \in G \colon (\psi_j(g)W_j)_{j=1}^k = (W_j')_{j=1}^k\}
\end{equation}
is therefore also equal to the union of a finite collection of connected components of $G$, since it is equal to a finite intersection of sets of the form \eqref{eq:etc}. Since by definition there exist $h_1,h_2 \in \Gamma$ such that $(W_j)_{j=1}^k = (\phi_j(h_1)U_j)_{j=1}^k$ and $(W_j')_{j=1}^k = (\phi_j(h_2)U_j)_{j=1}^k$, the set \eqref{eq:its-a-set} contains $\phi(h_2h_1^{-1})$ and is therefore nonempty. We conclude that at least one of the Zariski-connected components of $G$ must be a subset of \eqref{eq:its-a-set}.

For the remainder of the proof we fix a component $G^i$ of $G$ such that $(\psi_j(g)W_j)_{j=1}^k = (W_j')_{j=1}^k$ for all $g \in G^i$, which is possible by the preceding discussion. We claim that if $j \in \{1,\ldots,k\}$, $u \in W_j$, and $v \in W_j'$ with $u,v \neq 0$, then the set
\begin{equation}\label{eq:nonempty}
\{g \in G^i \colon \langle \psi_j(g)u,v\rangle \neq 0\}
\end{equation}
is nonempty. Fix $j$ and suppose for a contradiction that the set is empty. In this case
\[
W=\Span\{ \psi_j(g)u \colon g \in G^i\}
\]
defines a vector subspace of $W_j'$ which contains at least one nonzero vector but does not contain $v$. In particular this subspace is a nonzero proper subspace of $W_j'$. Since $G^0G^i=G^i$ we moreover have $\psi_j(g)W=W$ for every $g \in G^0$. If $g_1,g_2 \in G$ belong to the same component $G^r$ of $G$ then $g_1^{-1}g_2 \in g_1^{-1}G^r=G^0$ and therefore $\psi_j(g_1^{-1}g_2)W=W$, so $\psi_j(g_1)W=\psi_h(g_2)W$ whenever $g_1$ and $g_2$ belong to the same component of $G$. Therefore $g \mapsto \psi_j(g)W$ takes only finitely many values as $g$ varies over $G$. Hence the set
\[
\{\phi_j(h)W \colon h \in \Gamma \} = \{\psi_j(\phi(h))W \colon h \in \Gamma\} \subseteq \{\psi_j(g)W \colon g \in G\}
\]
is finite and we have $0 < \dim W < \dim W_j' = \dim U_j$. This contradicts the definition of $U_j$ as a nonzero subspace of $V_j$ with finite orbit under $\phi_j(\Gamma)$ and having the least dimension among all such subspaces. We conclude that the set \eqref{eq:nonempty} must be nonempty as claimed.

Now let $u_j \in W_j$ and $v_j \in W_j'$ be nonzero vectors for each $j \in \{1,\ldots,k\}$. By the previous claim the set
\[
\{g \in G^i \colon \langle \psi_j(g)u_j,v_j\rangle \neq 0\}
\]
is nonempty for each $j$, and it is clearly Zariski open. Since $G^i$ is an irreducible variety all of its nonempty Zariski open subsets are Zariski dense, so the intersection
\[
\bigcap_{j=1}^k \{g \in G^i \colon \langle \psi_j(g)u_j,v_j\rangle \neq 0\}
\]
is nonempty and Zariski open. The set $\{\phi(h) \colon h \in \Gamma\}$ is by definition Zariski dense in $G$, so its intersection with this nonempty open set is nonempty. Hence there exists $h \in \Gamma$ such that $\langle \psi_j(\phi(h))u_j,v_j\rangle \neq 0$ for every $j \in \{1,\ldots,k\}$. This means precisely that  $\langle \phi_j (h)u_j,v_j\rangle \neq 0$ for every $j \in \{1,\ldots,k\}$.

For each $j$ let $S_{W_j}$ and $S_{W_j'}$ denote the unit spheres of $W_j$ and $W_j'$ respectively. The preceding paragraph implies that for every pair $((u_j)_{j=1}^k,(v_j)_{j=1}^k) \in (\prod_{j=1}^k S_{W_j}) \times (\prod_{j=1}^k S_{W_j'})$ there exists $h \in \Gamma$ such that $\langle \phi_j (h)u_j,v_j\rangle \neq 0$ for all $j \in \{1,\ldots,k\}$. But such an $h$ clearly also has this property for every $((u_j')_{j=1}^k,(v_j')_{j=1}^k) \in (\prod_{j=1}^k S_{W_j}) \times (\prod_{j=1}^k S_{W_j'})$ which is sufficiently close to $((u_j)_{j=1}^k,(v_j)_{j=1}^k)$. By the compactness of $(\prod_{j=1}^k S_{W_j}) \times (\prod_{j=1}^k S_{W_j'})$ it follows that there is a finite set $F \subseteq \Gamma$ such that 
\[
\max_{h \in F} \min_{j \in \{1,\ldots,k\}} |\langle \phi_j(h)u_j,v_j\rangle| \neq 0
\]
for every $((u_j)_{j=1}^k,(v_j)_{j=1}^k) \in (\prod_{j=1}^k S_{W_j}) \times (\prod_{j=1}^k S_{W_j'})$. By continuity and compactness it follows that the real number
\[
\kappa=\min_{(u_j)_{j=1}^k\in \prod_{j=1}^k S_{W_j}} \min_{(v_j)_{j=1}^k\in \prod_{j=1}^k S_{W_j'}}  \max_{h \in F} \min_{j \in \{1,\ldots,k\}} |\langle \phi_j(h)u_j,v_j\rangle|
\]
must be strictly positive. By homogeneity it follows that if for every $j \in \{1,\ldots,k\}$ we let $u_j \in W_j$, $v_j \in W_j'$ be arbitrary vectors then there exists $h \in F$ such that
\[
|\langle \phi_j(h)u_j,v_j\rangle| \geq \kappa \|u_j\|\|v_j\|
\]
simultaneously for every $j \in \{1,\ldots,k\}$. 

Now let $g_1,g_2 \in \Gamma$ be arbitrary. For each $j \in \{1,\ldots,k\}$ there exist a unit vector $u_j \in W_j$ such that $\|\phi_j(g_2)u_j\|=\|\phi_j(g_2)|_{W_j}\|$ and a unit vector $v_j \in \phi_j(g_1)W_j'$ such that $\|\phi_j(g_1)^Tv_j\|=\|\phi_j(g_1)^T|_{\phi_j(g_1)W_j'}\|=\|\phi_j(g_1)|_{W_j'}\|$. We may therefore choose $h \in F$ such that for all $j \in \{1,\ldots,k\}$ we have
\[
|\langle \phi_j(h)\phi_j(g_2)u_j,\phi_j(g_1)^Tv_j\rangle| \geq \kappa \|\phi_j(g_2)u_j\|\|\phi_j(g_1)^Tv_j\|.
\]
But this implies the inequality 
\begin{align*}
  \|\phi_j(g_1hg_2)|_{W_j}\| &\geq \|\phi_j(g_1hg_2)u_j\| \geq |\langle \phi_j(g_1)\phi_j(h)\phi_j(g_2)u_j,v_j\rangle| \\ 
  &= |\langle \phi_j(h)\phi_j(g_2)u_j,\phi_j(g_1)^Tv_j\rangle| \geq  \kappa \|\phi_j(g_2)u_j\|\|\phi_j(g_1)^Tv_j\| \\ 
  &= \kappa \|\phi_j(g_1)|_{W_j}\| \|\phi_j(g_2)|_{W_j'}\|
\end{align*}
simultaneously for all $j \in \{1,\ldots,k\}$. The proposition is proved.

\subsection{Applications to quasi-multiplicativity}

The following result extends the main technical tool of \cite{BochiMorris2018} to the case of a countably infinite index set $\I$ and also makes explicit the dependence on the parameters $\beta_i$.

\begin{theorem} \label{thm:technical-irreducible}
Let $k \geq 1$, let $\I$ be finite or countably infinite, and for each $j \in \{1,\ldots,k\}$ let $V_j$ be a real or complex inner product space and let $\A^{(j)}=(A^{(j)}_i)_{i\in\I}\in\GL(V_j)^\I$ be irreducible. For each $j \in \{1,\ldots,k\}$ let $\ell_j \geq 1$ denote the smallest possible dimension of a nonzero subspace of $V_j$ which has finite orbit under the  action of $\A^{(j)}$. Then there exist constants $\kappa,\tau>0$, an integer $p \geq 1$, finite sets $\mathscr{W}_1,\ldots,\mathscr{W}_p \subseteq \prod_{j=1}^k \Gr_{\ell_j}(V_j)$, and a finite set $F \subset \I^*$ such that for every $t \in \{1,\ldots,p\}$ and all non-negative real numbers $\beta_1,\ldots,\beta_k \geq 0$ the function
\[
\Psi^{(t)}(\iii)=\max_{(W_j)_{j=1}^k \in \mathscr{W}_t} \prod_{j=1}^k \|A_\iii^{(j)}|_{W_j}\|^{\beta_j}
\]
and the number $\beta = \sum_{j=1}^k \beta_j \ge 0$ satisfy
\begin{equation}\label{eq:qm2}
\Psi^{(t)}(\iii\jjj) \le \Psi^{(t)}(\iii)\Psi^{(t)}(\jjj) \le \kappa^{-\beta} \max_{\kkk \in F} \Psi^{(t)}(\iii\kkk\jjj)
\end{equation}
for all $\iii,\jjj \in \I^*$, and 
\begin{equation}\label{eq:max-equivalence}
\tau^{\beta} \prod_{j=1}^k \|A_\iii^{(j)}\|^{\beta_j} \leq \max_{t \in \{1,\ldots,p\}} \Psi^{(t)}(\iii) \leq \prod_{j=1}^k \|A_\iii^{(j)}\|^{\beta_j}.
\end{equation}
Furthermore $\ell_j$ divides $\dim V_j$ for every $j \in \{1,\ldots,k\}$ and the integer $p$ satisfies
\begin{equation} \label{eq:p-number}
  p \leq \min_{t \in \{1,\ldots,k\}} \prod_{\atop{j \in \{1,\ldots,k\}}{j \neq t}} \frac{\dim V_j}{\ell_j} \le \prod_{j=1}^{k-1} \dim V_j.
\end{equation}
In particular, if $\A^{(j)}$ is strongly irreducible for at least $k-1$ values of $j$, then $p=1$.
\end{theorem}

\begin{proof}
For each $j \in \{1,\ldots,k\}$ choose an $\ell_j$-dimensional subspace $U_j \subseteq V_j$ which has finite orbit under the action of $\A^{(j)}$ and let $\{U_j^1,\ldots,U_j^{n_j}\}$ denote the orbit of $U_j$. Since $\Span \bigcup_{i=1}^{n_j} U_j^i$ is invariant under $\A^{(j)}$ and has nonzero dimension, by irreducibility it must equal $V_j$. 

We claim that for each $j \in \{1,\ldots,k\}$ we may write $V_j$ as a direct sum of a subset of the spaces $U_j^1,\ldots,U_j^{n_j}$. To see this, fix $j$ and let $U_j^{i_1},\ldots,U_j^{i_{m_j}}$ be elements of the orbit of $U_j$ which form a direct sum $U_j^{i_1}\oplus \cdots \oplus U_j^{i_{m_j}}$ with $m_j$ as large as possible. Since $U_j^1$ itself forms a direct sum with a single summand, the set of direct sums among the spaces $U_j^i$ is nonempty. The set of all such direct sums is clearly also finite and it follows that $m_j$ is well-defined. We wish to show that $U_j^{i_1}\oplus \cdots \oplus U_j^{i_{m_j}}=V_j$, and to demonstrate this it is sufficient to show that $m_j\ell_j=\dim V_j$. Obviously $\ell_jm_j= \dim U_j^{i_1}\oplus \cdots \oplus U_j^{i_{m_j}} \leq \dim V_j$, so suppose for a contradiction that $m_j\ell_j<\dim V_j$. Since $\Span \bigcup_{i=1}^{n_j} U_j^i=V_j$ we must be able to choose $U_j^{i'}$ such that $U_j^{i'}$ is not a subspace of $U_j^{i_1}\oplus \cdots \oplus U_j^{i_{m_j}}$, which implies $\dim (U_j^{i'} \cap (U_j^{i_1}\oplus \cdots \oplus U_j^{i_{m_j}}))<\dim U_j^{i'}=\ell_j$. On the other hand we cannot have $\dim (U_j^{i'} \cap (U_j^{i_1}\oplus \cdots \oplus U_j^{i_{m_j}}))=0$ since then $U_j^{i_1}\oplus \cdots \oplus U_j^{i_{m_j}} \oplus U_j^{i'}$ would be a direct sum with more than $m_j$ summands, contradicting the maximality of $m_j$; and  if $0<\dim (U_j^{i'} \cap (U_j^{i_1}\oplus \cdots \oplus U_j^{i_{m_j}}))<\ell_j$ then the subspace $U_j^{i'} \cap (U_j^{i_1}\oplus \cdots \oplus U_j^{i_{m_j}})$ has finite orbit under the action of $\A^{(j)}$ but has dimension strictly between $0$ and $\ell_j$, contradicting the definition of $\ell_j$. We conclude that the subspace $U_j^{i'}$ cannot exist since there are no viable possibilities for the dimension of the subspace $U_j^{i'} \cap (U_j^{i_1}\oplus \cdots \oplus U_j^{i_{m_j}})$ and therefore the inequality $m_j\ell_j<\dim V_j$ must be false. We conclude that $V_j$ is equal to the direct sum of $m_j$ spaces $U_j^{i_1},\ldots,U_j^{i_{m_j}}$.

By permuting the labels of the spaces $U_j^i$ if necessary, for the remainder of the proof we assume without loss of generality that $V_j=U_j^1\oplus U_j^2 \oplus \cdots \oplus U_j^{m_j}$ for each $j \in \{1,\ldots,k\}$. We observe that $m_j\ell_j=\dim V_j$ and in particular $\ell_j$ divides $\dim V_j$ for each $j$ as required. By permuting the indices $j \in \{1,\ldots,k\}$ we further assume without loss of generality that $m_1=\max_{j \in \{1,\ldots,k\}}m_j$ and therefore
\begin{equation} \label{eq:p-number1}
  \min_{t \in \{1,\ldots,k\}} \prod_{\atop{j \in \{1,\ldots,k\}}{j \neq t}} \frac{\dim V_j}{\ell_j}=\prod_{j=2}^k m_j.
\end{equation}
We next claim that there exists $\tau_1>0$ such that for every $j \in \{1,\ldots,k\}$ we have
\begin{equation}\label{eq:what-tau1-does}
\max_{i \in \{1,\ldots,m_j\}} \|B|_{U_j^i}\| \geq \tau_1\|B\|
\end{equation}
for every $B \in \End(V_j)$, where the set $\End(V)$ is the collection of all endomorphisms of $V$, i.e. the collection of all linear transformations $V \to V$. Clearly it suffices to prove this claim individually for each $j \in \{1,\ldots,k\}$ and then take $\tau_1$ to be the minimum of the $k$ distinct constants thus obtained. For fixed $j$ it is in turn clearly sufficient to show that
\[
\max_{i \in \{1,\ldots,m_j\}} \|B|_{U_j^i}\| \geq \tau_1
\]
for every $B \in \End(V_j)$ with norm $\|B\|=1$, and by compactness this will follow if $ \max_{i \in \{1,\ldots,m_j\}} \|B|_{U_j^i}\|>0$ for every nonzero  $B \in \End(V_j)$. But if $\max_{i \in \{1,\ldots,m_j\}} \|B|_{U_j^i}\|=0$ then $B$ must be identically zero on $\bigcup_{i=1}^{m_j} U_j^i$ and is therefore also identically zero on $V_j = \Span \bigcup_{i=1}^{m_j} U_j^i$. The claim follows.

Define
\[
\mathcal{W}=\{ (U_j^{i_j})_{j=1}^k \colon i_j \in \{1,\ldots,n_j\}\text{ for every }j \in \{1,\ldots,k\}\} \subseteq \prod_{j=1}^k \Gr_{\ell_j}(V_j)
\]
which is clearly a finite set. Each $\iii \in \I^*$ induces a permutation of $\mathcal{W}$ by the map $(U_j^{i_j})_{j=1}^k\mapsto (A^{(j)}_\iii U_j^{i_j})_{j=1}^k$, so we may partition $\mathcal{W}$ into finitely many disjoint sets $\mathscr{W}_1,\ldots,\mathscr{W}_q$ each of which is closed with respect to this action and such that the action of $\I^*$ by these permutations is transitive on each $\mathscr{W}_t$. By relabelling $\mathscr{W}_1,\ldots,\mathscr{W}_q$ if necessary, we assume without loss of generality that
\begin{equation}\label{eq:what-w-does}
\{(U_j^{i_j})_{j=1}^k \colon i_1=1\text{ and }i_j \in \{1,\ldots,m_j\} \text{ for all }j \in \{2,\ldots,k\}\} \subseteq \bigcup_{t=1}^p \mathscr{W}_t
\end{equation}
where $p \leq \prod_{j=2}^k m_j$. This is possible since the former set contains exactly $\prod_{j=2}^k m_j$ distinct elements and these clearly must be distributed among no more than $\prod_{j=2}^k m_j$ of the distinct transitivity classes $\mathscr{W}_t$. By recalling \eqref{eq:p-number1}, we have now shown \eqref{eq:p-number}.

To see the final claim of the theorem, suppose that $\A^{(j)}$ is strongly irreducible for at least $k-1$ values of $j$. Since $\ell_j \ge 1$ is the smallest dimension of a nonzero subspace of $V_j$ which has finite orbit under the action of $\A^{(j)}$, we see that $\ell_j = \dim V_j$ whenever $\A^{(j)}$ is strongly irreducible. Therefore $\dim V_j/\ell_j = 1$ for at least $k-1$ values of $j$ and \eqref{eq:p-number} gives
\begin{equation*}
  p \leq \min_{t \in \{1,\ldots,k\}} \prod_{\atop{j \in \{1,\ldots,k\}}{j \neq t}} \frac{\dim V_j}{\ell_j} = 1
\end{equation*}
as claimed.

Let us next prove \eqref{eq:max-equivalence}. We first claim that for every $i \in \{1,\ldots,n_1\}$ there exists $\lll_i \in \I^*$ such that $A_{\lll_i}^{(1)}U_1^1=U_1^i$. Indeed, by the definition of $U_1^1,\ldots,U_1^{n_1}$ there exist $\kkk_1,\ldots,\kkk_{n_1} \in \I^*$ such that $U_1^i=A_{\kkk_i}^{(1)}U_1$ for every $i \in \{1,\ldots,n_1\}$. We allow here $\kkk_i$ to be an empty word, in which case $A_{\kkk_i}^{(1)}$ is the identity map. For each $i_0 \in \{1,\ldots,n_1\}$ the map $U_1^i \mapsto A_{\kkk_{i_0}}^{(1)}U_1^i$ induces a permutation of the set $\{U_1^1,\ldots,U_1^{n_1}\}$ and therefore the map $U_1^i \mapsto A_{\kkk_{i_0}^{n_1!-1}}^{(1)}U_1^i$ induces its inverse permutation. We therefore have $A_{\kkk_i}^{(1)}A_{\kkk_1^{n_1!-1}}^{(1)}U_1^1=A_{\kkk_i}^{(1)}U_1=U_1^i$ for each $i \in \{1,\ldots,n_j\}$ and the claim follows by taking $\lll_i=\kkk_i\kkk_1^{n_1!-1}$ for each $i$. Recall that for every $t \in \{1,\ldots,p\}$ and $\iii \in \I^*$,
\[
  \Psi^{(t)}(\iii)=\max_{(W_j)_{j=1}^k \in \mathscr{W}_t} \prod_{j=1}^k \|A_\iii^{(j)}|_{W_j}\|^{\beta_j},
\]
and define also
\[
  \tau_2=\min_{j \in \{1,\ldots,k\}} \min_{i \in \{1,\ldots,n_1\}} \bigl(\|A_{\lll_i}^{(j)}\|^{-1} \|(A_{\lll_i}^{(j)})^{-1}\|^{-1} \bigr).
\]
We will show that
\[
  (\tau_1\tau_2)^{\beta} \prod_{j=1}^k \|A_\iii^{(j)}\|^{\beta_j} \leq \max_{t \in \{1,\ldots,p\}} \Psi^{(t)}(\iii)
\]
for every $\iii \in \I^*$. Given $\iii \in \I^*$, using \eqref{eq:what-tau1-does} we may choose an integer $i \in \{1,\ldots,m_1\}$ such that $\|A_\iii^{(1)}|_{U_j^i}\| \geq \tau_1 \|A_\iii^{(1)}\|$. We then have 
\begin{equation}
\begin{split}
  \prod_{j=1}^k \|A_\iii^{(j)}\|^{\beta_j} &\leq \tau^{-\beta_1}_1 \|A_\iii^{(1)}|_{U_j^i}\|^{\beta_j}  \prod_{j=2}^k \|A_\iii^{(j)}\|^{\beta_j} \\
  &=\tau^{-\beta_1}_1 \|A_\iii^{(1)}A_{\lll_i}^{(1)}(A_{\lll_i}^{(1)})^{-1} |_{U_j^i}\|^{\beta_j}  \prod_{j=2}^k \|A_\iii^{(j)}A_{\lll_i}^{(j)}(A_{\lll_i}^{(j)})^{-1}\|^{\beta_j} \\
  &\leq \tau^{-\beta_1}_1\biggl(\prod_{j=1}^k \|(A_{\lll_i}^{(j)})^{-1}\|^{\beta_j}\biggr) \|A_\iii^{(1)}A_{\lll_i}^{(1)}|_{(A_{\lll_i}^{(1)})^{-1}U_j^i}\|^{\beta_j}  \prod_{j=2}^k \|A_\iii^{(j)}A_{\lll_i}^{(j)}\|^{\beta_j}\\\label{eq:last}
  &\leq \tau^{-\beta_1}_1\biggl(\prod_{j=1}^k \|(A_{\lll_i}^{(j)})^{-1}\|^{\beta_j}\biggr) \|A_\iii^{(1)}A_{\lll_i}^{(1)}|_{U_1^1}\|^{\beta_j}  \prod_{j=2}^k \|A_\iii^{(j)}A_{\lll_i}^{(j)}\|^{\beta_j}
\end{split}
\end{equation}
where we have used the definition of $\lll_i$. Now define $i_1=1$ and, again using \eqref{eq:what-tau1-does}, choose $i_2,\ldots,i_k$ such that $i_j \in \{1,...,m_j\}$ and $\|A_\iii^{(j)}A_{\lll_i}^{(j)}|_{U_j^{i_j}}\|\geq \tau_1 \|A_\iii^{(j)}A_{\lll_i}^{(j)}\|$ for each $j \in \{2,\ldots,k\}$. Combining this property with \eqref{eq:last} we have
\begin{align*}
  \prod_{j=1}^k \|A_\iii^{(j)}\|^{\beta_j} &\leq \tau^{-\beta}_1\biggl(\prod_{j=1}^k \|(A_{\lll_i}^{(j)})^{-1}\|^{\beta_j}\biggr)\biggl(\prod_{j=1}^k \|A_\iii^{(j)}A_{\lll_i}^{(j)}|_{U_j^{i_j}}\|^{\beta_j} \biggr) \\
  &\leq \tau^{-\beta}_1\biggr(\prod_{j=1}^k \|(A_{\lll_i}^{(j)})^{-1}\|^{\beta_j}\|A_{\lll_i}^{(j)}\|^{\beta_j}\biggr)\biggr(\prod_{j=1}^k \|A_\iii^{(j)}|_{A_{\lll_i}^{(j)}U_j^{i_j}}\|^{\beta_j} \biggr) \\
  &\leq  (\tau_1\tau_2)^{-\beta}\prod_{j=1}^k \|A_\iii^{(j)}|_{A_{\lll_i}^{(j)}U_j^{i_j}}\|^{\beta_j}.
\end{align*}
By \eqref{eq:what-w-does}, the tuple $(U_j^{i_j})_{j=1}^k$ belongs to some  $\mathscr{W}_t$ such that $t \in \{1,\ldots,p\}$. Since each $\mathscr{W}_t$ is invariant under each of the maps $(W_j)_{j=1}^k \mapsto (A_\jjj^{(j)}W_j)_{j=1}^k$ for $\jjj \in \I^*$, we have
 $(A_{\lll_i}^{(j)}U_j^{i_j})_{j=1}^k \in \mathscr{W}_t$ also. We conclude that
 \[
   \prod_{j=1}^k \|A_\iii^{(j)}\|^{\beta_j} \leq (\tau_1\tau_2)^{-\beta} \max_{t \in \{1,\ldots,p\}}\max_{(W_j)_{j=1}^k \in \mathscr{W}_t} \prod_{j=1}^k \|A_\iii^{(j)}|_{W_j}\|^{\beta_j}.
\]
The inequality
\[
\max_{t \in \{1,\ldots,p\}}\max_{(W_j)_{j=1}^k \in \mathscr{W}_t} \prod_{j=1}^k \|A_\iii^{(j)}|_{W_j}\|^{\beta_j} \leq \prod_{j=1}^k \|A_\iii^{(j)}\|^{\beta_j}
\]
is trivial. Since $\iii$ was arbitrary we have proved \eqref{eq:max-equivalence}.

It remains only to prove \eqref{eq:qm2}, for which we use Proposition \ref{pr:qm}. We note that $\I^*$ is a semigroup with respect to the operation $(\iii,\jjj) \mapsto \iii\jjj$, and for each $j \in \{1,\ldots,k\}$ the map $\phi_j \colon \I^* \to \GL(V_j)$ defined by $\phi_j(\iii)=A_\iii^{(j)}$ is an irreducible representation. For each $\mathscr{W}_t$ let $(\hat U_j)_{j=1}^k \in \mathscr{W}_t$ be arbitrary and observe that
\[
  \mathscr{W}_t = \{(\phi_j(\iii) \hat U_j)_{j=1}^k \colon \iii \in \I^*\}
\]
since the action of $\I^*$ by $(\iii, (W_j)_{j=1}^k) \mapsto (\phi_j(\iii)W_j)_{j=1}^k$ is by definition transitive on $\mathscr{W}_t$. By Proposition \ref{pr:qm} there exist for each $t \in \{1,\ldots,p\}$ a finite set $F_t \subset \I^*$ and a real number $\kappa_t>0$ such that for every $(W_j)_{j=1}^k, (W_j')_{j=1}^k \in \mathscr{W}_t$ there exists $\kkk \in F_t$ such that
\begin{equation}\label{eq:that-one}
  \|A_{\iii}^{(j)}A_{\kkk}^{(j)}A_{\jjj}^{(j)} |_{W_j}\| \geq \kappa_t \|A_{\iii}^{(j)}|_{W_j'}\| \|A_{\jjj}^{(j)}|_{W_j}\|
\end{equation}
for all $j \in \{1,\ldots,k\}$. Define $F=\bigcup_{t=1}^p F_t$ and $\kappa = \min_{t \in \{1,\ldots,p\}} \kappa_t$. By \eqref{eq:that-one} it follows easily that for every $(W_j)_{j=1}^k, (W_j')_{j=1}^k \in \mathscr{W}_t$ we have
\[
  \max_{\kkk \in F} \prod_{j=1}^k \|A_\iii^{(j)}A_\kkk^{(j)}A_\jjj^{(j)}|_{W_j}\|^{\beta_j} \geq \kappa^{\beta} \biggl(\prod_{j=1}^k \|A_\iii^{(j)}|_{W_j'}\|^{\beta_j}\biggr)\biggl(\prod_{j=1}^k \|A_\jjj^{(j)}|_{W_j}\|^{\beta_j}\biggr)
\]
and by taking the maximum with respect to $(W_j)_{j=1}^k, (W_j')_{j=1}^k \in \mathscr{W}_t$ we find that
\[
  \max_{\kkk \in F} \Psi^{(t)}(\iii \kkk \jjj ) \geq \kappa^{\beta} \Psi^{(t)}(\iii)\Psi^{(t)}(\jjj)
\]
for every $\iii,\jjj \in \I^*$ and $t \in \{1,\ldots,p\}$. On the other hand if $t$ is fixed then for every $\iii, \jjj \in \I^*$ and every $(W_j)_{j=1}^k \in \mathscr{W}_t$ we clearly have 
\begin{align*}
  \prod_{j=1}^k \|A_\iii^{(j)}A_\jjj^{(j)} |_{W_j}\|^{\beta_j} &\leq \prod_{j=1}^k \|A_\iii^{(j)}|_{A_\jjj^{(j)}W_j} \|^{\beta_j} \|A_\jjj^{(j)} |_{W_j}\|^{\beta_j} \\
  &=\biggl(\prod_{j=1}^k \|A_\iii^{(j)}|_{A_\jjj^{(j)}W_j} \|^{\beta_j} \biggr) \biggl(\prod_{j=1}^k \|A_\jjj^{(j)} |_{W_j}\|^{\beta_j}\biggr) \leq \Psi^{(t)}(\iii)\Psi^{(t)}(\jjj)
\end{align*}
where we have used the fact that $(A_\jjj^{(j)}W_j)_{j=1}^k \in \mathscr{W}_t$ by the definition of $\mathscr{W}_t$. The inequality
\[
  \Psi^{(t)}(\iii\jjj) \leq \Psi^{(t)}(\iii)\Psi^{(t)}(\jjj)
\]
follows straightforwardly by taking the maximum over $(W_j)_{j=1}^k \in \mathscr{W}_t$. We have established \eqref{eq:qm2} and the proof of the theorem is complete.
\end{proof}

Let us next extend Theorem \ref{thm:technical-irreducible} into the completely reducible case. Observe that Theorem \ref{thm:conc} follows immediately from Theorem \ref{thm:technical-completely} which we state shortly. Indeed, by recalling for example \cite[\S 3.4]{KaenmakiMorris2018}, we have
\begin{equation*}
  \fii^s(A) = \|A^{\wedge \lfloor s \rfloor}\|^{\lceil s \rceil - s}\|A^{\wedge \lceil s \rceil}\|^{s - \lfloor s \rfloor}
\end{equation*}
for all $A \in \GL_d(\R)$ and $0 \le s < d$ with the convention that $\|A^{\wedge 0}\|=1$. Therefore, if $\A = (A_i)_{i \in \I} \in \GL_d(\R)^{\I}$ is completely reducible, then, by Proposition \ref{pr:cr}, also $\A^{(\lfloor s \rfloor)} = (A_i^{\wedge \lfloor s \rfloor})_{i \in \I} \in \GL(\wedge^{\lfloor s \rfloor}\R^d)^{\I}$ and $\A^{(\lceil s \rceil)} = (A_i^{\wedge \lceil s \rceil})_{i \in \I} \in \GL(\wedge^{\lceil s \rceil}\R^d)^{\I}$ are completely reducible, and Theorem \ref{thm:technical-completely} shows there exist an integer $p$ such that
\begin{equation*}
  \begin{cases}
    1 \le p \le \dim \wedge^{\lfloor s \rfloor}\R^d = \binom{d}{\lfloor s \rfloor}, &\text{if } s = \lfloor s \rfloor, \\ 
    1 \le p \le \dim \wedge^{\lfloor s \rfloor}\R^d \dim \wedge^{\lceil s \rceil}\R^d = \binom{d}{\lfloor s \rfloor}\binom{d}{\lceil s \rceil}, &\text{if } s > \lfloor s \rfloor,
  \end{cases}
\end{equation*}
and the functions $\Phi_s^{(t)} = \Psi_{\lceil s \rceil - s, s - \lfloor s \rfloor}^{(t)}$, $t \in \{1,\ldots,p\}$, satisfy the claimed properties.

\begin{theorem} \label{thm:technical-completely}
  Let $k \geq 1$, let $\I$ be finite or countably infinite, and for each $j \in \{1,\ldots,k\}$ let $V_j$ be a real inner product space and let $\A^{(j)}=(A^{(j)}_i)_{i\in\I}=(\bigoplus_{t_j=1}^{r_j}B^{(j,t_j)}_i)_{i\in\I}\in\GL(V_j)^\I$ be completely reducible. Then there exist an integer $p$ such that
  \begin{equation*}
    1 \le p \le \Bigl( \min_{j \in \{1,\ldots,k\}} \frac{r_j}{\dim V_j} \Bigr) \prod_{j=1}^k \dim V_j \le \prod_{j=1}^k \dim V_j
  \end{equation*}
  with functions $\Psi^{(1)}_{(\cdot)},\ldots,\Psi^{(p)}_{(\cdot)} \colon [0,\infty)^k \times \I^* \to (0,\infty)$, constants $\kappa,\tau>0$, and a finite set $F \subset \I^*$ such that writing $\beta = \sum_{j=1}^k \beta_j$ the following three properties hold: 
  \begin{enumerate}[(i)]
    \item\label{it:tech-comp1} We have
    \begin{equation*}
      \tau^{\beta} \prod_{j=1}^k \|A_\iii^{(j)}\|^{\beta_j} \le \max_{t \in \{1,\ldots,p\}} \Psi^{(t)}_{\beta_1,\ldots,\beta_k}(\iii) \le \prod_{j=1}^k \|A_\iii^{(j)}\|^{\beta_j}
    \end{equation*}
    for all $\iii \in \I^*$.
    \item\label{it:tech-comp2} For every $t \in \{1,\ldots,p\}$ we have
    \begin{equation*}
      \Psi^{(t)}_{\beta_1,\ldots,\beta_k}(\iii\jjj) \le \Psi^{(t)}_{\beta_1,\ldots,\beta_k}(\iii)\Psi^{(t)}_{\beta_1,\ldots,\beta_k}(\jjj) \le \kappa^{-\beta} \max_{\kkk \in F} \Psi^{(t)}_{\beta_1,\ldots,\beta_k}(\iii\kkk\jjj)
    \end{equation*}
    for all $\iii,\jjj \in \I^*$.
    \item\label{it:tech-comp3} For every $t \in \{1,\ldots,p\}$ and $\iii \in \I^*$ the function $(\beta_1,\ldots,\beta_k) \mapsto \Psi^{(t)}_{\beta_1,\ldots,\beta_k}(\iii)$ defined on $[0,\infty)^k$ is continuous.
  \end{enumerate}
\end{theorem}

\begin{proof}
  Let
  \begin{equation*}
    \mathfrak{R}=\{(t_1,\ldots,t_k) \in \N^k \colon t_j \in \{1,\ldots,r_j\}\text{ for all }j \in \{1,\ldots,k\}\}
  \end{equation*}
  and observe that $(B_i^{(j,t_j)})_{i \in \I}$ is irreducible for every $j \in \{1,\ldots,k\}$. Note also that there exists a splitting $V_j = \bigoplus_{t_j=1}^{r_j} V_{j,t_j}$ such that $A_iV_{j,t_j} = B_i^{(j,t_j)}V_{j,t_j} = V_{j,t_j}$. For each $\mathfrak{r} = (t_1,\ldots,t_k) \in \mathfrak{R}$ let $\Psi^{(\mathfrak{r},t)}_{\beta_1,\ldots,\beta_k}$ be the functions associated to irreducible tuples $(B_i^{(j,t_j)})_{i \in \I}$, $j \in \{1,\ldots,k\}$, given by Theorem \ref{thm:technical-irreducible}. By definition, the functions $(\beta_1,\ldots,\beta_k) \mapsto \Psi^{(\mathfrak{r},t)}_{\beta_1,\ldots,\beta_k}(\iii)$ are clearly continuous proving \eqref{it:tech-comp3} and therefore, for the rest of the proof, we may consider $\beta_1,\ldots,\beta_k \ge 0$ being fixed and omit it in notation of $\Psi^{(\mathfrak{r},t)}_{\beta_1,\ldots,\beta_k}$.

  By permuting the indices $j \in \{1,\ldots,k\}$, we assume without loss of generality that $\min_{j \in \{1,\ldots,k\}} r_j/\dim V_j = r_k/\dim V_k$. For a fixed $\mathfrak{r} = (t_1,\ldots,t_k) \in \mathfrak{R}$, Theorem \ref{thm:technical-irreducible} shows that there are at most $p_{\mathfrak{r}} \le \prod_{j=1}^{k-1} \dim V_{j,t_j}$ many functions $\Psi^{(\mathfrak{r},t)}$. Therefore, the total number $p$ of functions $\Psi^{(\mathfrak{r},t)}$ is bounded above by
  \begin{align*}
    p &\le \sum_{\mathfrak{r} \in \mathfrak{R}} \prod_{j=1}^{k-1} \dim V_{j,t_j} = r_k\prod_{j=1}^{k-1} \sum_{t_j=1}^{r_j} \dim V_{j,t_j} \\ 
    &= r_k\prod_{j=1}^{k-1} \dim V_j = \Bigl( \min_{j \in \{1,\ldots,k\}} \frac{r_j}{\dim V_j} \Bigr) \prod_{j=1}^k \dim V_j
  \end{align*}
  as claimed. 

  For each $\mathfrak{r} = (t_1,\ldots,t_k) \in \mathfrak{R}$, let the constants $\kappa_{\mathfrak{r}}, \tau_{\mathfrak{r}} > 0$ and the finite set $F_{\mathfrak{r}} \subset \I^*$ be as in Theorem \ref{thm:technical-irreducible}. Define $\kappa = \min_{\mathfrak{r} \in \mathfrak{R}} \kappa_{\mathfrak{r}}$, $\tau = \min_{\mathfrak{r} \in \mathfrak{R}} \tau_{\mathfrak{r}}$, and $F = \bigcup_{\mathfrak{r} \in \mathfrak{R}} F_{\mathfrak{r}}$. Observe that \eqref{it:tech-comp2} follows immediately from Theorem \ref{thm:technical-irreducible}. Since
  \begin{equation*}
    \tau^{\beta} \prod_{j=1}^k \|B_\iii^{(j,t_j)}\|^{\beta_j} \leq \max_{t \in \{1,\ldots,p_{\mathfrak{r}}\}} \Psi^{(\mathfrak{r},t)}(\iii) \leq \prod_{j=1}^k \|B_\iii^{(j,t_j)}\|^{\beta_j}
  \end{equation*}
  by Theorem \ref{thm:technical-irreducible} and
  \begin{equation*}
    \prod_{j=1}^k \|A_{\iii}^{(j)}\|^{\beta_j} = \max_{\mathfrak{r}\in\mathfrak{R}} \prod_{j=1}^k \|B_{\iii}^{(j,t_j)}\|^{\beta_j},
  \end{equation*}
  we have shown \eqref{it:tech-comp1} and finished the proof.
\end{proof}


\begin{thebibliography}{10}

\bibitem{Baranski2007}
K.~Bara{\'n}ski.
\newblock {Hausdorff dimension of the limit sets of some planar geometric
  constructions}.
\newblock {\em Adv. Math.}, 210(1):215--245, 2007.

\bibitem{BaranyHochmanRapaport2019}
B.~B\'{a}r\'{a}ny, M.~Hochman, and A.~Rapaport.
\newblock Hausdorff dimension of planar self-affine sets and measures.
\newblock {\em Invent. Math.}, 216(3):601--659, 2019.

\bibitem{BaranyJordanKaenmakiRams2021}
B.~B\'{a}r\'{a}ny, T.~Jordan, A.~K\"{a}enm\"{a}ki, and M.~Rams.
\newblock Birkhoff and {L}yapunov spectra on planar self-affine sets.
\newblock {\em Int. Math. Res. Not. IMRN}, (10):7966--8005, 2021.

\bibitem{BaranyKaenmaki2017}
B.~B{\'a}r{\'a}ny and A.~K{\"a}enm{\"a}ki.
\newblock {Ledrappier-{Y}oung formula and exact dimensionality of self-affine
  measures}.
\newblock {\em Adv. Math.}, 318:88--129, 2017.

\bibitem{BaranyKaenmakiKoivusalo2018}
B.~B\'{a}r\'{a}ny, A.~K\"{a}enm\"{a}ki, and H.~Koivusalo.
\newblock Dimension of self-affine sets for fixed translation vectors.
\newblock {\em J. Lond. Math. Soc. (2)}, 98(1):223--252, 2018.

\bibitem{BaranyKaenmakiMorris2020}
B.~B\'{a}r\'{a}ny, A.~K\"{a}enm\"{a}ki, and I.~D. Morris.
\newblock Domination, almost additivity, and thermodynamic formalism for planar
  matrix cocycles.
\newblock {\em Israel J. Math.}, 239(1):173--214, 2020.

\bibitem{BaranyKaenmakiRossi2021}
B.~B\'{a}r\'{a}ny, A.~K\"{a}enm\"{a}ki, and E.~Rossi.
\newblock Assouad dimension of planar self-affine sets.
\newblock {\em Trans. Amer. Math. Soc.}, 374(2):1297--1326, 2021.

\bibitem{BaranyKaenmakiYu2021}
B.~B\'{a}r\'{a}ny, A.~K{\"a}enm{\"a}ki, and H.~Yu.
\newblock Finer geometry of planar self-affine sets.
\newblock Preprint, available at arXiv:2107.00983, 2021.

\bibitem{BochiMorris2018}
J.~Bochi and I.~D. Morris.
\newblock Equilibrium states of generalised singular value potentials and
  applications to affine iterated function systems.
\newblock {\em Geom. Funct. Anal.}, 28(4):995--1028, 2018.

\bibitem{Bowen}
R.~Bowen.
\newblock {\em {Equilibrium states and the ergodic theory of {A}nosov
  diffeomorphisms}}, volume 470 of {\em {Lecture Notes in Mathematics}}.
\newblock Springer-Verlag, Berlin, revised edition, 2008.
\newblock With a preface by David Ruelle, Edited by Jean-Ren{\'e} Chazottes.

\bibitem{DasSimmons2017}
T.~Das and D.~Simmons.
\newblock The {H}ausdorff and dynamical dimensions of self-affine sponges: a
  dimension gap result.
\newblock {\em Invent. Math.}, 210(1):85--134, 2017.

\bibitem{Falconer1988}
K.~J. Falconer.
\newblock {The {H}ausdorff dimension of self-affine fractals}.
\newblock {\em Math. Proc. Cambridge Philos. Soc.}, 103(2):339--350, 1988.

\bibitem{Feng2023}
D.-J. Feng.
\newblock Dimension of invariant measures for affine iterated function systems.
\newblock {\em Duke Math. J.}, 172(4):701--774, 2023.

\bibitem{FengFeng2023preprint}
D.-J. Feng and Z.~Feng.
\newblock Typical self-affine sets with non-empty interior.
\newblock {\em Asian J. Math.}, 27(5):621--638, 2023.

\bibitem{FengKaenmaki2011}
D.-J. Feng and A.~K{\"a}enm{\"a}ki.
\newblock {Equilibrium states of the pressure function for products of
  matrices}.
\newblock {\em Discrete Contin. Dyn. Syst.}, 30(3):699--708, 2011.

\bibitem{FengShmerkin2014}
D.-J. Feng and P.~Shmerkin.
\newblock {Non-conformal repellers and the continuity of pressure for matrix
  cocycles}.
\newblock {\em Geom. Funct. Anal.}, 24(4):1101--1128, 2014.

\bibitem{Fraser2012}
J.~M. Fraser.
\newblock {On the packing dimension of box-like self-affine sets in the plane}.
\newblock {\em Nonlinearity}, 25(7):2075--2092, 2012.

\bibitem{Hochman2014}
M.~Hochman.
\newblock {On self-similar sets with overlaps and inverse theorems for
  entropy}.
\newblock {\em Ann. of Math. (2)}, 180(2):773--822, 2014.

\bibitem{HochmanRapaport2022}
M.~Hochman and A.~Rapaport.
\newblock Hausdorff dimension of planar self-affine sets and measures with
  overlaps.
\newblock {\em J. Eur. Math. Soc. (JEMS)}, 24(7):2361--2441, 2022.

\bibitem{HornJohnson1991}
R.~A. Horn and C.~R. Johnson.
\newblock {\em Topics in matrix analysis}.
\newblock Cambridge University Press, Cambridge, 1991.

\bibitem{GodofredoToddVelozo2020}
G.~Iommi, M.~Todd, and A.~Velozo.
\newblock Upper semi-continuity of entropy in non-compact settings.
\newblock {\em Math. Res. Lett.}, 27(4):1055--1077, 2020.

\bibitem{JordanPollicottSimon2007}
T.~Jordan, M.~Pollicott, and K.~Simon.
\newblock {Hausdorff dimension for randomly perturbed self affine attractors}.
\newblock {\em Comm. Math. Phys.}, 270(2):519--544, 2007.

\bibitem{Kaenmaki2004}
A.~K{\"a}enm{\"a}ki.
\newblock {On natural invariant measures on generalised iterated function
  systems}.
\newblock {\em Ann. Acad. Sci. Fenn. Math.}, 29(2):419--458, 2004.

\bibitem{KaenmakiMorris2018}
A.~K\"{a}enm\"{a}ki and I.~D. Morris.
\newblock Structure of equilibrium states on self-affine sets and strict
  monotonicity of affinity dimension.
\newblock {\em Proc. Lond. Math. Soc. (3)}, 116(4):929--956, 2018.

\bibitem{KaenmakiReeve2014}
A.~K{\"a}enm{\"a}ki and H.~W.~J. Reeve.
\newblock {Multifractal analysis of {B}irkhoff averages for typical infinitely
  generated self-affine sets}.
\newblock {\em J. Fractal Geom.}, 1(1):83--152, 2014.

\bibitem{KaenmakiVilppolainen2010}
A.~K{\"a}enm{\"a}ki and M.~Vilppolainen.
\newblock {Dimension and measures on sub-self-affine sets}.
\newblock {\em Monatsh. Math.}, 161(3):271--293, 2010.

\bibitem{MauldinUrbanski1996}
R.~D. Mauldin and M.~Urba\'{n}ski.
\newblock Dimensions and measures in infinite iterated function systems.
\newblock {\em Proc. London Math. Soc. (3)}, 73(1):105--154, 1996.

\bibitem{MauldinUrbanski2003}
R.~D. Mauldin and M.~Urba{\'n}ski.
\newblock {\em Graph directed {Markov} systems. {Geometry} and dynamics of
  limit sets}, volume 148 of {\em Camb. Tracts Math.}
\newblock Cambridge: Cambridge University Press, 2003.

\bibitem{MauldinWilliams1986}
R.~D. Mauldin and S.~C. Williams.
\newblock Random recursive constructions: asymptotic geometric and topological
  properties.
\newblock {\em Trans. Amer. Math. Soc.}, 295:325--346, 1986.

\bibitem{Milne2017}
J.~S. Milne.
\newblock {\em Algebraic groups: the theory of group schemes of finite type
  over a field}, volume 170 of {\em Cambridge Studies in Advanced Mathematics}.
\newblock Cambridge University Press, Cambridge, 2017.

\bibitem{Morris2016}
I.~D. Morris.
\newblock {An inequality for the matrix pressure function and applications}.
\newblock {\em Adv. Math.}, 302:280--308, 2016.

\bibitem{Morris2018}
I.~D. Morris.
\newblock Ergodic properties of matrix equilibrium states.
\newblock {\em Ergodic Theory Dynam. Systems}, 38(6):2295--2320, 2018.

\bibitem{Morris2021}
I.~D. Morris.
\newblock Totally ergodic generalised matrix equilibrium states have the
  {B}ernoulli property.
\newblock {\em Comm. Math. Phys.}, 387(2):995--1050, 2021.

\bibitem{MorrisSert2019}
I.~D. Morris and {\c{C}}.~Sert.
\newblock A converse statement to {H}utchinson's theorem and a dimension gap
  for self-affine measures.
\newblock {\em J. Eur. Math. Soc. (JEMS)}, 25(11):4315--4367, 2023.

\bibitem{MorrisSert2023preprint}
I.~D. Morris and C.~Sert.
\newblock A variational principle relating self-affine measures to self-affine
  sets.
\newblock Preprint, available at arXiv:2303.03437, 2023.

\bibitem{MorrisShmerkin2019}
I.~D. Morris and P.~Shmerkin.
\newblock On equality of {H}ausdorff and affinity dimensions, via self-affine
  measures on positive subsystems.
\newblock {\em Trans. Amer. Math. Soc.}, 371(3):1547--1582, 2019.

\bibitem{Ornstein1972}
D.~S. Ornstein.
\newblock On the root problem in ergodic theory.
\newblock In {\em Proceedings of the {S}ixth {B}erkeley {S}ymposium on
  {M}athematical {S}tatistics and {P}robability ({U}niv. {C}alifornia,
  {B}erkeley, {C}alif., 1970/1971), {V}ol. {II}: {P}robability theory}, pages
  347--356. Univ. California Press, Berkeley, Calif., 1972.

\bibitem{Rapaport2022preprint}
A.~Rapaport.
\newblock On self-affine measures associated to strongly irreducible and
  proximal systems.
\newblock {\em Adv. Math.}, 449, paper no. 109734, 2023.

\bibitem{Rapaport2023preprint}
A.~Rapaport.
\newblock Dimension of diagonal self-affine sets and measures via non-conformal
  partitions.
\newblock Preprint, available at arXiv:2309.03985, 2023.

\bibitem{Rossi2014}
E.~Rossi.
\newblock Local dimensions of measures on infinitely generated self-affine
  sets.
\newblock {\em J. Math. Anal. Appl.}, 413(2):1030--1039, 2014.

\bibitem{Urbanski1998}
M.~Urba\'{n}ski.
\newblock Hausdorff measures versus equilibrium states of conformal infinite
  iterated function systems.
\newblock volume~37, pages 153--205. 1998.
\newblock International Conference on Dimension and Dynamics (Miskolc, 1998).

\bibitem{Walters1982}
P.~Walters.
\newblock {\em {An introduction to ergodic theory}}, volume~79 of {\em
  {Graduate Texts in Mathematics}}.
\newblock Springer-Verlag, New York, 1982.

\end{thebibliography}
\end{document}